\documentclass[12pt]{amsart}
\usepackage{amsfonts,amsmath,amssymb,amsthm,ifthen,latexsym}

\newcommand{\Tr}{{\rm Tr}}
\newcommand{\RTr}{{\rm RTr}}
\newcommand{\W}{{\rm W}}
\newcommand{\WF}{{\rm WF}}
\newcommand{\ie}{{\rm i.e.,} }
\newcommand{\cf}{{\rm cf.~}}
\newcommand{\nat}{\omega}
\newcommand{\n}{\ensuremath{n \in \nat}}
\newcommand{\iin}{\ensuremath{i \in \nat}}
\newcommand{\ep}{\ensuremath{\varepsilon}}
\newcommand{\eq}{\ensuremath{\Longleftrightarrow}}
\newcommand{\del}{\ensuremath{\Delta^1_1}}
\newcommand{\sig}{\ensuremath{\Sigma^1_1}}
\newcommand{\pii}{\ensuremath{\Pi^1_1}}
\newcommand\tboldsymbol[1]{%
\protect\raisebox{0pt}[0pt][0pt]{%
$\underset{\widetilde{}}{\boldsymbol{#1}}$}\mbox{\hskip 1pt}}
\newcommand{\boldpii}{\ensuremath{\tboldsymbol{\Pi}^1_1}}
\newcommand{\om}{\ensuremath{\omega}}
\newcommand{\Q}{\ensuremath{\mathbb Q}}
\newcommand{\R}{\ensuremath{\mathbb R}}
\newcommand{\ca}[1]{\ensuremath{\mathcal{#1}}}
\newcommand{\set}[2]{\ensuremath{\{#1 \hspace{0.3mm} \mid \hspace{0.3mm} #2\}}}
\newcommand{\barr}[1]{\ensuremath{\overline{#1}}}
\newcommand{\inj}{\rightarrowtail}
\newcommand{\bij}{\mbox{\,$\rightarrowtail\kern -.8em \rightarrow$\,}}
\newcommand\surj{\twoheadrightarrow}
\newcommand{\lh}{\mbox{\rm lh}}
\newcommand{\empt}{\ensuremath{\emptyset}}
\newcommand{\pairint}[2]{\ensuremath{\lfloor#1,#2\rfloor}}
\newcommand{\prodt}[2]{\ensuremath{#1\circledast#2}}
\newcommand{\tu}[1]{\textup{#1}}
\newcommand{\att}[1]{\ensuremath{\mathtt{a}#1}}
\newcommand{\iniseg}{\ensuremath{{\rm seg}}}
\newcommand{\rfn}[1]{\ensuremath{\{#1\}}}
\newcommand\dense{{\mathbf{r}}}
\newcommand{\ggf}{{\varphi}}
\newcommand{\lo}{{\rm L}}
\newcommand{\Seq}{\ensuremath{\mathtt{Seq}}}
\newcommand{\omseq}{\ensuremath{\om^{<\om}}}
\newcommand{\spat}[1]{\ensuremath{\ca{N}^{#1}}}
\newcommand{\cn}[2]{\ensuremath{#1 \ \hat{} \ #2}}
\newcommand{\dec}[1]{\ensuremath{u[{#1}}]}
\newcommand{\ckr}[1]{\ensuremath{\omega_1^{#1}}}
\newcommand{\ck}{\ensuremath{\omega_1^{{\rm CK}}}}
\newcommand{\pdel}{{\mathbf{d}}}
\newcommand{\pker}[1]{\ensuremath{#1_{\rm p}}}
\newcommand{\scat}[1]{\ensuremath{#1_{\rm sc}}}
\newcommand{\hless}{\ensuremath{<_{\rm h}}}
\newcommand{\hleq}{\ensuremath{\leq_{\rm h}}}
\newcommand{\heq}{\ensuremath{=_{\rm h}}}
\newcommand{\dleq}{\ensuremath{\preceq_{\del}}}
\newcommand{\dequal}{\ensuremath{\simeq_{\del}}}
\newcommand{\dless}{\ensuremath{\prec_{\del}}}
\newcommand{\dgreat}{\ensuremath{\succ_{\del}}}

\newcommand{\tleq}{\ensuremath{\leq_{\rm T}}}

\newcommand{\kbleq}{\ensuremath{\leq_{\rm KB}}}
\newcommand{\kbgreat}{\ensuremath{>_{\rm KB}}}
\newcommand{\kbless}{\ensuremath{<_{\rm KB}}}
\newcommand{\Field}{\textrm{\normalfont Field}}
\newcommand{\Dom}{\textrm{\normalfont Domain}}
\newcommand{\suc}{\textrm{\normalfont Succ}}
\newcommand{\Gr}{{\rm Gr}}
\newcommand{\lin}{\ensuremath{\trianglelefteq}}
\newcommand{\slin}{\ensuremath{\triangleleft}}

\newtheorem{theorem}{Theorem}[section]
\newtheorem{lemma}[theorem]{Lemma}
\newtheorem{definition}[theorem]{Definition}

\newtheorem{corollary}[theorem]{Corollary}
\newtheorem{remark}[theorem]{Remark}
\newtheorem{question}{Question}

\title[Classes of Polish spaces]{Classes of Polish spaces under effective Borel isomorphism}

\author[V. Gregoriades]{Vassilios Gregoriades}
\address{Technische Universit\"{a}t Darmstadt,
Fachbereich Mathematik,
Arbeitsgruppe Logik,
Schlo{\ss}gartenstra{\ss}e 7,
64289 Darmstadt
Germany}
\email{gregoriades[at]mathematik[dot]tu-darmstadt[dot]de}

\date{}



\begin{document}

\maketitle

\thispagestyle{empty}

\begin{abstract}
We study the equivalence classes under \del \ isomorphism, otherwise effective-Borel isomorphism, between complete separable metric spaces which admit a recursive presentation and we show the existence of strictly increasing and strictly decreasing sequences as well as of infinite antichains under the natural notion of \del-reduction, as opposed to the non-effective case, where only two such classes exist, the one of the Baire space and the one of the naturals. A key tool for our study is a mapping $T \mapsto \spat{T}$ from the space of all trees on the naturals to the class of Polish spaces, for which every recursively presented space is \del-isomorphic to some \spat{T} for a recursive $T$, so that the preceding spaces are representatives for the classes of \del-isomorphism. We isolate two large categories of spaces of the type \spat{T}, the Kleene spaces and the Spector-Gandy spaces and we study them extensively. Moreover we give results about hyperdegrees in the latter spaces and characterizations of the Baire space up to \del-isomorphism.
\end{abstract}

\vspace*{55mm}

\emph{This is a preprint of the work published in the \textbf{Memoirs of the American Mathematical Society}, 240 (2016), no. 1135, vii+87}.\newline 

\newpage

\section*{Acknowledgments.} 

This article contains work carried between 2009 and 2013. It includes parts of my Ph.D. Thesis, which was submitted to and approved by the Mathematics Department of the University of Athens Greece in 2009 and supervised by {\sc Y. N. Moschovakis} and {\sc A. Tsarpalias}. On a rough estimate these parts include most of the material in the first two Chapters, the first Section in Kleene spaces as well as results about the Cantor-Bendixson decomposition of Spector-Gandy spaces$-$although the latter spaces are not identified with this name in my Thesis.

The extended abstract of a very early version of this article has been published in the proceedings of the 8th Panhellenic Logic Symposium in 2011, Ioannina, Greece. \newline I would like to express my gratitude to my supervisors and especially to {\sc Y. N. Moschovakis} for his invaluable guidance through the whole process of preparation of this article. \newline The realization of this project would not be possible without {\sc U. Kohlenbach}, who, aside from his valuable advice, has offered to the author a long term postdoctoral position in the Logic Group at TU Darmstadt. I would like to take the opportunity to express my sincere thanks to him. \newline The list of people who have helped me in this task contains one prominent entry, my wife Stella, who knows first hand the difficulties of being the spouse of a researcher in mathematics. I am deeply grateful to her for the enduring support and encouragement.

\pagenumbering{Roman}
\setcounter{tocdepth}{2}\newpage
\tableofcontents
\newpage
\setcounter{page}{1}
\pagenumbering{arabic}

\section{\sc Introduction}
\label{section introduction}

The topic of our research is the class of complete separable metric spaces which admit a recursive presentation and their classification under \del \ (otherwise effective-Borel) isomorphism. Following \cite{yiannis_dst} we refer to the latter spaces simply as recursively presented metric spaces. In the sequel we will explain how this can be carried out in all Polish spaces. We begin this introductory section with a discussion about the motivation and a short description of this article. Then we proceed to the necessary definitions and basic theorems.

\subsection{A few words about the setting} This article originates from the following simple remark about the development of effective descriptive set theory as it is presented in the fundamental textbook \cite{yiannis_dst}, which is perhaps easy to overlook: one assumes that every recursively presented metric space which is not discrete has no isolated points. This assumption may seem unnecessary as all basic notions can be given in a perfect-free context. However the preceding assumption is required for the development of some parts of the theory. For example it is not even clear that the inequality $$\sig \upharpoonright \ca{X} \setminus \pii \upharpoonright \ca{X} \neq \emptyset $$ holds for every recursively presented metric space \ca{X}, and although we will prove this correct, it is still an open problem whether $$\sig \upharpoonright \ca{X} \setminus \tboldsymbol{\Pi}^1_1 \upharpoonright \ca{X} \neq \emptyset$$ holds, if \ca{X} is uncountable non-perfect.

At first sight the preceding assumption seems to have no impact to the applications of effective theory to classical theory, at least from the level of Borel sets and above, since every uncountable Polish space is Borel isomorphic to a perfect Polish space, namely the Baire space. One has to be careful though, for it may be the case that uncountably many spaces are involved, for example the sections of some closed set. In this case one would need to collect the preceding Borel isomorphisms in a uniform way, which may not be so easy to achieve.

Effective theory has also found its way to ``mixed-type" results, where classical and effective notions coexist in the statement, see for example \cite{louveau_a_separation_theorem_for_sigma_sets} and \cite{debs_effective_properties_in_compact_sets} for some very interesting results. In this type of results one should also take care not to apply any theorems of effective theory which use the assumption about perfectness without the space being so. This is often self-evident (as it is for example in \cite{louveau_a_separation_theorem_for_sigma_sets} and \cite{debs_effective_properties_in_compact_sets}), but nevertheless it would be good if one makes sure that this issue will not raise any concerns in future research on the area.

For these reasons it seemed useful to identify the main tools of effective theory and then trace back their proofs to see if the assumption about perfectness is indeed required. As it is probably expected most of them do hold in the general setting and below we give a list of those. There is however a handful of results which do not go through, and they all point to a single problem: \emph{every perfect recursively presented metric space is \del-isomorphic to the Baire space}, but the proof does not go through in the non-perfect setting. This brings us to the motivation of this article.\bigskip

\emph{Does there exist an uncountable recursively presented metric space which is not \del-isomorphic to the Baire space? And if yes what can be said about the structure of such a space and how do these spaces relate to each other?}\bigskip

We will show that there are infinitely many spaces which are incomparable with respect to \del-isomorphism, and that the natural notion
of \del-reduction carries a rich structure. The study of this structure is one of our aims.

We will see that there is a canonical way of choosing representatives from the preceding equivalence classes. More specifically we will define an operation
\[
T \mapsto \spat{T}
\]
which maps a tree $T$ on the naturals to a complete and separable metric space \spat{T} such that if $T$ is recursive then \spat{T} admits a recursive presentation. Every recursively presented metric space is \del-isomorphic to one of the spaces \spat{T} for some recursive tree $T$. Thus the spaces \spat{T} play the role of the Baire space in the non-perfect context -this is the justification for the notation.

For various choices of $T$ the space \spat{T} is not \del-isomorphic to the Baire space. We will focus on two large categories of spaces of this type: the \emph{Kleene spaces} and the \emph{Spector-Gandy spaces}. We will see that Kleene spaces have neither maximal nor minimal element and that they form antichains under the natural notion of \del-reduction. With some considerable elaboration of the arguments used in Kleene spaces, we are able to prove some similar (however weaker) results about Spector-Gandy spaces.

Our research will inevitably lead us to some questions about hyperdegrees. We will give some results on the topic in connection with Kleene and Spector-Gandy spaces. Moreover our results on Kleene spaces can be carried to recursive pseudo well-orderings in a natural way.

We will also deal with the problem of characterizing the Baire space up to \del-isomorphism and more specifically we will provide three such characterizations in terms of the intrinsic properties of trees, the Cantor-Bendixson decomposition and measures.

Occasionally we will review some well-known results of effective theory in modern day language and provide some new short proofs, (as it is for example with the Gandy Basis Theorem \ref{theorem Gandy basis} and Kreisel's results on the Cantor-Bendixson decomposition Theorem \ref{theorem kreisel scattered part}) or even extend them (as it is with the strong form of the Spector-Gandy Theorem for Polish spaces Theorem \ref{theorem strong form spector-gandy general}).

We conclude this article with a section on questions together with some relevant results.

It is worth pointing out that the spaces \spat{T} can be viewed as $\Pi^0_1$ subsets of the Baire space which contain a dense recursive sequence. There are various notions of comparison between $\Pi^0_1$ subsets of the Baire space, for example \emph{Muchnik} and \emph{Medvedev} reducibility\footnote{If $A$ and $B$ are $\Pi^0_1$ subsets of $\om^\om$ we say that $A$ is \emph{Muchnic-reducible} to $B$ if for all $\beta \in B$ there exists $\alpha \in A$ such that $\alpha \tleq \beta$, and $A$ is \emph{Medvedev-reducible} to $B$ if there exists a  recursive function $f: \ca{N} \to \ca{N}$ with $f[A] \subseteq B$. Sometimes this definition is given for subsets of $2^ \om$.}, where some very interesting results have been obtained (see for example \cite{binns_stephen_a_splitting_theorem_Medvedev_Muchnik_lattices}, \cite{binns_simpson_embeddings_into_the_medvedev_and_muchnik_lattices} and \cite{simpson_mass_problems_and_randomness}). However their definition does not go above the level of recursive functions. Moreover in the case of $\Pi^ 0_1$ subsets of $2^ \om$ the questions about \del-recursive functions trivialize because of compactness arguments. (This is the effective version of a result of Kunugui, \cf \cite{kunugui_sur_un_probleme_de_szpilrajn} and \cite{yiannis_dst} 4F.15.)

Another type of comparison, which is closer to our standards, is the \del-version of Muchnik reducibility for subsets of \ca{N}, \ie we replace $\alpha \tleq \beta$ with $\alpha \in \del(\beta)$ in the definition of Muchnik reducibility. Regarding the latter there is a relatively recent incomparability result in \cite{fokina_friedman_sy_toernquist_the_effective_theory_of_Borel_equivalence_relations}, which has been very inspiring for this article. We emphasize though that our approach is somewhat different, for we study a special category of $\Pi^0_1$ subsets of the Baire space (the ones which admit a recursive presentation) and as a result of this we insist on \del-injectivity, for otherwise the problem trivializes by taking constant functions. Despite this, some of our results on Kleene spaces are also results about this \del-version of Muchnik reducibility, see for example Corollary \ref{corollary del of gamma is dense in a Kleene tree} and Lemma \ref{lemma for every kleene tree there exists an incomparable kleene tree}.

\subsection{The basic notions} We give a very brief summary of the notions of effective descriptive set theory that we need. For a proper exposition of the subject the reader should refer to Chapter 3 of \cite{yiannis_dst}. We point out that almost all definitions and results in the remaining of this article that are not numbered are well-known. Numbered results which are also well-known will be indicated as such.

We denote by \omseq \ the set of all finite sequence of naturals. We fix once and for all a recursive injection
\[
\langle \cdot \rangle : \omseq \inj \om
\]
and we let $\Seq$ be the image of $\langle \cdot \rangle$. We think of the preceding function as a recursive encoding of the set of all finite sequences of naturals by a natural number. We may assume that the number $1$ is the code of the empty sequence. We denote the latter by \empt. If $u = (u_0,\dots,u_{n-1}) \in \omseq$ we denote by $\lh(u)$ the preceding natural $n$ -when $n=0$ we mean the empty sequence. As usual the notation $u \sqsubseteq v$ means that $v$ is an extension of $u$ and $u \perp v$ that $u$ and $v$ are incompatible, \ie it is not the case $u \sqsubseteq v$ or $v \sqsubseteq u$. We denote by $\cn{u}{v}$ the \emph{concatenation} of $u$ and $v$.

It is useful to fix a notation for the pre-images of elements of \Seq \ under $\langle \cdot \rangle$. We denote by $\dec{s}$ the unique $u = (u_0,\dots,u_{n-1}) \in \omseq$ such that $\langle u_0,\dots,u_{n-1} \rangle = s$ for all $s \in \Seq$.

The \emph{Baire space} is the set $\om^\om$ of all infinite sequences of naturals with the product topology. We denote the Baire space by \ca{N}. The members of \ca{N} will be denoted by lowercase greek letters such as $\alpha$, $\beta$ and $\gamma$. The usual distance function $p_\ca{N}$ on $\ca{N}$ is defined by
\[
p_\ca{N}(\alpha,\beta) = ((\textrm{least} \ k \ \alpha(k) \neq \beta(k)) + 1)^{-1}
\]
for $\alpha \neq \beta$. The \emph{Cantor space} \ca{C} is $2^\om$ with the product topology.

For all $i \in \om$ we define the function
\[
( \cdot )_i: \ca{N} \to \ca{N}: (\alpha)_i(n) = \alpha(\langle i,n \rangle).
\]
It is clear that for all sequences $(\alpha_i)_{\iin}$ in the Baire space there exists some $\alpha \in \ca{N}$ such that $(\alpha)_i = \alpha_i$ for all \iin.

Also we define
\begin{align*}
\hspace*{-10mm}\alpha \upharpoonright n =& (\alpha(0),\dots,\alpha(n-1))\\
\hspace*{-10mm}\textrm{and} \hspace*{5mm} \barr{\alpha}(n)         =& \langle \alpha(0),\dots,\alpha(n-1) \rangle.
\end{align*}
for all $\alpha \in \ca{N}$ and all \n. Unless stated otherwise by $\alpha^\ast$ we mean the function $(n \mapsto \alpha(n+1))$.

A \emph{tree} on \om \ is a set $T \subseteq \omseq$ which is closed under initial segments. The \emph{body} $[T]$ of a tree $T$ is the set of all infinite sequences $\alpha$ such that $\alpha \upharpoonright n \in T$ for all $n$. A tree $T$ is \emph{well-founded} if $[T] = \emptyset$ and \emph{ill-founded} if $[T] \neq \emptyset$. For every finite sequence $u$ and every tree $T$ we define the \emph{subtree}
\[
T_u = \set{v \in T}{v \ \textrm{is compatible with} \ u}.
\]
The \emph{Lusin-Sierpinski} or \emph{Kleene-Brouwer} ordering\footnote{The reader can refer to \cite{yiannis_dst} Chapter 4 Footnote (9) for an interesting discussion regarding the ambiguity of the origins of this ordering.} \kbleq \ (\cf \cite{lusin_sierpinski_sur_un_enseble_non_measurable_B}, \cite{brouwer_uniformly_continuous}, \cite{kleene_on_the_forms_of_the_predicates_in_the_theory_of_constructive_ordinals_II}) is defined on the set of all finite sequences of naturals by
\begin{align*}
&u \kbleq v \iff v \sqsubseteq u \ \vee \ [u \perp v \ \& \ u \leq_{\rm lex} v],
\end{align*}
where $\leq_{\rm lex}$ stands for the usual lexicographical ordering. It is not hard to verify that a tree $T$ is well-founded exactly when the linearly ordered space $(T,\kbleq)$ is well-ordered.

We will often identify relations with the sets that they define and write $P(x)$ instead of $x \in P$. Given Polish spaces $\ca{X}$, $\ca{Y}$ and $P \subseteq \ca{X} \times \ca{Y}$ we define
\begin{align*}
\hspace*{-10mm}\exists^\ca{Y} P =& \set{x \in \ca{X}}{\exists y \in \ca{Y}P(x,y)}\\
\hspace*{-10mm}\textrm{and}\hspace*{10mm}P_x =& \set{y \in \ca{Y}}{P(x,y)}
\end{align*}
for all $x \in \ca{X}$, \ie $\exists^\ca{Y}P$ is the projection of $P$ along \ca{Y} and $P_x$ is the $x$-section of $P$. By the term \emph{pointclass} we mean any collection of sets in Polish spaces, for example the collection of all open sets.

\subsubsection*{Parametrization and universal sets} Suppose that \ca{X} and \ca{Y} are Polish spaces and that $\Gamma$ is a pointclass.  We denote by $\Gamma \upharpoonright \ca{X}$ the family of all subsets of \ca{X} which are in $\Gamma$.

A set $G \subseteq \ca{Y} \times \ca{X}$ \emph{parametrizes} $\Gamma \upharpoonright \ca{X}$ if for all $P \subseteq \ca{X}$ we have that
\[
P \in \Gamma \iff \textrm{exists $y \in \ca{Y}$ such that} \ P = G_y,
\]
\cf \cite{lusin_diagonal_method} and \cite{lebesgue_sur_les_fonctions}. With think of $y$ as a \emph{$\Gamma$-code} for $P$. The set $G$ is \emph{universal} for $\Gamma \upharpoonright \ca{X}$ if $G$ is in $\Gamma$ and parametrizes $\Gamma \upharpoonright \ca{X}$. A pointclass $\Gamma$ is \emph{$\ca{Y}$-parametrized} if for all Polish spaces \ca{Z} there exists some $G \subseteq \ca{Y} \times \ca{Z}$ which is universal for $\Gamma \upharpoonright \ca{Z}$.

For every pointclass $\Gamma$ of subsets we assign the \emph{boldface} class $\tboldsymbol{\Gamma}$ by saying that $P \subseteq \ca{X}$ is in $\tboldsymbol{\Gamma}$ if there exists a set $Q \subseteq \ca{N} \times \ca{X}$ in $\Gamma$ and an $\ep \in \ca{N}$ such that $P = Q_\ep$.

\subsubsection*{Uniformization} Suppose that $\ca{X}$ and $\ca{Y}$ are Polish spaces and that $P$ is a subset of $\ca{X} \times \ca{Y}$. A set $P^\ast \subseteq P$ \emph{uniformizes} $P$ if $\exists^{\ca{Y}}P^\ast = \exists^\ca{Y}P$ and for all $x \in \exists^\ca{Y}P$ there exists exactly one $y \in \ca{Y}$ with $P^\ast(x,y)$. In other words $P^\ast$ is the graph of a function which is defined on $\exists^\ca{Y}P$ such that $P(x,f(x))$ for all $x \in \exists^\ca{Y}P$.

We say that a pointclass $\Gamma$ has the \emph{uniformization property} if for all Polish spaces \ca{X}, \ca{Y} and all $P \subseteq \ca{X} \times \ca{Y}$ in $\Gamma$ there exists a set $P^\ast \subseteq P$ in $\Gamma$, which uniformizes $P$.

\subsubsection*{Recursive presentations} Suppose that $(\ca{X},d)$ is a complete and separable metric space. A sequence $(x_n)_{\n}$ is a \emph{recursive presentation} of $(\ca{X},d)$ if the set \set{x_n}{\n} is dense in \ca{X} and the relations $P_<, P_\leq \subseteq \om^4$ defined by
\begin{align*}
P_{<}(i,j,k,m) &\eq d(x_i,x_j) < k \cdot (m+1)^{-1}\\
P_{\leq} (i,j,k,m) &\eq d(x_i,x_j) \leq k \cdot (m+1)^{-1}
\end{align*}
are recursive.

We say that $(\ca{X},d)$ \emph{admits a recursive presentation}, or that \emph{$(\ca{X},d)$ is recursively presented}, if there is a sequence $(x_n)_{\n}$ in \ca{X} which is a recursive presentation of $(\ca{X},d)$. Standard examples of recursively presented Polish spaces are the Baire space \ca{N}, the space of real numbers and \om.

For every complete space $(\ca{X},d)$ with a recursive presentation $(x_n)_{\n}$ we define the sets
\[
B(x_n,m,k) = \set{x \in \ca{X}}{d(x,x_n) < k \cdot (m+1)^{-1}}
\]
for all $n,m,k \in \om$ and
\[
N(\ca{X},s) = B(x_{(s)_0},(s)_1,(s)_2)
\]
for all $s \in \om$. The family \set{N(\ca{X},s)}{s \in \om} is the associated \emph{neighborhood system of \ca{X}} with respect to $(x_n)_{\n}$ and $d$ and it is clear that it forms a basis for the topology of \ca{X}. When referring to a recursively presented metric space we will often omit the metric and the recursive presentation.

It is easy to see that recursive presentations are carried to finite products and topological sums of recursively presented metric spaces.

For every recursively presented metric space \ca{X} we say that $A \subseteq \ca{X}$ is \emph{semirecursive} if there is a recursive function $f: \om \to \om$ such that $$A = \cup_{s \in \om} N(\ca{X},f(s)).$$

The definition of semirecursive sets seems to depend on the way we have encoded the neighborhood system of \ca{X}, but as one can see from 3C.12 in \cite{yiannis_dst} this is not true.

It is worth mentioning that one can axiomatize the notion of a semirecursive set the same way one axiomatizes the notion of an open set when shifting from metric spaces to topological spaces. This is the setting of \emph{recursive Polish spaces} which are the effective analogue of Polish spaces: every recursively presented metric space determines a recursive Polish space and every recursive Polish space admits a compatible metric and a compatible recursive presentation. The term \emph{recursively presented metric space} can be replaced in all of its instances in this article by the term \emph{recursive Polish space}. We are not going to pursue this issue further, as the reader can refer to the upcoming notes \cite{gregoriades_moschovakis_notes_on_edst} for a detailed exposition of the topic.

\subsubsection*{The Kleene pointclasses} One defines by recursion the (Kleene) pointclasses
\begin{align*}
\Sigma^0_1 = \textrm{all semirecursive sets}, &\quad \Sigma^1_1 = \exists^\ca{N} \neg \Sigma^0_1,\\
\Sigma^0_{n+1} = \exists^\om \neg \Sigma^0_n, &\quad \Sigma^0_{n+1} = \exists^{\ca{N}} \neg \Sigma^0_n,\\
\Pi^0_n = \neg \Sigma^0_n, &\quad \Pi^1_n = \neg \Sigma^1_n,\\
\Delta^0_n = \Sigma^0_n \cap \Pi^0_n, &\quad \Delta^1_n = \Sigma^1_n \cap \Pi^1_n.
\end{align*}
The pointclasses $\Sigma^0_n$ are the \emph{arithmetical} and $\Sigma^1_n$ are the \emph{analytical} Kleene pointclasses.

The Kleene pointclasses can be relativized with respect to some parameter. Suppose that \ca{X} and \ca{Y} are recursively presented metric spaces, $y \in \ca{Y}$ and $A \subseteq \ca{X}$. We say that $A$ is \emph{$y$-semirecursive} if there is a semirecursive $P \subseteq \ca{Y} \times \ca{X}$ such that $A = P_y$. Similarly one defines the pointclasses $\Sigma^0_n(y)$, $\Pi^0_n(y)$, $\Delta^0_n(y)$, $\Sigma^1_n(y)$, $\Pi^1_n(y)$ and $\Delta^1_n(y)$.

The boldface version of the preceding pointclasses are exactly the corresponding classical pointclasses, \ie $\tboldsymbol{\Sigma}^0_1$ is the pointclass of all open sets and so on.

\subsubsection*{Points in $\Gamma$} For every pointclass $\Gamma$ and every recursively presented metric space \ca{X}, a point $x \in \ca{X}$ \emph{is in $\Gamma$} or \emph{it is a $\Gamma$ point} if the relation $U \subseteq \om$ defined by $$U(s) \eq x \in N(\ca{X},s)$$ is in $\Gamma$. We use the notation $x \in \Gamma$ to denote that $x$ is a $\Gamma$ point.

We say that $x$ is a \emph{recursive} point if $x \in \Sigma^0_1$ and that $x$ is \emph{recursive in $y$} if $x \in \Sigma^0_1(y)$. This is equivalent to the \emph{Turing reducibility} (see 3D.19 in \cite{yiannis_dst}).

We will be mostly concerned with the cases $\Gamma = \del$ and $\Gamma = \del(y)$. The point $x$ is \emph{hyperarithmetical} if $x$ is a \del \ point and $x$ is \emph{hyperarithmetical} in $y$ if $x$ is a $\del(y)$ point. We will also say that $x$ and $y$ have the \emph{same hyperdegree} if $x \in \del(y)$ and $y \in \del(x)$.

Furthermore we write
\begin{align*}
x \hleq y, \ \ & \ \textrm{if} \ x \in \del(y),\\
x \hless y, \ \  & \ \textrm{if} \ x \in \del(y) \ \textrm{and} \ y \not \in \del(x),\\
x \heq y, \ \  & \ \textrm{if} \ x \in \del(y) \ \textrm{and} \ y \in \del(x).
\end{align*}

We will often identify a subset of \om \ with its characteristic function, so when we write $x \in \Gamma(P)$ for some $P \subseteq \om$ we mean that $x$ is a $\Gamma(\chi_P)$ point, where $\chi_P \in \ca{C}$ is the characteristic function of $P$.

\subsubsection*{$\Gamma$-recursive functions} Suppose that $\Gamma$ is a pointclass. A function $f: \ca{X} \to \ca{Y}$ between recursively presented metric spaces is \emph{$\Gamma$-recursive} if the relation $R^f \subseteq \ca{X} \times \om$ defined by
\[
R^f(x,s) \iff f(x) \in N(\ca{Y},s)
\]
is in $\Gamma$. In most cases $\Gamma$ will be $\Sigma^0_1$ or \del. We will often say \emph{recursive} rather than \emph{$\Sigma^0_1$-recursive}.

We think of $\Gamma$-recursive functions as the \emph{effective refinement} of $\tboldsymbol{\Gamma}$-measurable functions. In particular \del-recursive functions can be considered as the effective analogue of Borel-measurable functions.

A function $f: \ca{X} \to \ca{Y}$ is a \emph{$\Gamma$-injection} if it is injective and $\Gamma$-recursive. Moreover $f$ is a \emph{$\Gamma$-isomorphism} if $f$ is bijective, $\Gamma$-recursive and the inverse $f^{-1}$ is also $\Gamma$-recursive.

We will also consider \emph{partial functions} $f: \ca{X} \rightharpoonup \ca{Y}$. The \emph{domain} $\Dom(f)$ of $f$ is the set of all $x \in \ca{X}$ for which $f$ is defined. We also write $f(x) \downarrow$ to say that $f$ is defined in $x$.

A partial function $f: \ca{X} \rightharpoonup \ca{Y}$ is \emph{$\Gamma$-recursive on} $A \subseteq \Dom(f)$ if there exists some $P^f \in \Gamma$ such that
\[
P^f(x,s) \iff f(x) \in N(\ca{Y},s)
\]
for all $x \in A$. We say that the preceding set $P^f$ \emph{computes $f$ on $A$}. The partial function $f$ is \emph{$\Gamma$-recursive} if it is $\Gamma$-recursive on its domain and $\Dom(f) \in \Gamma$.

\begin{lemma}[Well-known, \cf 3D.7, 4D.1 and 4D.7 in \cite{yiannis_dst}.]
\label{lemma the inverse function is del and same hyperdegree}
Suppose that \ca{X} and \ca{Y} are recursively presented metric spaces and that $$f : \ca{X} \to
\ca{Y}$$ is \del-recursive function. Then the following hold.
\begin{enumerate}
\item[(1)] $f(x) \in \del(x)$ for all $x \in \ca{X}$.
\item[(2)] If $A$ is non-empty \sig \ subset of \ca{X} such that  $f$ is injective on $A$ then $x \in \del(f(x))$ for all $x \in A$. It follows that $f(x) \heq x$ for all $x \in A$.
\item[(3)]If $A$ is non-empty \sig \ subset of \ca{X} such that $f$ is injective on $A$, then the inverse function $f^{-1}: \ca{Y} \rightharpoonup \ca{X}$ is partial \del-recursive on $f[A]$. In particular if $f$ is a \del-injection and surjective then $f$ is a \del-isomorphism.
\end{enumerate}
\end{lemma}

\begin{proof}[\textit{Proof} \tu{(Well-known).}]
Assertion (1) follows from 3D.7 in  \cite{yiannis_dst}. Suppose that $A$ is a non-empty \sig \ subset of \ca{X} such that $f$ is injective on $A$. Then we have that
\[
z = x \iff z \in A \ \& \  f(z) = f(x)
\]
for all $z \in \ca{X}$ and all $x \in A$. From this it follows that the singleton $\{x\}$ is $\sig(f(x))$ and so from 3E.16 in \cite{yiannis_dst} $x$ is in $\del(f(x))$. This proves (2).

Regarding (3) we notice that
\[
f^{-1}(y) = x \iff x \in A \ \& \ f(x) = y.
\]
The graph of $f$ is \sig \ from 3E.5 in \cite{yiannis_dst}, and so from the preceding equivalence the graph of $f^{-1}$ is \sig \ as well. Hence again from 3E.5 in \cite{yiannis_dst} the function $f^{-1}$ is \del-recursive.
\end{proof}

\subsubsection*{Codes of $\Gamma$-recursive functions} (See 7A in \cite{yiannis_dst}.) Suppose that $\Gamma$ is a Kleene pointclass. For all recursively presented metric spaces \ca{X} and \ca{Y} we fix a set $G_{\Gamma}^\ca{X} \subseteq \om \times \ca{X} \times \om$, which is universal for $\Gamma \upharpoonright (\ca{X} \times \om)$ and we define the partial function
\[
U^{\ca{X},\ca{Y}}_{\Gamma} : \ca{X} \rightharpoonup \ca{Y}
\]
as follows:
\begin{align*}
&U^{\ca{X},\ca{Y}}_{\Gamma}(e,x) \ \textrm{is defined} \iff \ \textrm{there exists a unique $y$ in \ca{Y} such that}\\
&\hspace*{51mm} (\forall s)[y \in N(\ca{Y},s) \longleftrightarrow G(e,x,s)]\\
&U^{\ca{X},\ca{Y}}_{\Gamma}(e,x) = \textrm{the unique $y$ as above}.
\end{align*}
It is clear that $U^{\ca{X},\ca{Y}}_{\Gamma}$ is $\Gamma$-recursive on its domain. For every $e \in \om$ we define the partial $\Gamma$-recursive function
\[
\rfn{e}^{\ca{X},\ca{Y}}_{\Gamma}: \ca{X} \rightharpoonup \ca{Y}: \rfn{e}^{\ca{X},\ca{Y}} = U^{\ca{X},\ca{Y}}_{\Gamma}(e,x).
\]
It is not hard to verify that a partial function $f: \ca{X} \rightharpoonup \ca{Y}$ is $\Gamma$-recursive on its domain if there exists some $e \in \om$ such that $f(x) = \rfn{e}^{\ca{X},\ca{Y}}_\Gamma(x)$ for all $x \in \Dom(f)$.

\subsubsection*{Codes of linear orderings on countable domains} In the sequel we denote by $\rfn{e}$ the function $\rfn{e}^{\om,\om}_{\Sigma^0_1}$. The set \lo \ of  \emph{codes of linear orderings} on a subset of \om \ is
\begin{align*}
\set{e \in \om}{\rfn{e} \ \textrm{is total and the relation}& \ \set{(n,m)}{\rfn{e}(\langle n,m \rangle) = 0}\\& \textrm{is a linear ordering}}.
\end{align*}
It is easy to verify that \lo \ is a $\Sigma^0_2$ subset of \om. We denote by $\leq_e$ the linear ordering
\[
\set{(n,m) \in \om}{\rfn{e}(\langle n,m \rangle) = 0}
\]
for every $e \in \lo$.

\subsubsection*{Spector's \W} It is useful to fix a canonical \pii \ subset of \om. \emph{Spector's \W} \ is the set
\[
\set{e \in \lo}{\leq_e \ \textrm{has no strictly decreasing sequences}},
\]
It is well-known that \W \ is a \pii \ set and moreover that it is \pii-\emph{complete} for subsets of \om, \ie for all \pii \ sets $P \subseteq \om$ there exists a recursive function $f: \om \to \om$ such that
\[
e \in P \iff f(e) \in \W
\]
for all $e \in \om$, \cf \cite{spector_recursive_well-orderings}. In particular \W \ is not a \sig \ set.

\subsubsection*{Relativization} The context of recursively presented metric spaces might seem somewhat restrictive, but our results are carried out in all Polish spaces in a natural way with the proper \emph{relativization} arguments.

A sequence $(x_n)_{\n}$ in a complete separable metric space $(\ca{X},d)$ is an \emph{\ep-recursive presentation of $(\ca{X},d)$}, where $\ep \in \ca{N}$, if the set $\set{x_n}{\n}$ is dense in \ca{X} and the preceding relations $P_<$, $P_\leq$ are \ep-recursive. By replacing the term ``recursive" with ``\ep-recursive" in all of its instances one can repeat all preceding definitions in this relativized setting. It is easy to see that every complete separable metric space admits an $\ep$-recursive presentation for some $\ep \in \ca{N}$. Analogously every Polish space is an \emph{\ep-recursive Polish space} for some $\ep \in \ca{N}$.

We think of the class of $\ep$-recursively presented metric spaces as a ``fiber" among the class of complete separable metric spaces. Each complete separable metric space belongs to one such fiber, in fact to infinitely many. All theorems about recursively presented metric spaces are transferred to $\ep$-recursively presented metric spaces, \ie these fibers have a similar effective structure. For more information about relativization the reader can refer to 3I in \cite{yiannis_dst} and \cite{gregoriades_turning_borel_sets_into_clopen_sets_effectively}. For the rest of this article we restrict ourselves in one such fiber of spaces, which we may assume  without loss of generality that is the fiber of recursively presented metric spaces.

\subsection{Theorems in the general context} As mentioned before most of the main results of effective descriptive set theory do hold in the non-perfect setting. All following results go through either in a straightforward way or with some mild
modifications.\footnote{For example the proof of Theorem \ref{theorem
alltogether}-(\ref{theorem 4D.2 and theorem 4D.15}) requires the
following modification; instead of taking a \del-bijection $\pi :
\ca{N} \to \ca{X}$ it is enough to take the recursive surjection
$\pi: \ca{N} \twoheadrightarrow \ca{X}$ of Theorem \ref{theorem
alltogether}-(\ref{theorem 3E.6}) and notice that for all $x \in
\ca{X}$ there is an $\alpha \in A$ such that $\alpha \in \del(x)$
and $\pi(\alpha) = x$.} The references below are to
\cite{yiannis_dst}.

\allowbreak

\begin{theorem}
\label{theorem alltogether}
\newcounter{remarks}
Suppose that \ca{X} is a recursively presented metric space.\vspace{1mm}

\begin{list}{\textup{(\arabic{remarks})}}{\usecounter{remarks}\leftmargin=2mm}
\item{\textup{(3E.6).}} \label{theorem 3E.6}
There is a recursive surjection $\pi: \ca{N} \twoheadrightarrow
\ca{X}$ and a $\Pi^0_1$ subset of the Baire space $A$ such that
$\pi$ is injective on $A$ and $\pi[A] = \ca{X}$. Furthermore
there is a \del-injection $f: \ca{X} \rightarrowtail \ca{N}$
which is the inverse of $\pi$ restricted on $A$, \ie for all
$\alpha \in A$ we have that $f(\pi(\alpha)) = \alpha$ and for all
$x \in f^{-1}[A]$ we have that $\pi(f(x)) = x$.\vspace{3mm}

\item{\textup{The Effective Perfect set Theorem (4F.1, see also \cite{harrison_phd_thesis} and \cite{mansfield_perfect_subsets_of_definable_sets_of_real_numbers}).}} \label{theorem 4F.1}
Every \sig \ subset of \ca{X}, which has at least one member
not in \del, contains a perfect subset.\vspace{3mm}

\item{\textup{The Parametrization Theorem for points in \del (4D.2 and 4D.15).}}
\label{theorem 4D.2 and theorem 4D.15} There is a
\pii-recursive partial function $\pdel : \om \rightharpoonup \ca{X}$ such that
for all $x \in \ca{X}$,
\[
x \in \del \iff (\exists i)[\pdel(i) \ \textrm{\normalfont is defined and} \
\pdel(i)=x].
\]
Moreover there is a \pii-recursive partial function $\mathbf{c}: \ca{X}
\rightharpoonup \om$ such that for all $x \in \ca{X}$, $\mathbf{c}(x)$ is
defined if and only if $x$ is \del, and $\pdel(\mathbf{c}(x))=x$ for all $x \in
\del$.\vspace{3mm}

\item{\textup{Upper bound of the set of points in \del \ (4D.14).}} \label{theorem
4D.14} The set of points of \ca{X} which are in \del \ is a \pii \
set.\vspace{3mm}

\item{\textup{Theorem on Restricted Quantification, Kleene (4D.3).}} \label{theorem 4D.3}
Let \ca{Y} be a recursively presented metric space and let $Q
\subseteq \ca{X} \times \ca{Y}$ be \pii. The set $P \subseteq \ca{X}$ defined by $$P(x) \iff (\exists
y \in \del(x))Q(x,y)$$ is \pii \ as well.\vspace{3mm}

\item{\textup{\del-Uniformization Criterion (4D.4).}} \label{theorem 4D.4}
Let \ca{Y} be a recursively presented Polish space and let $P
\subseteq \ca{X} \times \ca{Y}$ be in \del. Assume that every
section $P_x$ is either empty or contains a member in $\del(x)$. Then the projection
$\exists^\ca{Y}P$ is
also in \del. Furthermore the set $P$ can be uniformized by some
set in \del.\vspace{3mm}

\item{\textup{Characterization of \del \ sets (4D.9).}} \label{theorem 4D.9}
A set $P \subseteq \ca{X}$ is in \del \ if and only
if there is a $\Pi^0_1$ set $A \subseteq \ca{N}$ and a recursive
function $\pi: \ca{N} \to \ca{X}$ which is injective on $A$ and
$\pi[A] = P$.\vspace{3mm}

\item{\textup{Parametrization (3F.6 and 3H.1, see also \cite{kleene_recursive_predicates_and_quantifiers}).}} \label{theorem parametrization over N}
Every Kleene pointclass is $\om$- and $\ca{N}$-parametrized. In fact if $\Gamma$ is a Kleene pointclass then for every recursively presented metric space \ca{X} there exists a set $G \subseteq \ca{N} \times \ca{X}$ in $\Gamma$, which parametrizes $\tboldsymbol{\Gamma} \upharpoonright \ca{X}$. Moreover a set $P \subseteq \ca{X}$ is in $\Gamma$ exactly when $P = G_\ep$ for some recursive $\ep \in \ca{N}$.\footnote{The same is true for perfect \ca{Y} instead of \ca{N} but the proof does not go through for non-perfect \ca{Y}. Whether this assertion holds for non-perfect \ca{Y} is an open problem, \cf Question \ref{question parametrization}.}.\vspace{3mm}

 \item{\textup{The Suslin-Kleene Theorem (4B.11 and 7B.4).}} \label{theorem 4B.11}
Let $A, B$ be two disjoint subsets of \ca{X} which are in \sig.
Then there is a \del \ set $C \subseteq \ca{X}$ such that $A
\subseteq C$ and $C \cap B = \emptyset$. In particular all \del \
subsets of \ca{X} are Borel.

In fact there is a recursive function $\nu: \ca{N} \times \ca{N} \to
\ca{N}$ such that if $\alpha, \beta \in \ca{N}$ are
$\tboldsymbol{\Sigma}^1_1$ codes of sets $A, B \subseteq \ca{X}$
respectively and $A \cap B = \emptyset$, then $\nu(\alpha,\beta)$
is a Borel code of some set $C$ which separates $A$ from $B$.

In particular there is a recursive function $\nu^\ast: \ca{N} \to
\ca{N}$ such that if $\alpha$ is a $\tboldsymbol{\Delta}^1_1$-code of
$A \subseteq \ca{X}$, then $\nu^\ast(\alpha)$ is a Borel code of
$A$.\footnote{$\tboldsymbol{\Sigma}^1_1$ and $\tboldsymbol{\Pi}^1_1$ codes are defined through the corresponding universal sets. One defines $\tboldsymbol{\Delta}^1_1$-codes as pairs of $\tboldsymbol{\Sigma}^1_1$ and $\tboldsymbol{\Pi}^1_1$ codes of the same set. Borel codes are defined ``from inside" using transfinite induction. The reader can refer to 3H and 7B in \cite{yiannis_dst} for more information.}\vspace{3mm}

\item{\textup{The Kondo-Addison Theorem - Uniformization of \pii \ (4E.4, see also Kondo \cf \cite{kondo_uniformisation} and 4E.4 \cite{yiannis_dst}).}} \label{theorem 4E.4}
The pointclass \pii \ and consequently the pointclass \boldpii \ have the uniformization property.
\end{list}
\end{theorem}

\subsection{The Cantor-Bendixson decomposition} The \emph{scattered part} \scat{A} of a closed set $A$ is the set
\[
\scat{A} = \set{x \in A}{(\exists \ \textrm{open} \ V ) [ x \in V \ \& \ V \cap A \ \textrm{is countable}] },
\]
and the \emph{perfect kernel} \pker{A} of $A$ is $A \setminus \scat{A}$. The \emph{Cantor-Bendixson Theorem} states that \scat{A} is countable, \pker{A} is perfect and that the pair $(\pker{A},\scat{A})$ is the unique such decomposition of $A$.

An early indication that the proof of the statement ``every perfect recursively presented metric space is \del-isomorphic to the Baire space" cannot be modified successfully in the non-perfect setting, can be found in the following result of Kreisel \cite{kreisel_cantor_bendixson}: there is a $\Pi^0_1$ subset of the real numbers whose perfect kernel is not \pii. In other words we cannot insist on staying inside the perfect kernel (a standard argument for embedding the Baire space) and preserve \del-recursiveness. However this is only a part of the picture, for as we are going to see it is possible for a recursively presented metric space to be non-\del-isomorphic to the Baire space and yet its perfect kernel is a $\Pi^0_1$ set. The deeper reason for not being \del-isomorphic to the Baire space is the absence of ``many" \del-points inside the perfect kernel of the space, \cf Theorem \ref{theorem characterization of Baire space in terms of perfect sets}.

We review the main results of \cite{kreisel_cantor_bendixson} using modern day notation and proofs.\footnote{The reader may find
Kreisel's notation a bit confusing compared to today's notation.
According to Kreisel a closed subset $F$ of the unit interval is
called \emph{\pii- \tu{(}\sig-, recursive-\tu{)} closed} if the
relation $$P(n,m) \eq (r_n,r_m) \subseteq [0,1] \setminus F$$ is
\pii, (\sig, recursive), where the sequence $(r_n)_{\n}$ is a
recursive enumeration of the rational numbers in $[0,1]$. Notice
the interchange of the symbols $\Sigma$ and $\Pi$. For instance a
closed set is \pii-closed if and only if it is a \sig \ set.}

\begin{theorem}[\cf \cite{kreisel_cantor_bendixson}]
\label{theorem kreisel scattered part}
Let $\ca{X}$ be a recursively presented metric space and $A$
be a closed \sig \ subset of \ca{X}.

\tu{(1)} The perfect kernel \pker{A} of $A$ is a \sig \ subset of \ca{X}.
It follows that the scattered part \scat{A} of $A$ is the difference
of two \sig \ sets. In particular if $A$ is a \del \ set then its
scattered part is a \pii \ set.

\tu{(2)} If \scat{A} is non-empty then there is a \W-recursive sequence $(x_i)_{\iin}$ which enumerates \scat{A}.

\tu{(3)} There exists a closed set in $\pii \setminus \sig$ whose perfect kernel is in $\sig \setminus \pii$.
\end{theorem}

Kreisel \cite{kreisel_cantor_bendixson} credits assertion (1) to Lorenzen. In Theorem \ref{theorem CB decomposition of Spector-Gandy spaces} we show that the perfect kernel of the whole space \ca{X} \ may not be a \pii \ set and thus the first and the last assertion of (1) are optimal. In Theorem \ref{theorem closed sig set} we show that the scattered part of a \sig \ closed set is not necessarily \pii. This shows that the second assertion of (1) is also optimal.

\begin{proof}
Define
\[
Q(t) \iff \textrm{the set $N(\ca{X},t) \cap A$ is countable}.
\]
We claim that $Q$ is a \pii \ relation on \om. Indeed using the Effective Perfect Set
Theorem \ref{theorem alltogether}-(\ref{theorem 4F.1}) we
compute
\begin{align*}
Q(t) \iff& \textrm{the set $A \cap N(\ca{X},t)$ is countable}\\
        \iff& (\forall y)[y \in A \cap N(\ca{X},t) \ \rightarrow \ y \in\del],
\end{align*}
so that
\begin{align*}
x \in \pker{A} \iff& x \in A \ \& \ (\forall t)[ x \in N(\ca{X},t) \
\rightarrow \ t \not \in Q]
\end{align*}
Hence \pker{A} is in \sig. The remaining assertions in (1) follow
easily.

Now we prove (2). Let $Q \subseteq \om$ be as above and let $\pdel : \om \rightharpoonup \ca{X}$ be a \pii-recursive partial function which parametrizes the points of \ca{X} which are in
\del. We adopt the notation $\pdel(n) \downarrow$ for when $\pdel(n)$
is defined. Since $\pdel$ is \pii-recursive we have in particular that the set
\[
I = \set{n}{\pdel(n) \downarrow}
\]
is \pii. Put
$$
R(t,n) \iff Q(t) \ \& \ d(n) \downarrow \ \& \ d(n) \in N(\ca{X},t)
$$
Clearly $R$ is a \pii \ set and
\begin{align*}
x \in \scat{A} \eq (\exists t,n)[R(t,n) \ \& \ \pdel(n) \in A \ \& \ x =
\pdel(n)] \ \ \ \ \ \ \ (*)
\end{align*}
Put also
$$
P(n,w) \eq [ \ (\exists t)R(t,n) \ \& \ w=0 \ ] \ \vee \ [ \ \pdel(n)
\downarrow \ \& \ \pdel(n) \not \in A \ \& \ w=1 \ ]
$$
Using that $\pdel$ is partial \pii-recursive one can verify that $P$ is a \pii \ set.

From $(*)$ it follows that
\[
\scat{A} = \set{\pdel(n)}{(n,0) \in P \
\& \ (n,1) \not \in P}.
\]
Since $\scat{A}$ is non-empty there exists some $n$ such that $(n,0) \in P$ and $(n,1) \not \in P$. Let $n_0$ be the least such integer.
Finally we define the function
\[
f: \om \to \om: f(x) = \begin{cases}
                                      \pdel(n), & \ \textrm{if} \ (n,0) \in P \
\& \ (n,1) \not \in P\\
                                      \pdel(n_0), & \ \textrm{else}.
                                      \end{cases}
\]
It is clear that $f$ enumerates \scat{A}. Moreover by taking a \pii \ set $R^\pdel \subseteq \om \times \om$ which computes $\pdel$ on its domain and using the \pii-completeness of \W, we have that $f$ is \W-recursive.

Assertion (3) is a corollary of Theorem \ref{theorem CB decomposition of Spector-Gandy spaces}.
\end{proof}

We single out one of the arguments used in the preceding proof, that we will use often in the sequel.

\begin{remark}
\label{remark scattered part consists of del points}\normalfont
The scattered part of a \sig \ set consists of \del \ points.
\end{remark}

\section{\sc The spaces \spat{T}}
\label{section spaces Nt}

In this section we define the scheme $T \mapsto \spat{T}$ that we mentioned in the Introduction and we prove the basic properties about it. We also give some preliminary results about \del-isomorphisms between recursively presented metric spaces.

\subsection{Definition and properties}
We consider every finite sequence $u$ as a partial function on an initial
segment of \om \ to \om. For all $\alpha \in \ca{N}$ and $i, n \in \om$ we write $u(i) \neq \alpha(n)$, if either $ i < \lh(u)$ and
$u(i) \neq \alpha(n)$, or $i \geq \lh(u)$ -and so $u(i)$ is
undefined.

\begin{definition}
\label{definition of spaces Nt}\normalfont
For every tree $T$ on \om \ we define the space
$(\spat{T},d^T)$ as follows
\[
\spat{T} = T \cup [T]
\]
and
\[
d^T(x,y) = (\textrm{least} \ n[x(n) \neq y(n)] + 1)^{-1}
\]
for $x,y \in \spat{T}$. It is easy to verify that $d^T$ is a metric on $\spat{T}$.
\end{definition}

There are various ways to view the spaces \spat{T}. One way is to view them as $G_\delta$ subsets of the plane \cf Figure \ref{fig:space-nt}.\footnote{The idea of drawing a tree on the plane seems promising for connecting recursion theory with topology. There are some interesting recent results regarding computability and connectedness in compact $\Pi^0_1$ subsets of the plane \cite{kihara_incomputability_simply_connected_planar_continua}, in which the notion of a \emph{dendrite} (a tree-like compact connected set) is a key tool.}

Another perhaps more useful way is to view these spaces as subsets of the Baire space. For this we need to identify every finite sequence $u$ with the total
function, which agrees with $u$ on all $i < \lh(u)$ and takes the constant value $-1$ from \lh(u) and on. For an infinite sequence of
integers $x$ we adopt the notation $x+1$ for the sequence
\[
(x(0)+1, \dots,x(n-1)+1,\dots),
\]
so that (under the preceding convention about finite sequences) we
have that
\[
u+1 = (u(0)+1,\dots,u(\lh(u)-1)+1,0,0,\dots)
\]
for all $u \in \omseq$.

\begin{lemma}
\label{lemma the space Nt is embedded in the Baire space}
Suppose that $T$ is a tree on \om. We consider the shifted tree
\[
T^{+1} = \set{(u(0)+1,\dots,u(\lh(u)-1)+1)}{u \in T}
\]
and the tree
\[
S^T = T^{+1} \cup \set{v \ \hat{} \ 0^{(n)}}{v \in T^{+1}, \n}.
\]
Then the function $\rho^T: \spat{T} \inj \ca{N}$ defined by
\[
\rho^T(x) = (x(0)+1,\dots,x(n-1)+1,\dots)
\]
is an isometry between $(\spat{T},d^T)$ and $([S^T], d_{\ca{N}})$, where $d_\ca{N}$ is the usual metric on the Baire space restricted accordingly.\smallskip

Therefore the space \spat{T} is topologically isomorphic to a closed subset of the Baire space and in particular it is Polish.
\end{lemma}

The proof of the preceding lemma is straightforward and we omit it. Although viewing the members of $\spat{T}$ as infinite sequences may sometimes be more convenient, we point out that we will not alter the definition of $T$, \ie unless stated otherwise in all of our statements and
proofs the points of $T$ are formally finite sequences.

A tree $T$ is \emph{recursive} if the set of codes $\set{s}{\dec{s} \in T}$ is a recursive subset of \om, where $\dec{\cdot}$ is the decoding function defined in the Introduction. The following result describes the connection between recursive trees and $\Pi^0_1$ subsets of \ca{N}.

\begin{lemma}[Well-known \cf 4A.1 \cite{yiannis_dst}]
\label{lemma pi-zero-one sets recursive tree}
A set $P \subseteq \ca{N}$ is in $\Pi^0_1$ exactly when there exists a recursive tree $T$ on \om \ such that $P = [T]$.
\end{lemma}

We state the following fundamental theorem of Kleene.

\begin{theorem}[Kleene \cf \cite{kleene_hierarchies_number_theoretic_predicates} and 4D.10 \cite{yiannis_dst}]
\label{theorem Kleene there exists a Kleene tree}
There exists a recursive tree $T$ which has non-empty body but no $\alpha \in \del$ is a member of $[T]$.
\end{theorem}

\begin{figure}[t]
\begin{picture}(400,130)(0,-15)
\put (50,80){\vector(-1,-1){20}}
\put (50,80){\vector(0,-1){21}}
\put (50,80){\vector(1,-1){21}}
\put (50,80){\vector(2,-1){40}}
\put (50,80){\vector(3,-1){60}}
\put (120,60){$\dots$}
\put (30,60){\vector(-2,-3){13}}
\put (50,59.5){\vector(3,-4){15}}
\put (65,40){\vector(-3,-4){15}}
\put (50.5,21){\vector(0,-1){20}}
\put(49,-12){$\vdots$}
\qbezier(120,99)(160,120)(198,101)
\put (200,99){\vector(1,-1){1}}

\put (235,80){\tiny{$\bullet$}}
\put (210,55){\tiny{$\bullet$}}
\put (240,55){\tiny{$\bullet$}}
\put (265,55){\tiny{$\bullet$}}
\put (285,55){\tiny{$\bullet$}}
\put (300,55){\tiny{$\bullet$}}
\put (310,55){$\dots$}
\put (200,35){\tiny{$\bullet$}}
\put (250,35){\tiny{$\bullet$}}
\put (245,20){\tiny{$\bullet$}}
\put (243,12){\tiny{$\bullet$}}
\put (243,-17.5){\tiny{$\bullet$}}
\put (190,-15){\line(1,0){140}}
\put (235,-8){$\vdots$}
\qbezier[20](238,82)(235,80)(242,57)
\qbezier[17](242,57)(246,55)(252,37)
\qbezier[20](252,37)(254,35)(246,20)
\qbezier[10] (247,22)(245,20)(245,10)
\qbezier[20](245,12)(243,-8)(246,-16)
\put (75,100){$T$}
\put (260,100){\spat{T}}
\end{picture}
\caption{The space \spat{T} as a subset of the plane.}
\label{fig:space-nt}
\end{figure}
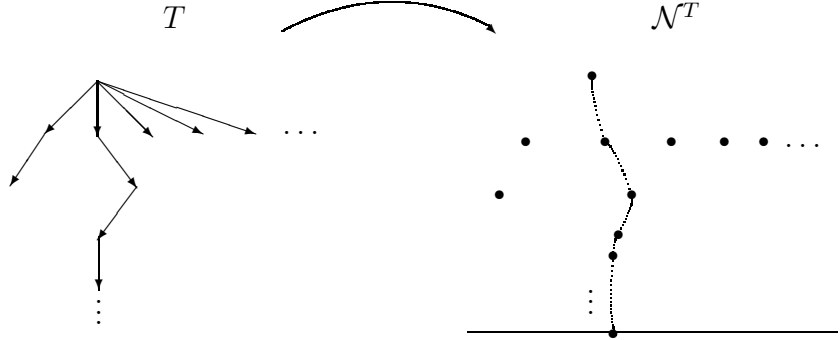

Now we summarize the basic properties of the spaces \spat{T}.

\begin{theorem}
\label{theorem properties of Nt}

For every tree $T$ on \om \ the space \spat{T} satisfies the
following properties.

\tu{(1)} Every point of $T$ is an isolated point of \spat{T}. Moreover
\[
\scat{\spat{T}} = \scat{[T]} \cup T,
\]
where $\scat{X}$ stands for the scattered part of $X$.

\tu{(2)} For any $u \in T$ consider the set $N(T,u) = \set{x \in
\spat{T}}{u \sqsubseteq x}$. The family
\[
\set{N(T,u)}{u \in T} \cup \set{\{u\}}{u \in T}
\]
forms a basis for the topology of \spat{T}. In particular the tree
$T$ is dense in \spat{T}. Moreover the body $[T]$ is a closed and nowhere dense subset of $\spat{T}$.

\tu{(3)} The function $\dense^T: \om \to \spat{T}:$
\[
r_s =
\begin{cases}
\dec{s}, & \ \textrm{if} \ s \in \Seq \ \& \ \dec{s} \in T,\\
\emptyset, & \ \textrm{else},
\end{cases}
\]
is a $T$-recursive presentation of $(\spat{T},d^T)$. In particular if $T$ is a recursive tree the space $(\spat{T},d^T,\dense^T)$ is recursively presented and is isometric to a $\Pi^0_1$ subset of \ca{N}.

\tu{(4)} If $T$ is recursive then $T$ is a $\Sigma^0_1$ subset of
\spat{T} and so $[T]$ is a $\Pi^0_1$ subset of \spat{T}.
\end{theorem}

\begin{proof} Assertions (1)-(3) are immediate from the definitions. Regarding (4) we notice that
\begin{align*}
x \in T \iff (\exists s)[s \in \Seq \ \& \ \dec{s} \in T \ \& \ d^T(\dense^T(s),x) <
\frac{1}{\lh(s)+1}]
\end{align*}
since $u$ is the unique point in the ball about it with radius smaller than $1/(\lh(u)+1)$.
\end{proof}

The body $[T]$ is open in \spat{T} exactly when it is the empty set. This is because $T$ is dense in \spat{T} and $T \cap [T] = \emptyset$. Therefore the upper bound given in (4) of the preceding theorem is the best possible in general.

\begin{lemma}
\label{lemma transition between nhbds Nt and N}
For all recursive trees $T$ there are recursive relations $R^T \subseteq \spat{T} \times \omega$ and $Q^T
\subseteq \ca{N} \times \omega$ such that
\begin{align*}
              x \in N(\spat{T},s) &\iff Q^T(x,s)\\
              y \in N(\ca{N},t)     &\iff R^T(y,t)
\end{align*}
for all $(x,s), (y,t) \in [T] \times \om$.
\end{lemma}

One can verify the preceding lemma by standard computations or by observing that the function $\rho^T$ in Lemma \ref{lemma the space Nt is embedded in the Baire space} is recursive and its inverse (as a partial function) is recursive. We omit the details.

\begin{theorem}
\label{theorem points of Nt}
For every recursive tree $T$ and for every $x \in [T]$ we have that
\[
x \ \textrm{is a $\Gamma$-point of \spat{T}} \iff x \ \textrm{is a
$\Gamma$-point of \ca{N}},
\]
where $\Gamma$ is any one of the pointclasses $\Sigma^0_n$, $\Sigma^1_n$ and $\Pi^1_n$ for $n \geq 1$.
\end{theorem}

\begin{proof}
Consider the recursive relations $R^T \subseteq \spat{T} \times \omega$ and $Q^T
\subseteq \ca{N} \times \omega$ in Lemma \ref{lemma transition between nhbds Nt and N}. From 3C.5 in \cite{yiannis_dst} there are semirecursive $R^\ast, Q^\ast \subseteq \om^2$ such that
\begin{align*}
R^T(x,t) \iff& (\exists s)[x \in N(\spat{T},s) \ \& \ R^\ast(t,s)]
\end{align*}
for all $(x,t) \in \spat{T} \times \omega$ and
\begin{align*}
Q^T(x,t) \iff& (\exists s)[x \in N(\ca{N},s) \ \& \ Q^\ast(t,s)]
\end{align*}
for all $(x,t) \in \ca{N} \times \omega$. For any $x \in \ca{N}$ we have that
\begin{align*}
x \in N(\ca{N},t) \iff& R^T(x,t)\\
                  \iff& (\exists s)[x \in N(\spat{T},s) \ \& \ R^\ast(t,s)],
\end{align*}
which shows that if $x \in [T]$ is a $\Gamma$-point of \spat{T} then $x$ is a $\Gamma$-point of \ca{N}.

For the converse direction we use the relations $Q^T$ and $Q^\ast$ from above.
\end{proof}

\begin{theorem}
\label{theorem pointclasses of Nt}
Suppose that $T$ is a recursive tree and that $\Gamma$ is one of the
pointclasses $\Pi^0_1$, $\Sigma^0_{n+1}$, $\Pi^0_{n+1}$, $\Sigma^1_{n+1}$ and
$\Pi^1_{n+1}$, \tu{\n}. For every recursively presented metric space
\ca{Z} a set $A \subseteq [T] \times \ca{Z}$ is in $\Gamma$ as a subset of $\spat{T} \times \ca{Z}$ exactly when there
exists $B \subseteq \ca{N} \times \ca{Z}$ in $\Gamma$ such that $A = B \cap ([T] \times \ca{Z})$. In particular we have
that
\[
A \in \Gamma(\spat{T} \times \ca{Z}) \iff A \in \Gamma(\ca{N}
\times \ca{Z})
\]
for every $A \subseteq [T] \times \ca{Z}$.
\end{theorem}

\begin{proof}
We show first the assertion for $\Pi^0_1$. For simplicity we show the non-parametrized version, \ie we omit the space
\ca{Z}. Let $A \subseteq [T]$ be a $\Pi^0_1$ subset of \spat{T}. The complement of $A$ in \spat{T} is a $\Sigma^0_1$ and
so there exists a semirecursive $P^\ast \subseteq \omega$ such that
\[
x \in \spat{T} \setminus A \iff (\exists s)[x \in N(\spat{T},s) \ \& \ P^\ast(s)]
\]
for all $x \in \spat{T}$. Consider the recursive relations $R^T \subseteq \spat{T} \times \omega$ and $Q^T \subseteq
\ca{N} \times \omega$ in Lemma \ref{lemma transition between nhbds Nt and N}. We have that
\begin{align*}
x \in A \iff& x \in [T] \ \& \ x \in A\\
        \iff& x \in [T] \ \& \ (\forall s)[x \not \in N(\spat{T},s) \ \vee \ \neg P^\ast(s)]\\
        \iff& x \in [T] \ \& \ (\forall s)[\neg Q^T(x,s) \ \vee \ \neg P^\ast(s)]
\end{align*}
for all $x \in \ca{N}$. Therefore if we take
\[
B(x) \iff (\forall s)[\neg Q^T(x,s) \ \vee \ \neg P^\ast(s)]
\]
we have that $B$ is $\Pi^0_1$ and $A = B \cap [T]$.

Conversely if $B \subseteq \ca{N}$ is $\Pi^0_1$ there exists a semirecursive $Q^\ast \subseteq \om$ such that
\[
\neg B(x) \iff (\exists t)[x \in N(\ca{N},t) \ \& \ Q^\ast(t)]
\]
for all $x \in \ca{N}$ and all $t \in \omega$. Hence we have that
\begin{align*}
x \in [T] \ \& \ B(x) \iff& x \in [T] \ \& \ (\forall t)[x \not \in N(\ca{N},t) \ \vee \ \neg Q^\ast(t)]\\
                      \iff& x \in [T] \ \& \ (\forall t)[\neg R^T(x,t) \ \vee \ \neg Q^\ast(t)]
\end{align*}
for all $x \in \ca{N}$ and all $t \in \omega$. This shows that the set $B\cap [T]$ is a $\Pi^0_1$ subset of $\spat{T} \times \omega$.

The assertion for the other pointclasses follows using induction on $n$.
\end{proof}

\begin{remark}
\label{remark optimal theorem pointclasses of Nt}\normalfont
Theorem \ref{theorem pointclasses of Nt} is not true for $\Gamma = \Sigma^0_1$. To see this let $T$ be \omseq, so that $[T] = \ca{N}$. According to the remarks preceding Lemma \ref{lemma transition between nhbds Nt and N} the body $[T]$  is not even an open subset of \spat{T}, while clearly $[T]$ is a $\Sigma^0_1$ subset of \ca{N}.
\end{remark}

\subsubsection*{A counterexample on the Cantor-Bendixson decomposition.} With the help of the spaces \spat{T} we can give a counterexample which is related to Theorem \ref{theorem kreisel scattered part}.

\begin{theorem}
\label{theorem closed sig set}
There is a closed set $B \subseteq \ca{N}$ which is \sig \ and
not \pii \ such that the scattered part $\scat{B}$ is not \pii.
\end{theorem}

\begin{proof}
Let $Q$ be a \sig \ subset of \om \ which is not \pii. Define
$$
u \in T \iff (\forall i < \lh(u))[u(i) \in Q].
$$
Clearly $T$ is a tree and moreover the set $[T]$ is a \sig \ perfect
subset of \ca{N}.

We consider the space \spat{T} and let $B$ be the set $\rho^T[\spat{T}]$, where $\rho^T: \spat{T} \inj \ca{N}$ is the embedding in Lemma \ref{lemma the space Nt is embedded in the Baire space}.

Clearly $B$ is a closed subset of \ca{N}.  Also $B$ is a \sig \ set since
\begin{align*}
\alpha \in B \iff& \big\{(\exists n)(\forall m<n)[\alpha(n) = 0 \ \& \
\alpha(m) >0 \ \& \ \alpha(m)-1 \in Q ] \big\}
\\
& \ \vee \big\{(\forall n)[\alpha(n) >0 \ \& \ \alpha(n)-1 \ \in Q]\big\}
\end{align*}
Finally we claim that
$$
n \in Q \iff (\forall \beta)[ \ [\beta(0) = n+1 \ \& \ (\forall m
\geq 1) \beta(m) = 0] \ \ \longrightarrow \ \beta \in \scat{B} \ ]
$$
for all \n. If we prove this equivalence we are done since if \scat{B} was
\pii \ then so would be $Q$.

To prove the preceding equivalence let us assume
first that $n$ is in $Q$ and that $\beta$ is as above. It follows that $(n,0,0,\dots) \in T$ and that $$\beta = (n+1,0,0,\dots) = \rho^T(n,0,0,\dots) \in B.$$
Also the set
\[
\set{\gamma \in \ca{N}}{\gamma(0) = \beta(0) \ \& \
\gamma(1) = \beta(1)} \cap B
\]
is easily the singleton $\{\beta\}$, and thus $\beta$ is an isolated point of $B$. In particular $\beta \in \scat{B}$. This settles the left-to-right-hand direction. For the converse direction clearly $\beta
= (n+1,0,0,\dots)$ is in $\scat{B}$, so in particular it
is a member of $B = \rho^T[\spat{T}]$. Hence $(n) \in T$ and so $n \in Q$.
\end{proof}

\subsubsection*{The space of trees}
In the coming sections we will need to view trees as points in a given space. Following \cite{kechris_classical_dst} (Exercise 4.32) we define the set
\[
\Tr = \set{T}{T \ \textrm{is a tree on \om}}.
\]
By identifying \omseq \ with \Seq \ and a tree with its characteristic function we can view \Tr \ as a closed subset of $2^\Seq$ with the product topology. In particular \Tr \ is a compact Polish space. It is not difficult to see that there exists a compatible metric $d$ on \Tr \ such that the space $(\Tr,d)$ admits a recursive presentation. We fix such a metric and a recursive presentation for the remaining of this article.

We consider the relations $R \subseteq \Tr$, $P \subseteq \om \times \Tr$, Q$ \subseteq \ca{N} \times \Tr$ defined by
\begin{align*}
R(T) &\iff T \ \textrm{is a recursive tree},\\
P(s,T) &\iff \dec{s} \in T,\\
Q(\alpha,T) &\iff \alpha \in [T],
\end{align*}
It is easy to verify that the preceding relations are $\Sigma^0_3$, $\Delta^0_1$ and $\Pi^0_1$ respectively.

\subsubsection*{Sums and products of trees} Finite products and sums of spaces of the form \spat{T} can be simulated by natural operations on trees. We just need to take care of the recursive aspect.

The \emph{sum} of the trees $T$ and $S$ is the tree $T \oplus S$ which satisfies
\begin{align*}
\cn{(n)}{u} \in T \oplus S \iff& \hspace*{7mm}[n = \langle 0,k \rangle \ \& \ \cn{(k)}{u} \in T]\\
                               & \ \vee \ [n = \langle 1,k\rangle  \ \& \ \cn{(k)}{u} \in S].
\end{align*}
for all $n, u$.

The product of trees is a bit more elaborate. We fix a recursive isomorphism
\[
\pairint{ \ }{\ }: \om \cup \{-1\} \times \om \cup \{-1\} \bij \om \cup \{-1\}
\]
with $\pairint{-1}{-1} = -1$, so that the numbers \pairint{n}{-1}, \pairint{-1}{n}, \pairint{m}{n} are naturals for all $m,n \in \om$. For finite sequences $u$ and $v$ in $\om$ of equal length $n$ we define
\[
\prodt{u}{v} = (\pairint{u(0)}{v(0)},\dots,\pairint{u(n-1)}{v(n-1)}).
\]
If $n = \lh(u) < \lh(v)$ we define
\begin{align*}
\prodt{u}{v} = (\pairint{u(0)}{v(0)},\dots,& \pairint{u(n-1)}{v(n-1)},\\ &\pairint{-1}{v(n)},\dots,\pairint{-1}{v(\lh(v)-1)}).
\end{align*}
(If the natural $n$ in the preceding cases is $0$ by $\prodt{u}{v}$ we mean the empty sequence.) Similarly we define $\prodt{u}{v}$ if $\lh(u) > \lh(v)$. It is clear that \prodt{u}{v} is a finite sequence of naturals and that the function $\prodt{}{}$ -viewed as function from $\Seq^2$ to \Seq \ through the codes of sequences- is a recursive injection. Moreover the inverse of \prodt{}{} viewed as a partial function is recursive as well.

We now define the \emph{product} \prodt{T}{S} of the trees $T$ and $S$ by
\begin{align*}
w \in \prodt{T}{S} \iff& w = \empt \ \vee \ [w = \prodt{u}{v} \ \& \ u \in T \ \& \ v \in S].
\end{align*}
In other words given $(u,v) \in T \times S$, we add $-1$'s to one of $u$, $v$ (if needed) so that the two sequences obtain the same length, and then we encode them in a single sequence using $\pairint{ \ }{\ }$, see Figure \ref{fig:producttrees}.
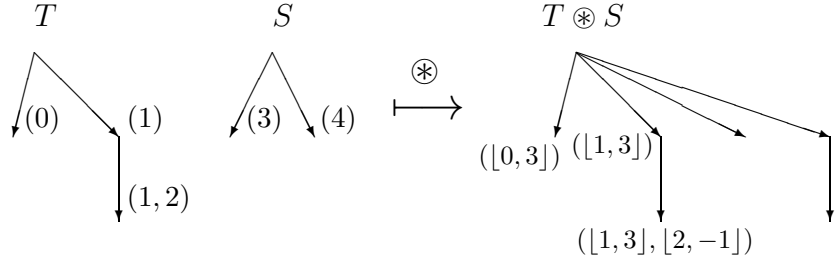
\begin{figure}[t]
\begin{picture}(400,90)(0,30)
\put(30,110){$T$}
\put(30,100){\vector(-1,-4){8}}
\put(26,72){\small{$(0)$}}
\put(30,100){\vector(1,-1){32}}
\put(65,72){\small{$(1)$}}
\put(62,68){\vector(0,-1){32}}
\put(65,42){\small{$(1,2)$}}
\put(120,110){$S$}
\put(120,100){\vector(-1,-2){16}}
\put(110,72){\small{$(3)$}}
\put(120,100){\vector(1,-2){16}}
\put(138,72){\small{$(4)$}}
\put(165,75){\Large{$\longmapsto$}}
\put(172,89){\large{$\prodt{}{}$}}
\put(222,110){$\prodt{T}{S}$}
\put(235,100){\vector(-1,-4){8}}
\put(198,58){\Small{$(\pairint{0}{3})$}}
\put(235,100){\vector(1,-1){32}}
\put(233,62){\Small{$(\pairint{1}{3})$}}
\put(267,68){\vector(0,-1){32}}
\put(235,26){\Small{$(\pairint{1}{3},\pairint{2}{-1})$}}
\put(235,100){\vector(2,-1){64}}
\put(235,100){\vector(3,-1){96}}
\put(331,68){\vector(0,-1){32}}
\end{picture}
\caption{The product of trees under \prodt{}{}.}
\label{fig:producttrees}
\end{figure}
It is clear that \prodt{T}{S} is a tree and if $w = \prodt{u}{v}$ is an initial segment of $w' = \prodt{u'}{v'}$ then $u$ and $v$ are initial segments (not necessarily proper) of $u'$ and $v'$ respectively. The converse fails, one can see this by taking the sequences $u = (1)$, $u' = (1,2)$, $v = v' = (4,5)$.

The operations
\begin{align*}
&\oplus: \Tr^2 \to \Tr: (T,S) \mapsto T \oplus S\\
&\prodt{}{}: \Tr^2 \to \Tr: (T,S) \mapsto \prodt{T}{S}
\end{align*}
are recursive, so that if $T$ and $S$ are recursive trees then $T \oplus S$ and $\prodt{T}{S}$ are recursive trees as well.

For infinite sequences $x,y: \om \to \om \cup \{-1\}$ we define
\[
\prodt{x}{y} = (\pairint{x(0)}{y(0)},\dots,\pairint{x(n)}{y(n)},\dots).
\]
Since the function $\pairint{\ }{\ }$ is bijective every $z: \om \to \om \cup \{-1\}$ has the form \prodt{x}{y} for some $x,y: \om \to \om \cup \{-1\}$.

\begin{lemma}
\label{lemma sum and product of trees}
For all recursive trees $T$ and $S$ we have that
\begin{align*}
\spat{T} \oplus \spat{S} &\simeq_{{\rm rec}} \spat{T \oplus S}\\
\spat{T} \times \spat{S} &\simeq_{{\rm rec}} \spat{\prodt{T}{S}}
\end{align*}
where $\simeq_{{\rm rec}}$ stands for the existence of a recursive isomorphism between the involved spaces.
\end{lemma}

\begin{proof}
The assertion about $\oplus$ is clear so we deal with \prodt{}{}. For the needs of this proof we view a typical member of $\spat{T}$ as an infinite sequence by identifying every member of $T$ with a sequence which is eventually $-1$. We define the function
\[
f:\spat{T} \times \spat{S} \to \spat{\prodt{T}{S}}: f(x,y) = \prodt{x}{y}.
\]
It is not hard to check that the function $f$ is well-defined, \ie $f(x,y) \in \spat{\prodt{T}{S}}$ for $x \in \spat{T}$ and $y \in \spat{S}$, and that $f$ is a recursive injection. Moreover the inverse $f^{-1}$ is recursive as well.

It remains to show that $f$ is surjective. Suppose that $z = \prodt{x}{y} = f(x,y)$ is member of $\spat{\prodt{T}{S}}$. We will show that $x \in \spat{T}$ and $y \in \spat{S}$. Since $z$ belongs to $\spat{\prodt{T}{S}}$ there are only two possibilities: (A) $z$ is eventually $-1$ or (B) $z(n) \geq 0$ for all \n.

In case (A) it follows that $x$ and $y$ are eventually $-1$ as well. Let $n$ be the least such that $z(n)=-1$ so that $z \upharpoonright n \in \prodt{T}{S}$. Clearly $$z \upharpoonright n = \prodt{x \upharpoonright n}{y \upharpoonright n}$$ and $x(m)=y(m)=-1$ for all $m \geq n$. We consider the largest initial segments $u$ and $v$ of $x \upharpoonright n$ and $y \upharpoonright n$ respectively such that $x(u(i)), y(v(j)) \geq 0$ for all $i < \lh(u)$ and all $j < \lh(v)$. We notice that at least one of $u$ and $v$ has length $n$, for otherwise $n$ would not be the least natural $k$ with $z(k) = -1$. Therefore $z \upharpoonright n = \prodt{u}{v}$ and since $z \upharpoonright n \in \prodt{T}{S}$ we have from the injectiveness of \prodt{}{} that $u \in T$ and $v \in S$. From the choice of $u$ and $v$ it is also clear that $x(m) = y(k) = -1$ for all $m \geq \lh(u)$ and all $k \geq \lh(v)$. Hence $x \in \spat{T}$ and $y \in \spat{S}$.

In case (B) $z$ is a member of $[\prodt{T}{S}]$ and exactly one of the following subcases applies: (B1) $x(n), y(n) \geq 0$ for all $n$, (B2) $x$ is eventually $-1$ and $y(n) \geq 0$ for all $n$, (B3) $x(n) \geq 0$ for all $n$ and $y$ is eventually $-1$. In (B1) we have that $(x,y) \in [T] \times [S]$, in (B2) there exists an $n$ such that $x \upharpoonright n \in T$ and $x(m) = -1$ for all $m \geq n$ and $y \in [S]$, and in (B3) $x \in [T]$ and there exists some $n$ such that $y \upharpoonright n \in S$ and $y(m) = -1$ for all $m \geq n$. In any case we have that $(x,y) \in \spat{T} \times \spat{S}$.
\end{proof}

\subsection{Elementary facts about the classes of \del-isomorphism} We conclude this section with a preliminary investigation on \del-injections in recursively presented metric spaces and their connection with the spaces \spat{T}.

\begin{definition}
\label{definition of dleq}\normalfont
We introduce the relations $\dleq$, $\dequal$ and $\dless$,
\begin{align*}
\ca{X} \dleq \ca{Y} \iff& \textrm{there exists a \del-injection} \ f: \ca{X} \inj \ca{Y}\\
\ca{X} \dequal \ca{Y} \iff& \textrm{there exists a \del-bijection} \ f: \ca{X} \bij \ca{Y}\\
\ca{X} \dless \ca{Y} \iff& \ca{X} \dleq \ca{Y} \ \& \ \ca{X} \not \dleq \ca{Y}
\end{align*}
where $\ca{X}$ and $\ca{Y}$ are recursively presented metric spaces.
\end{definition}
The usual proof of the Schr\"{o}der-Bernstein Theorem is effective (see the proof of 3E.7 in \cite{yiannis_dst}) so that
\[
\ca{X} \dequal \ca{Y} \ \ \eq \ \ \ca{X} \dleq \ca{Y} \ \& \
\ca{Y} \dleq \ca{X}.
\]
The class of \del-isomorphism of recursively presented metric spaces which are countable sets is characterized easily. First let us recall the following.

\begin{lemma}[\cf \cite{gregoriades_moschovakis_notes_on_edst}]
\label{lemma enumeration by a del sequence}
Every countably infinite \del \ subset of a recursively presented metric space can be enumerated by a \del-recursive injective sequence.
\end{lemma}

\begin{theorem}
\label{theorem countable space}
Every infinite countable recursively presented metric space is \del-isomorphic to \om.
\end{theorem}

\begin{proof}
Suppose that \ca{X} is an infinite countable recursively presented metric
space. Since \ca{X} is trivially a \sig \ subset of \ca{X} from
the Effective Perfect Set Theorem \ref{theorem alltogether}-(\ref{theorem 4F.1}) \ca{X} consists of \del \
points. From Lemma \ref{lemma enumeration by a del sequence} there exists
a \del-recursive injective sequence $(x_n)_{\n}$ such that $\ca{X} = \set{x_n}{\n}$. It is clear then that the function $n \mapsto x_n$ is a \del-bijection.
\end{proof}

\begin{theorem}
\label{theorem all between om and N}

For every infinite recursively presented metric space \ca{X} we have that
\[
\om \dleq \ca{X} \dleq \ca{N}.
\]
\end{theorem}

\begin{proof}
Suppose that \ca{X} is an infinite recursively presented metric space. Fix a
\del- sequence $(x_n)_{\n}$ in \ca{X} (for example the recursive
presentation with respect to some compatible metric). From Lemma \ref{lemma enumeration by a del sequence} we may assume that we
have deleted repetitions so that the function $n \mapsto x_n$ is
injective. The latter witnesses that $\om \dleq \ca{X}$.

Regarding the second inequality recall from Theorem \ref{theorem alltogether}-(\ref{theorem 3E.6}) that there exists a
recursive surjection $\pi: \ca{N} \to \ca{X}$ and a $\Pi^0_1$ set
$A \subseteq \ca{N}$ such that $\pi \upharpoonright A$ is
injective and $\pi[A] = \ca{X}$. From Lemma \ref{lemma the inverse function is del and same hyperdegree} it follows that
the inverse function $\pi^{-1}: \ca{X} \to A$ is \del-recursive and hence $\ca{X} \dleq
\ca{N}$.
\end{proof}

A key property of the spaces \spat{T} is that they characterize all recursively presented metric spaces up to \del-isomorphism.

\begin{theorem}

\label{theorem Nt del isomorphism}

Every recursively presented metric space \ca{X} is \del-isomorphic with a space
of the form \spat{T} for a recursive $T$.\smallskip

In fact one can choose a \del-isomorphism $f: \spat{T} \bij \ca{X}$ such that $f$ agrees with a recursive function $\pi: \ca{N} \to \ca{X}$ on $[T]$.
\end{theorem}

\begin{proof}
If \ca{X} is infinite countable then from Theorem \ref{theorem
countable space} \ca{X} is \del-isomorphic to \om, which is easily
recursively isomorphic to a space of the form \spat{T}.

So assume that \ca{X} is uncountable. Pick an infinite
\del-recursive sequence $(x_n)_{\n}$ in \ca{X}. The set $C = \ca{X} \setminus
\set{x_n}{\n}$ is \del. From Theorem \ref{theorem alltogether}-(\ref{theorem 4D.9}) there exists a $\Pi^0_1$ set $A
\subseteq \ca{N}$ and a recursive function $\pi: \ca{N} \to
\ca{X}$ which is injective on $A$ and $\pi[A] = C$. From Lemma \ref{lemma pi-zero-one sets recursive tree} there
exists a recursive tree $T$ on \om \ such that $A = [T]$. Clearly $T$ is infinite for otherwise $C$ would be empty
which is not the case since \ca{X} is uncountable. Give $T$ a
\del-recursive injective enumeration $(t_n)_{\n}$ and define the
function $f: \spat{T} \to \ca{X}$ as follows
\[
f(z) = \begin{cases} x_n, & \ \textrm{if} \ z = t_n\\
                    \pi(z), & \ \textrm{if} \ z \in [T].
\end{cases}
\]
It is clear that the function $f$ is \del-recursive. Since $C \cap
\set{x_n}{\n} = \emptyset$ and the functions $n \mapsto x_n$ and
$\pi \upharpoonright [T]$ are injective we have that $f$ is
injective as well. Moreover since $C \cup \set{x_n}{\n} = \ca{X}$
and $\pi[[T]] = C$ it is clear that $f$ is surjective. Therefore
the function $f$ is a \del-bijection between \spat{T} and \ca{X}.
\end{proof}\smallskip

The following simple lemma allows us to pass from \del-injections on $[T]$ to $[S]$ to \del-injections between the corresponding spaces. We will see that the converse is not true in general, however it is true in the category of spaces that we will focus on.

\begin{lemma}
\label{lemma bodyT less than bodyS implies Nt less than Ns}
Suppose that $T$ and $S$ are recursive trees with non-empty body and assume that there exists a \del-recursive
function $\pi: \ca{N} \to \ca{N}$ which is injective on $[T]$ and $\pi[[T]] \subseteq [S]$. Then there exists a \del-injection $f: \spat{T} \inj \spat{S}$.
\end{lemma}

\begin{proof}
Consider the sets
\[
I_T = \set{t}{\dec{t} \in T} \quad I_S = \set{s}{\dec{s} \in S}.
\]
Clearly $I_T$ and $I_S$ are recursive and infinite. There exist a recursive
function $\tau: \om \to \om$ which is injective on $I_T$ and $\tau[I_T] = I_S$. Define
\[
f: \spat{T} \to \spat{S}: f(x) = \begin{cases}
                                     \dec{\tau(t)}, & \ \textrm{if} \ x=\dec{t} \in T\\
                                     \pi(x), & \ \textrm{if} \ x \in [T].
                                     \end{cases}
\]
Then the function $f$ is a \del-injection.
\end{proof}

\section{\sc Kleene spaces}

\label{section Kleene spaces}

We introduce a category of uncountable spaces of the form \spat{T}, the \emph{Kleene spaces}, which are all different from the Baire space under \dequal. We will show that their class is closed downwards under \dleq \ (in uncountable spaces), they form antichains and have no maximal or minimal elements under \dleq. We will also give analogous results in the language of recursive \emph{pseudo-well-orderings} and investigate some of the properties of hyperdegrees in these spaces.

\subsection{Definition and basic properties} Let us say a few words about the idea behind Kleene spaces. Suppose that $T$ is a recursive tree whose body is non-empty and has no hyperarithmetical members, such a tree exists from Kleene's Theorem \ref{theorem Kleene there exists a Kleene tree}. It is clear that there is no \del \ function on \ca{N} to $[T]$, let alone a \del-bijection. Moreover from the Effective Perfect Set Theorem \ref{theorem alltogether}-(\ref{theorem 4F.1}) the set $[T]$ is uncountable. It looks as if $[T]$ with the metric of the Baire space were a good candidate for an uncountable space which is not \del-isomorphic to \ca{N}. However the absence of hyperarithmetical points inside $[T]$ is also the reason that this metric space does not admit a recursive presentation. The idea is to add to $[T]$ a recursive presentation, and do so in such a way, that still no \del-injection on \ca{N} inside the new space can exist. This is the motivation behind the definition of the spaces \spat{T}.

\begin{definition}
\label{definition of a Kleene space}
\normalfont
A recursive tree $T$ on \om \ is a \emph{Kleene tree} if the body $[T]$ is non-empty and does not contain
\del-members. A space of the form \spat{T} is a \emph{Kleene space} if $T$ is a Kleene tree.
\end{definition}
Notice that a Kleene space is an uncountable set. Also it is clear that the sum and product of two Kleene trees is a Kleene tree. It follows from Lemma \ref{lemma sum and product of trees} that the finite sum or product of Kleene spaces is recursively isomorphic to a Kleene space.

\begin{theorem}
\label{theorem counterexample to del with Baire (Kleene space)}
For every Kleene tree $T$ we have that
\[
\del \cap \spat{T} = \scat{\spat{T}} = \spat{T}_{\rm iso} = T,
\]
where $\spat{T}_{\rm iso}$ and \scat{\spat{T}} is the set of all isolated points and the scattered part of \spat{T} respectively.\smallskip

In particular the set $\del \cap \spat{T}$ is semirecursive and so there does not exist a \del-injection $f: \ca{N} \inj
\spat{T}$. It follows that no Kleene space is \del-isomorphic to the Baire space.
\end{theorem}

\begin{proof}
The inclusions $T \subseteq \spat{T}_{\rm iso} \subseteq \scat{\spat{T}} \subseteq \del \cap \spat{T}$ are clear from Remark \ref{remark scattered part consists of del points} and Theorem \ref{theorem properties of Nt} -and hold for every recursive tree. Suppose that $x \in \spat{T}$ is in \del. Then $x$ cannot be a member of $[T]$, for otherwise we would have from Theorem \ref{theorem points of Nt} that $x$ is a \del \ point of \ca{N} belonging to $[T]$, contradicting that $T$ is a Kleene tree. Hence $x \in T$ and the equalities are proved.

If there were a \del-injection $f: \ca{N} \inj \spat{T}$ then (using the injectiveness of $f$ and Lemma \ref{lemma the inverse function is del and same hyperdegree}) we would have that
\[
\alpha \in \del \iff f(\alpha) \in T
\]
for all $\alpha \in \ca{N}$. Since $T$ is a $\Sigma^0_1$ subset of \spat{T} it would follow that the set of all \del \ points of \ca{N} would be \del, contradicting the Lower Classification of \del \ \cf 4D.16 \cite{yiannis_dst}.
\end{proof}

It is perhaps surprising that there are recursively presented metric spaces other than the countable ones, whose set of \del-points is \del \ and
in fact semirecursive, \ie the Lower Classification of \del \ is violated in the worst possible way.\footnote{The set of all \del-points of an uncountable space cannot be closed since it is countable dense.} We prove that these spaces do not go far from Kleene spaces.

\begin{theorem}
\label{theorem characterization the space is bad}
An uncountable recursively presented metric space \ca{X} is \del-isomorphic
to a Kleene space if and only if $\del \cap \ca{X}$ is \del.
\end{theorem}

\begin{proof}
The right-to-left-hand direction has been proved in Theorem
\ref{theorem counterexample to del with Baire (Kleene space)}. For the inverse direction let \ca{X}
be such that the set $\del \cap \ca{X}$ of \del-points of \ca{X} is in \del. Set $D = \del \cap \ca{X}$ and $C = \ca{X}
\setminus D$. Notice that $D$ is countably infinite and $C$ is uncountable because \ca{X} is uncountable. Since $D$ is a
countable \del \ set there exists an injective \del-recursive
sequence $(x_n)_{\n}$ which enumerates $D$.

Now we deal with $C$. Using Lemma \ref{lemma pi-zero-one sets recursive tree} and Theorem \ref{theorem alltogether}-(\ref{theorem 4D.9}) we obtain a recursive tree $T$ on \om \ and a recursive function $\pi : \ca{N} \to \ca{X}$ which is
injective on $[T]$ and $\pi[[T]] = C$. From its definition the set
$C$ has no \del \ members; hence $[T]$ has no \del \ members as well. Since $C \neq \emptyset$ we have that $T$ is a
Kleene tree. We define
\[
f: \spat{T} \to \ca{X}: f(y) = \begin{cases}
                                  x_n, & \ \textrm{if} \ y = \dec{n} \in T\\
                                  \pi(y), & \ \textrm{if} \ y \in [T].
                                 \end{cases}
\]
Since $T$ is a semirecursive subset of \spat{T} it is clear that the function $f$ is \del-recursive. Moreover $f$ is bijective
from the properties of $\pi$ and $(x_n)_{\n}$.
\end{proof}

One interesting consequence of the preceding theorem is that the class of Kleene spaces is closed from below under \dleq \ in uncountable spaces.

\begin{corollary}
\label{corollary Kleene spaces are closed from below}
If \ca{Y} is a Kleene space, \ca{X} is uncountable and $\ca{X} \dleq \ca{Y}$ then $\ca{X}$ is \del-isomorphic to a Kleene space.
\end{corollary}

\begin{proof}
Suppose that $f: \ca{X} \inj \ca{Y}$ is a \del \ injection. We have that
\[
x \in \del \iff f(x) \in \del
\]
for all $x \in \ca{X}$. Since \ca{Y} is a Kleene space the set of its \del \ points is semirecursive. It follows from
the preceding equivalence that the set of \del \ points of $\ca{X}$ is \del \ and since \ca{X} is uncountable from Theorem \ref{theorem
characterization the space is bad} we have that \ca{X} is \del \ isomorphic to a Kleene space.
\end{proof}

We now show that the converse of Lemma \ref{lemma bodyT less than bodyS implies Nt less than Ns} is true if the ``smaller" tree is a
Kleene tree, \ie one can pass from \del \ injections $\pi: \spat{T} \inj \spat{S}$ to \del-injections on $[T]$ to $[S]$, where $T$ is a Kleene tree.

\begin{lemma}
\label{lemma converse of lemma bodyT less than bodyS in Kleene trees}
Suppose that $T$ and $S$ are recursive trees and $$\pi: \spat{T} \inj
\spat{S}$$ is a \del-injection. If $T$ is a Kleene tree then $\pi[[T]] \subseteq [S]$ and so there exists a \del-recursive function $$f: \ca{N} \to \ca{N}$$ which is injective on $[T]$ and carries $[T]$ inside $[S]$.
\end{lemma}

\begin{proof}
Let $x \in \spat{T}$ be such that $\pi(x) \in S$. Then $\pi(x)$ is a \del \ point of \spat{S}. Since $\pi$ is
\del-recursive and injective we have that $x \in \del(\pi(x))$ and so $x \in \del$. Since $T$ is a Kleene tree it
follows that $x \in T$.

Regarding the function $f$ we just extend the function $\pi \upharpoonright [T]$ to the Baire space by assigning to points outside of $[T]$ a fixed recursive point.
\end{proof}

Given a recursive infinite tree $T$ it is fairly simple to construct a \sig \ subset of $T$ which is not \pii. Hence
using Theorem \ref{theorem Nt del isomorphism} we have that $\sig \upharpoonright \ca{X} \neq \pii \upharpoonright
\ca{X}$ for all recursively presented metric spaces. Although the next result is not directly related to Kleene spaces the idea of the proof of Theorem \ref{theorem characterization the space is bad} enables us to find a $\sig \setminus \pii$ subset of \spat{T} which is in fact a subset of the body of $T$.

\begin{theorem}
\label{theorem sig not in pii subset of bodyT}

For every recursive tree $T$ with uncountable body there exists a set $P \subseteq [T]$ which is a $\sig
\setminus \pii$ subset of \ca{N}.\smallskip

It follows that every uncountable \del \ subset of a recursively presented metric space contains a $\sig \setminus \pii$ set.
\end{theorem}

\begin{proof}
The second assertion follows from the fact that every \del \ set is the injective recursive image of a $\Pi^0_1$ subset of the Baire space (\cf Theorem \ref{theorem alltogether}- (\ref{theorem 4D.2 and theorem 4D.15})), so we prove the first assertion.

Let $D$ be the set $\del \cap [T]$. Clearly $D$ is a \pii \ subset of \ca{N}. If $D$ is not \del \ then we can take $P = [T] \setminus D$, so assume that $D$ is \del. From Theorem \ref{theorem alltogether}- (\ref{theorem 4D.2 and theorem 4D.15}) there exists a $\Pi^0_1$ set $A \subseteq \ca{N}$ and a recursive function $\pi: \ca{N} \to \ca{N}$ which is injective on $A$ and $\pi[A] = [T] \setminus D$. We pick a recursive tree $S$ such that $[S] = A$.

Since $[T]$ is uncountable the set $[T] \setminus D$ is non-empty and so $[S]$ is non-empty as well. Moreover $[S]$ does not contain \del \ members for otherwise the recursive function $\pi$ would produce a \del \ point of $[T]$ outside of $D$, contradicting its definition. Hence $S$ is a Kleene tree.

Let $\alpha_L$ be the leftmost infinite branch of $S$. We verify
that the singleton $\{\alpha_L\}$ is a \pii \ subset
of $\ca{N}$. Indeed
\begin{align*}
\alpha \in \{\alpha_L\} \iff& \alpha \in [S] \ \& \ (\forall \beta)\big \{[\beta \in [S] \ \& \ \beta
\neq \alpha ] \longrightarrow\\
                            &\hspace*{27mm} (\exists n)[\alpha \upharpoonright n = \beta \upharpoonright n \ \& \
\alpha(n) < \beta(n) \big \}
\end{align*}
for all $\alpha \in \ca{N}$. So $\{\alpha_L\}$ is a \pii \ subset of \ca{N}.

The point $\pi(\alpha_L) \in [T]$ is not in \del \ since $\pi$ is recursive and $\alpha_L$ is not a \del \ point. Using that $\pi$ is injective on $[S]$ we have that $\alpha \in \del(\pi(\alpha))$ for all $\alpha \in [S]$, since
\[
\beta = \alpha \iff \beta \in [S] \ \& \ \pi(\beta) = \pi(\alpha)
\]
for all $\beta \in [S]$. From this it follows that $\{\pi(\alpha_L)\}$ is a \pii \ subset of \ca{N} because
\[
x \in \{\pi(\alpha_L)\} \iff x \in [T] \ \& \ (\exists \beta \in \del(x))[\beta=\alpha_L \ \& \ \pi(\beta) =x],
\]
for all $x \in \ca{N}$. Finally $\{\pi(\alpha_L)\}$ is not a \sig \ subset of \ca{N} for otherwise $\pi(\alpha_L)$ would be a \del \ point of \ca{N}. Therefore we may take $P = [T] \setminus \{\alpha_L\}$ and we are done.
\end{proof}

\subsection{Expanding the toolbox} We give a brief review of the results of recursion theory that we will need in the sequel. Recall Spector's \W, which is defined in the Introduction.

\begin{theorem}[Kleene Basis Theorem \cf \cite{kleene_arithmetical_predicates_and_function_quantifiers}]
\label{theorem Kleene's Basis Theorem}
Every non-empty \sig \ subset of a recursively presented metric space contains a member which is recursive in Spector's \W.
\end{theorem}

\begin{theorem}[Spector \cf \cite{spector_recursive_well-orderings}]
\label{theorem spector two hyperdegrees}
For every \pii \ set $P \subseteq \om$ either $P \in \del$ or $P \heq \W$.
\end{theorem}

\begin{lemma}[\cf \cite{friedman_harvey_borel_sets_and_hyperdegrees}, \cite{jockush_soare_encodability_of_kleenes_O} and \cite{sacks_higher_recursion_theory} Ex. 1.8 Ch. III]
\label{lemma leftmost branch is del or has hyperdegree W}
The leftmost infinite branch of a recursive tree is either \del \ or has the same hyperdegree as $\W$.
\end{lemma}

To see the latter we consider the set of all finite sequences in the given recursive tree which are $\leq$-lexicographically than its leftmost infinite branch and apply Spector's Theorem \ref{theorem spector two hyperdegrees}.

\subsubsection*{The Gandy Basis Theorem}

Kleene's Theorem \ref{theorem Kleene there exists a Kleene tree} admits a considerable generalization due to Gandy, which in turn
implies an important result on the theory of hyperdegrees.

\begin{theorem}[Gandy \cf \cite{gandy_on_a_problem_of_Kleenes}]
\label{theorem Gandy extension of Kleene}
There exists a $\Pi^0_1$ set $K \subseteq \ca{N} \times \ca{N}$ such that for all $\alpha \in \ca{N}$ the
$\alpha$-section $K_\alpha$ is non-empty and for all $\beta \in K_\alpha$ we have that $\beta \not \in \del(\alpha)$.
\end{theorem}

\begin{proof}
Consider the \sig \ set
\[
P(\alpha,\beta) \iff \beta \not \in \del(\alpha)
\]
and let $F \subseteq \ca{N}^3$ be $\Pi^0_1$ such that
\[
P(\alpha,\beta) \iff (\exists \gamma)F(\alpha,\beta,\gamma).
\]
Define
\[
K(\alpha,\beta^ \ast) \iff F(\alpha,(\beta^\ast)_0,(\beta^ \ast)_1).
\]
Clearly $K$ is a $\Pi^0_1$ set. Moreover for all $\alpha \in \ca{N}$ there exists $\beta \not \in \del(\alpha)$ and
hence there also exists $\gamma$ such that $F(\alpha,\beta,\gamma)$. Hence $K_\alpha \neq \emptyset$.

Finally if $\beta^\ast$ is in $K_\alpha$ we have in particular that $P(\alpha,(\beta^\ast)_0)$, \ie
$(\beta^\ast)_0 \not \in \del(\alpha)$. Since $(\beta^\ast)_0$ is recursive in $\beta^\ast$ it follows that
$\beta^\ast$ is not in $\del(\alpha)$ as well.
\end{proof}

Theorem \ref{theorem Gandy extension of Kleene} has the following interesting application.

\begin{theorem}[Gandy Basis Theorem, \cf \cite{gandy_on_a_problem_of_Kleenes}]
\label{theorem Gandy basis}
If \ca{X} is a recursively presented metric space and $A$ is a non-empty \sig \ subset of \ca{X} then there exists some $x \in A$
such that $x \hless \W$.
\end{theorem}

\begin{proof}[\textit{Proof} \tu{(Well-known).}]
We assume first that $\ca{X} = \ca{N}$. Let $K$ be as in the statement of Theorem \ref{theorem Gandy extension of
Kleene}, \ie $K$ is a $\Pi^0_1$ subset of $\ca{N} \times \ca{N}$, $K_\alpha \neq \emptyset$ and $K_\alpha \cap
\del(\alpha) = \emptyset$ for all $\alpha \in \ca{N}$.

Consider the set
\[
B(\alpha,\beta) \iff  \alpha \in A \ \& \ K(\alpha,\beta).
\]
Clearly $B$ is a non-empty \sig \ set and so from Kleene's Basis Theorem \ref{theorem Kleene's Basis Theorem} there exists a pair $(\alpha,\beta)$ in $B$ which is recursive in Spector's $\W$. In particular $\alpha \hleq
\W$. If $\W$ were in $\del(\alpha)$
then $\beta$ (being recursive in $\W$) would also be in $\del(\alpha)$, contradicting that $\beta \in K_\alpha$. Hence
$\alpha \hless \W$.

In the general case we apply Theorem \ref{theorem alltogether}-(\ref{theorem 3E.6}), \ie we consider a recursive surjection $\pi: \ca{N} \surj
\ca{X}$ and a $\Pi^0_1$ set $C \subseteq \ca{N}$ such that $\pi$ is injective on $C$ and $\pi[C] = \ca{X}$. The set
$\pi^{-1}[A] \cap C$ is a non-empty \sig \ subset of $\ca{N}$ and so there exists some $\alpha \in \pi^{-1}[A] \cap C$
such that $\alpha \hless \W$. The point $\pi(\alpha) \in A$ is clearly recursive in $\alpha$ since $\pi$ is a recursive
function.

Moreover using that $\pi$ is injective on $C$ we have that
\[
\beta = \alpha \iff \beta \in C \ \& \ \pi(\beta) = \pi(\alpha)
\]
which shows that the singleton $\{\alpha\}$ is $\sig(\pi(\alpha))$. Hence $\alpha \in \del(\pi(\alpha))$ and so
\[
\pi(\alpha) \heq \alpha \hless \W.
\]
This finishes the proof.
\end{proof}

\subsubsection*{Kreisel compactness} One of the most important tools for producing incomparable spaces is the following result of Kreisel.

\begin{theorem}[Kreisel compactness, \cf \cite{kreisel_compactness} and \cite{sacks_higher_recursion_theory} Theorem 5.1]
\label{theorem Kreisel compactness}
Suppose that \ca{X} is a recursively presented metric space and that $P \subseteq \om$ is \pii. Assume moreover that $D \subseteq \om \times \ca{X}$ is computed by a \sig \ set
on $P \times \ca{X}$, \ie there is a \sig \ set $R \subseteq \om \times \ca{X}$ such that for all $n \in P$
and all $x \in \ca{X}$ we have that $D(n,x) \iff R(n,x)$.\smallskip

If $D$ satisfies
\[
\cap_{n \in H} D_n \neq \emptyset
\]
for all \del \ sets $H \subseteq P$ then
\[
\cap_{n \in P} D_n \neq \emptyset.
\]
\end{theorem}

\begin{proof}[\textit{Proof} \tu{(Well-known).}]
Suppose towards a contradiction that $\cap_{n \in P} D_n = \emptyset$. This means that for all $x \in \ca{X}$
there is some $n \in P$ such that $x \not \in D_n$. Consider a \sig \ set $R \subseteq \om \times \ca{X}$,
which computes $D$ on $P \times \ca{X}$. Define $Q \subseteq \ca{X} \times \om$ as follows
\[
Q(x,n) \iff n \in P \ \& \ \neg D(n,x).
\]
It is clear that
\[
Q(x,n) \iff n \in P \ \& \ \neg R(n,x).
\]
Hence $Q$ is in \pii \ and $\exists^\om Q = \ca{X}$. We uniformize $Q$ by some set $Q^\ast$ in \pii. Then $Q^\ast$ is
the graph of a total function $f: \ca{X} \to \om$. Since $Q^\ast \subseteq Q$ we have that $f(x) \in P$ and
$x \not \in D_{f(x)}$ for all $x \in \ca{X}$.

The graph of $f$ is also a \sig \ set since
\[
n = f(x) \iff (\forall m)[Q^\ast(x,m) \ \longrightarrow m=n].
\]
Hence $f$ is \del-recursive. It follows that the set $f[\ca{X}]$ is a \sig \ subset of $P$. By separation we get a
\del \ set $H$ such that $f[\ca{X}] \subseteq H \subseteq P$. From our hypothesis there is some $x
\in \cap_{n \in H} D_n$. Since $H \supseteq f[\ca{X}]$ we have in particular that $x \in D_{f(x)}$,
contradicting the key property of $f$.
\end{proof}

\subsection{Chains and antichains in Kleene spaces under \dleq} Now we can proceed to results about \dleq \ in Kleene spaces.

\begin{theorem}

\label{theorem for every kleene space there is another one from below}

For every Kleene tree $T$, there is some initial segment $u$ of the leftmost
infinite branch of $T$ such that
\[
\spat{T_u} \dless \spat{T}.
\]
It follows that every Kleene space \ca{X} is the top of an infinite sequence
of Kleene spaces which is strictly decreasing under \dleq,
\[
\ca{X} \dgreat \ca{X}_1 \dgreat \ca{X}_2 \dgreat \dots.
\]
\end{theorem}

\begin{proof}
Suppose that $T$ is a Kleene tree and that $\alpha_L$ is its leftmost infinite branch. From Lemma \ref{lemma leftmost branch is del or has hyperdegree W} $\alpha_L$ has the same hyperdegree as $\W$. We apply the Gandy Basis
Theorem \ref{theorem Gandy basis} to get some $\gamma \in [T]$ such that $\W \not \in \del(\gamma)$.

The key remark is that there is some $u \sqsubseteq \alpha_L$ such that $\del(\gamma) \cap
[T_u] = \emptyset$. To see this we need the following.

\emph{Claim.} If $\del(\gamma) \cap [T_u] \neq \emptyset$ for all $u \sqsubseteq \alpha_L$ then
\begin{align}
\label{equation theorem for every kleene space there is another one from below} \hspace*{10mm} \alpha \in \{\alpha_L\}
\iff \alpha \in
[T] \ \& \ (\forall \beta \in \del(\gamma))[\beta \in [T] \longrightarrow \alpha \leq_{\rm lex} \beta]
\end{align}
for all $\alpha \in \ca{N}$, where $\alpha \leq_{\rm lex} \beta$ means that either $\alpha = \beta$ or for the unique
$n$ for which $\alpha \upharpoonright n = \beta \upharpoonright n$ and $\alpha(n) \neq \beta(n)$ we have that $\alpha(n)
< \beta(n)$.

\emph{Proof of the Claim.} The left-to-right-hand of the equivalence is clear. For the inverse direction suppose towards a contradiction that $\alpha \upharpoonright n = \alpha_L \upharpoonright n$ and $\alpha_L(n) \neq \alpha(n)$ for some \n.
Since $\alpha \in [T]$ and
$\alpha_L$ is the leftmost infinite branch of $T$ we have that $\alpha_L(n) < \alpha(n)$. We take $u =
(\alpha_L(0),\dots,\alpha_L(n))$. From our hypothesis there exists some $\beta$ in $\del(\gamma)$ which belongs to $[T_u]$. The right-hand-side of (\ref{equation theorem for every kleene space there is another one from below}) implies that $\alpha \leq_{\rm lex} \beta$. But on the other hand $\beta \upharpoonright n = \alpha_L \upharpoonright n = \alpha \upharpoonright n$ and $\beta(n) = \alpha_L(n) < \alpha(n)$, \ie $\alpha \not \leq_{\rm lex} \beta$, see Figure \ref{fig:for every kleene space there is another one from below}. The latter yields a contradiction and the Claim has been proved.
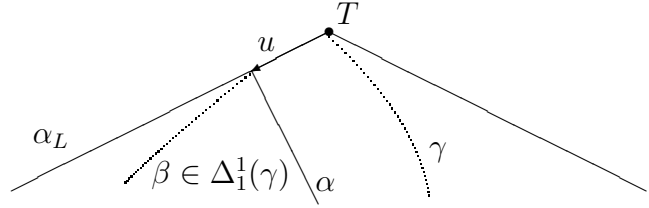
\begin{figure}[t]
\begin{picture}(300,85)(0,0)
\put(140,70){\tiny{$\bullet$}}
\put(142,72){\line(-2,-1){120}}
\put(142,72){\line(2,-1){120}}
\put(145,75){$T$}
\put(142,72){\vector(-2,-1){30}}
\put(115,65){$u$}
\put(113,57){\line(1,-2){25}}
\qbezier[40](142,70)(175,40)(180,10)
\qbezier[40](112,56)(80,30)(65,15)
\put(137,12){$\alpha$}
\put(180,25){$\gamma$}
\put(75,15){$\beta \in \del(\gamma)$}
\put(30,30){$\alpha_L$}
\end{picture}
\caption{Converging to $\alpha_L$ through $\del(\gamma)$.}
\label{fig:for every kleene space there is another one from below}
\end{figure}

If it were $\del(\gamma) \cap [T_u] \neq \emptyset$ for all $u \sqsubseteq \alpha_L$ from the equivalence (\ref{equation theorem for
every kleene space there is another one from below}) of the preceding Claim it would follow that the
singleton $\{\alpha_L\}$ is $\sig(\gamma)$ and so from 3E.16 in \cite{yiannis_dst} $\alpha_L$ would be a
$\del(\gamma)$ point. Hence we would have $\W \in \del(\gamma)$, which contradicts the choice of $\gamma$.

Therefore there exists some $u \sqsubseteq \alpha_L$ with $\del(\gamma) \cap [T_u] = \emptyset$. Clearly $\spat{T_u} \dleq \spat{T}$. If it were $\spat{T} \dleq
\spat{T_u}$ then there would be a $\del(\gamma)$ point inside $[T_u]$, which is a contradiction from the
choice of $u$.
\end{proof}

The idea of the preceding proof refines Lemma \ref{lemma leftmost branch is del or has hyperdegree W}
as follows.

\begin{corollary}
\label{corollary leftmost branch is del or has hyperdegree W refined}
Suppose that $T$ is a recursive tree with non-empty body and that $\alpha_L$ is its leftmost infinite branch. Then
exactly one of the following cases applies.
\begin{itemize}
\item[(a)] $\alpha_L$ is \del.
\item[(b)] $\alpha_L \heq \W$ and there exists some $u \sqsubseteq \alpha_L$ such that $T_u$ is a Kleene tree.
\end{itemize}
In the case of \tu{(b)} there exists $v \sqsubseteq \alpha_L$ such that $T_v$ is a Kleene tree and $\spat{T_v} \dless \spat{T}$.
\end{corollary}

\begin{proof}
If for all $u \sqsubseteq \alpha_L$ we have that $\del \cap [T_u] \neq \emptyset$
then as in the proof of Theorem \ref{theorem for every kleene space there is another one from below} we obtain that
\[
\alpha \in \{\alpha_L\} \iff \alpha \in
[T] \ \& \ (\forall \beta \in \del)[\beta \in [T] \longrightarrow \alpha \leq_{\rm lex} \beta].
\]
The latter equivalence implies that $\{\alpha_L\} \in \sig$ and hence $\alpha_L \in \del$. So if $\alpha_L \not \in
\del$ (and therefore from Lemma \ref{lemma leftmost branch is del or has hyperdegree W} $\alpha_L \heq \W$) there
exists some $u \sqsubseteq \alpha_L$ such that $\del \cap [T_u] = \emptyset$. Clearly $T_u$ is a Kleene tree.

Finally if (b) is true we apply Theorem \ref{theorem for every kleene space there is another one from below} to $T_u$ in order to find an initial segment of its leftmost infinite branch, which is $\alpha_L$ since $u \sqsubseteq \alpha_L$, such that $\spat{T_v} \dless \spat{T_u} \dleq \spat{T}$.

\end{proof}

The proof of Theorem \ref{theorem for every kleene space there is another one from below} has another
interesting consequence.

\begin{corollary}
\label{corollary del of gamma is dense in a Kleene tree}
For every Kleene tree $T$ and every $\gamma \in \ca{N}$ the class $\del(\gamma)$ is dense in $[T]$ exactly when $\W
\in \del(\gamma)$.
\end{corollary}

\begin{proof}
Suppose that $\W \hleq \gamma$ and that $[T_u] \neq \emptyset$. From the Kleene Basis Theorem \ref{theorem Kleene's
Basis Theorem} there exists some $\alpha \in [T_u]$ which is recursive in $\W$. Therefore $\alpha \hleq \gamma$ and
$\del(\gamma) \cap [T_u] \neq \emptyset$. Conversely if $\W \not \in \del(\gamma)$ from the Claim in the proof of Theorem
\ref{theorem for every kleene space there is another one from below} there exists some $u \in [T]$ with $[T_u] \neq
\emptyset$ and $\del(\gamma) \cap [T_u] = \emptyset$.
\end{proof}

We now proceed to another type of results, where the key tool for proving them is Kreisel compactness (Theorem \ref{theorem Kreisel compactness}). The idea behind it, is to choose the set $D$ in the statement of the preceding theorem so that it involves trees as arguments -this is the reason that we are considering the space \Tr. The set $D$ should be chosen in such a way that on one  hand membership in $\cap_{i \in P}D_i$ implies the incomparability property that we want to achieve, and on the other hand it is relatively simple to show that $\cap_{i \in H}D_i$ is non empty for all \del \ sets $H \subseteq P$. This technique is demonstrated in the next lemma.

\begin{lemma}
\label{lemma for every kleene tree there exists an incomparable kleene tree}
For all Kleene trees $T_1,\dots,T_n$ there exists a Kleene tree $S$ and $\alpha_1,\dots,\alpha_n,\beta$ in $[T_1],\dots, [T_n]$ and $[S]$ respectively such that
\[
\del(\beta) \cap [T_k] = \emptyset \quad \textrm{and} \quad \gamma \not \heq \alpha_k
\]
for all $\gamma \in [S]$ and all $k=1,\dots,n$.\smallskip

In particular for all $k=1,\dots,n$ there is no \del-recursive function $f: \ca{N} \to \ca{N}$ which carries $[S]$ inside $[T_k]$ or carries $[T_k]$ inside $[S]$ and is injective on $[T_k]$.
\end{lemma}

\begin{proof}
Let us first prove the result for just one Kleene tree $T_1 = T$. We will apply Kreisel compactness on some suitable chosen set $D$. Consider the \pii-recursive partial function $\pdel: \om \rightharpoonup \ca{N}$ which parametrizes the set of all points of \ca{N} in \del \ \cf Theorem \ref{theorem alltogether}-(\ref{theorem 4D.2 and theorem 4D.15}) and let $I$ be the set $\set{i \in \om}{\pdel(i) \downarrow}$. The set $I$ is \pii \ and not in \del, for otherwise the set of \del \ points of \ca{N} would be in \del.

We consider a regular $\pii$-norm $\ggf$ on $I$. There exist relations $\leq^\varphi_{\sig}$, $\leq^\varphi_{\pii}$ in \sig \ and \pii \ respectively such that
\[
[j \in I \ \& \ \varphi(j) \leq \varphi(i)] \iff j \leq^\varphi_{\sig} i \iff j \leq^\varphi_{\pii} i
\]
for all $i \in I$ and all $j \in \om$, \cf 4B \cite{yiannis_dst}.

Define the set $D \subseteq \om \times \ca{N}^2 \times \Tr$ by saying that $D(i,\alpha,\beta,S)$ holds exactly when
\begin{align}
\label{equation lemma for every kleene tree there exists an incomparable kleene tree 1} & \ \ \ \ \alpha \in [T] \ \& \ S \ \textrm{is a recursive tree} \ \& \ \beta \in [S]\\
\label{equation lemma for every kleene tree there exists an incomparable kleene tree 2} &\& \ (\forall \gamma \in \del)(\forall j \leq^\varphi_{\pii} i)[\gamma \in [S] \ \longrightarrow \ \gamma \neq \pdel(j)]\\
\label{equation lemma for every kleene tree there exists an incomparable kleene tree 3} &\& \ (\forall \delta \in \del(\beta))[\delta \not \in [T]]\\
\label{equation lemma for every kleene tree there exists an incomparable kleene tree 4} &\& \ (\forall \gamma \in \del(\alpha))[\alpha \in \del(\gamma) \ \longrightarrow \ \gamma \not \in [S]].
\end{align}
(Notice that it is only condition (\ref{equation lemma for every kleene tree there exists an incomparable kleene tree 2}) which depends on $i$.) Using the Theorem on Restricted Quantification \ref{theorem alltogether}-(\ref{theorem 4D.3}) and the fact that the relations $\alpha \in \del(\beta)$ and $\leq^\varphi_{\pii}$ are \pii \ it is easy to see that $D$ is a \sig \ set, so in particular $D$ is computed by a \sig \ set on $I \times \ca{N}^2 \times \Tr$. Moreover since the relation $\leq^\varphi_{\pii}$ is transitive we have that $D_i \subseteq D_j$ for all $j \leq^\varphi_{\pii} i$.

If $H$ is a \del \ subset of $I$ then from the Covering Lemma (4C.11 in \cite{yiannis_dst}) there exists some $\xi < \ck$ such that $H \subseteq \set{j}{\ggf(j) \leq \xi}$. Since the norm $\ggf$ is regular there exists some $i \in I$ such that $\ggf(i) = \xi$. Hence $\ggf(j) \leq \ggf(i)$ for all $j \in H$. It follows that $D_i \subseteq \cap_{j \in H}D_j$. So if we prove that $D_i \neq \emptyset$ for all $i \in I$ we will have that $\cap_{j \in H}D_j \neq \emptyset$ for all \del \ sets $H \subseteq I$ and therefore from Kreisel compactness (Theorem \ref{theorem Kreisel compactness}) it will follow that $\cap_{i \in I}D_i \neq \emptyset$.

Now we verify that $D_i \neq \emptyset$ for all $i \in I$. Fix some $i \in I$ and consider the set
\begin{align*}
C&=\set{\ep}{(\exists j \in I)[\varphi(j) \leq \varphi(i) \ \& \ \ep \ \textrm{is recursive in} \ \pdel(j)]}
\end{align*}
Since $i \in I$, using the relations $\leq^\varphi_{\sig}, \leq^\varphi_{\pii}$ from above and the fact that $\pdel$ is a \pii-recursive partial function, it is easy to verify that $C$ is a \del \ set. Moreover it is clear that $C \subseteq \del \cap \ca{N}$. The latter is not a \del \ set and so there exists some $\ep \in \del$ such that $\ep \not \in C$. The singleton $\{\ep\}$ is a \del \ set so from Theorem \ref{theorem alltogether}-(\ref{theorem 4D.9}) and Lemma \ref{lemma pi-zero-one sets recursive tree} there exists a recursive tree $S$ such that $[S]$ is a singleton, say $\{\beta\}$, and a recursive function $\pi: \ca{N} \to \ca{N}$ such that $\pi(\beta) = \ep$. We notice that $\beta \in \del$ and that $\ep$ is recursive in $\beta$.

We claim that $(\alpha,\beta,S) \in D_i$, where $\alpha$ is any infinite branch of $T$. Condition (\ref{equation lemma for every kleene tree there exists an incomparable kleene tree 1}) is clearly satisfied and since $\beta \in \del$ and $T$ is a Kleene tree condition (\ref{equation lemma for every kleene tree there exists an incomparable kleene tree 3}) is satisfied as well. Condition (\ref{equation lemma for every kleene tree there exists an incomparable kleene tree 4}) is also satisfied for if $\gamma$ has the same hyperdegree as $\alpha$ then $\gamma$ cannot be \del \ (as $\alpha \in [T]$ and $T$ is a Kleene tree). In particular $\gamma \not \in [S] = \{\beta\}$.

Now we show condition (\ref{equation lemma for every kleene tree there exists an incomparable kleene tree 2}). If $\beta$ were equal to $\pdel(j)$ for some $j \leq^\varphi_{\pii} i$ then $\ep$ being recursive in $\beta$ would be a member of $C$, a contradiction. Therefore $\beta \neq \pdel(j)$ for all $j \leq^\varphi_{\pii} i$. So if $\gamma$ is in $[S]$ then $\gamma = \beta$ and hence $\gamma \neq \pdel(j)$ for all $j \leq^\varphi_{\pii} i$. In particular condition (\ref{equation lemma for every kleene tree there exists an incomparable kleene tree 2}) is satisfied.

Consider now $(\alpha,\beta,S)$ in $\cap_{i \in I} D_i$. It is clear that $S$ is a recursive tree and that $\alpha$ and $\beta$ are infinite branches of $T$ and $S$ respectively. Moreover from condition (\ref{equation lemma for every kleene tree there exists an incomparable kleene tree 3}) we have that $\del(\beta) \cap [T] = \emptyset$ and if $\gamma$ has the same hyperdegree as $\alpha$ then from condition (\ref{equation lemma for every kleene tree there exists an incomparable kleene tree 4}) we have that $\gamma \not \in [S]$.

It remains to show that $S$ is a Kleene tree and in particular that $[S]$ does not contain \del \ members. Indeed if $\gamma \in \del$ and $\gamma = \pdel(i)$ for some $i \in I$ then since $(\alpha,\beta,S) \in D_i$ from condition (\ref{equation lemma for every kleene tree there exists an incomparable kleene tree 2}) it follows that $\gamma \not \in [S]$.

The result for a finite collection of Kleene trees $T_1,\dots,T_n$ is proved in a similar way: we define the set $D \subseteq \om \times \ca{N}^{n+1} \times \Tr$ by saying that $D(i,\alpha_1,\dots,\alpha_n,\beta,S)$ holds exactly when\newpage
\begin{align*}
 & \ \ \ (\forall k=1,\dots,n)[\alpha_k \in [T_k] \ \& \ S \ \textrm{is a recursive tree} \ \& \ \beta \in [S]]\\
 &\& \ (\forall \gamma \in \del)(\forall j \leq^\varphi_{\pii} i)[\gamma \in [S] \ \longrightarrow \ \gamma \neq \pdel(j)]\\
 &\& \ (\forall k=1,\dots,n)(\forall \delta \in \del(\beta))[\delta \not \in [T_k]]\\
 &\& \ (\forall k=1,\dots,n)(\forall \gamma \in \del(\alpha_k))[\alpha_k \in \del(\gamma) \ \longrightarrow \ \gamma \not \in [S]].
\end{align*}
Then all preceding steps go through.

To finish the proof suppose that $f: \ca{N} \to \ca{N}$ is a \del-recursive function. If $f$ was carrying $[S]$ inside some $[T_k]$ then $f(\beta)$ would be a $\del(\beta)$ member of $[T_k]$, a contradiction. Also if $f$ was carrying $[T_k]$ inside $[S]$ and $f$ were injective on $[T_k]$ then $\alpha_k$ would have the same hyperdegree as $f(\alpha_k) \in [S]$, which is again a contradiction.
\end{proof}

\begin{remark}
\label{remark on lemma for every kleene tree there exists an incomparable kleene tree}\normalfont
A similar version of the preceding lemma can be stated for an infinite sequence of Kleene trees $(T_n)_{\n}$ for which the relation
\[
P(t,n) \iff \dec{t} \in T_n
\]
is \del.
\end{remark}

\begin{theorem}
\label{theorem for every kleene tree there exists another one incomparable and incomparable Kleene spaces}

For every finite sequence $\ca{X}_1, \ca{X}_2, \dots, \ca{X}_n$ of Kleene spaces, there is a Kleene space \ca{Y} which is \dleq-incomparable with each $\ca{X}_i$.\smallskip

It follows that every Kleene space is the member of an infinite sequence of Kleene spaces which are pairwise incomparable under \dleq.
\end{theorem}

\begin{proof}
Apply Lemmas \ref{lemma for every kleene tree there exists an incomparable kleene tree} and \ref{lemma converse of lemma bodyT less than bodyS in Kleene trees}.
\end{proof}

Another consequence of Lemma \ref{lemma for every kleene tree there exists an incomparable kleene tree} is the existence of a strictly increasing infinite sequence of Kleene spaces under \dleq.

\begin{theorem}
\label{theorem strictly increasing Kleene spaces}
For every Kleene tree $T$ there exists a Kleene tree $K$ such that $\spat{T} \dless \spat{K}$.\smallskip

It follows that every Kleene space \ca{X} is the bottom of an infinite sequence
of Kleene spaces which is strictly increasing under \dleq,
\[
\ca{X} \dless \ca{X}_1 \dless \ca{X}_2 \dless \dots.
\]
\end{theorem}

\begin{proof}
Suppose $T$ is a Kleene tree. From Lemma \ref{lemma for every kleene tree there exists an incomparable kleene tree} there exists a Kleene tree $S$ such that no \del-recursive function carries $[S]$ inside $[T]$. We take $K$ to be the sum $T \oplus S$.

Then $K$ is a Kleene tree and clearly we have that $\spat{T} \dleq \spat{K}$. If it were $\spat{K} \dleq \spat{T}$ then from Lemma \ref{lemma converse of lemma bodyT less than bodyS in Kleene trees} there would be a \del-recursive function $f: \ca{N} \to \ca{N}$ carrying $[K]$ inside $[T]$. Therefore there would be a \del-recursive function carrying $[S]$ inside $[T]$, contradicting the choice of $S$.
\end{proof}

The authors of \cite{fokina_friedman_sy_toernquist_the_effective_theory_of_Borel_equivalence_relations} prove a strong incomparability result about Kleene trees using model theoretic methods and the more general \emph{compactness of Barwise}, \cf \cite{barwise_infinitary_logic_and_admissible_sets} and \cite{barwise_admissible_sets_and_structures} Chapter III Theorem 5.6.

\begin{theorem}[Fokina-S. Friedman-T\"ornquist, \cf \cite{fokina_friedman_sy_toernquist_the_effective_theory_of_Borel_equivalence_relations}]
\label{theorem Fokina-Friedman-Toernquist}
There exists a sequence of recursive trees $(T_i)_{\iin}$ and a sequence $(\alpha_i)_{\iin}$ in \ca{N} such that
$\alpha_i \in [T_i]$ and $\del(\alpha_i) \cap [T_j] = \emptyset$ for all $i \neq j$.\smallskip

In particular there does not exist a \del-recursive function $f: \ca{N} \to \ca{N}$ which carries $[T_i]$ inside
$[T_j]$ for all $i \neq j$.
\end{theorem}

\emph{Remark.} The $T_i$'s in this theorem are Kleene trees. Moreover using Lemma \ref{lemma converse of lemma bodyT less than bodyS in Kleene trees} one can see that Theorem \ref{theorem Fokina-Friedman-Toernquist} answers the question about \dleq-incomparability. Despite the fact that we can prove all our results about the relation $\ca{X} \dleq \ca{Y}$ without the help of Theorem \ref{theorem Fokina-Friedman-Toernquist}, this result has offered to the author a new insight into the problems of this type and in particular has inspired the use of Kreisel compactness.

Let us illustrate this by giving a description of a proof for Theorem \ref{theorem Fokina-Friedman-Toernquist}.\footnote{Our proof differs from the original one only in the use of Kreisel in favor of Barwise compactness and the consequent elimination of the model-theoretic considerations. All other arguments are the same.} For this one would need to familiarize themselves with a \emph{system of notation} of recursive ordinals (see \cite{yiannis_kleenes_amazing_second_recursion_theorem} for an excellent presentation of the topic) such as Kleene's $\mathit{O}$ (\cf \cite{kleene_on_notation_for_ordinal_numbers}) and \emph{transfinite Turing jumps}. We will use the following deep result of Harrington.

\begin{theorem}[Harrington]
\label{theorem Harrington xi-incomparable}
For every $\xi < \ck$ there exists a sequence $(\alpha_i)_{i \in \om}$ in $2^\om$ such that $\{\alpha_i\}$ is $\Pi^0_1$ and $\alpha_i \not \tleq (\oplus_{i \neq j}\alpha_j)^\xi$ for all $i \in \om$.\footnote{Harrington proved this result in the handwritten notes \cite{harrington_mclaughlins_conjecture}. For a presentation of Harrington's proof the reader can refer to \cite{gerdes_harringtons_solution}. Simpson \cite{simpson_implicit_definability_in_arithmetic} has recently obtained another interesting proof of this result.}
\normalfont

(By $\alpha^\xi$ we mean the $\xi^{\rm th}$ Turing jump of $\alpha$ determined by a notation $a$ for $\xi$ in Kleene's notation system $\mathit{O}$ and as usual $\tleq$ stands for Turing reducibility.)
\end{theorem}
\enlargethispage*{5mm}
We first show the result for two trees and then we indicate the changes needed to get it for infinitely many.

We define the set $D \subseteq \om \times (\ca{N} \times \Tr)^2$ as follows
\begin{align*}
D(i,\alpha,T,\beta,S) \iff& i \in \W \ \& \ \textrm{$T$ and $S$ are recursive trees}\\
                          &\& \ \alpha \in [T] \ \& \ \beta \in [S]\\
                          &\& \ \ckr{\alpha} = \ckr{\beta} = \ck\\
                          &\& \ (\forall \gamma \in \del(\alpha))\big \{\gamma \tleq \alpha^{|i|} \ \longrightarrow \ \gamma \not \in [S]\big \}\\
                          &\& \ (\forall \gamma \in \del(\beta))\big \{\gamma \tleq \beta^{|i|} \ \longrightarrow \ \gamma \not \in [T]\big \},
\end{align*}
where $\ckr{\alpha}$ is the least non $\alpha$-recursive ordinal and $|i|$ the order type of $\leq_i$ for $i \in \W$.

One can verify that the set $D$ is computed by a \sig \ set on $\W \times (\ca{N} \times \Tr)^2$ and that in order to show that every section $\cap_{j \in H}D_j$ is non-empty (where $H \subseteq \W$ is \del) it is enough to show that $D_i$ is non-empty for all $i \in \W$.

To prove the latter we apply Theorem \ref{theorem Harrington xi-incomparable} for $\xi = |i|$ and we get some $\alpha$ and $\beta$ in $2^\om$ such that $\{\alpha\}, \{\beta\}$ are $\Pi^0_1$ and $\alpha \not \tleq \beta^{|i|}$, $\beta \not \tleq \alpha^{|i|}$. Hence there exist recursive trees $T$ and $S$ such that $[T] = \{\alpha\}$ and $[S] = \{\beta\}$. It is easy to verify that $\alpha$ and $\beta$ are in \del \ and so $\ckr{\alpha} = \ckr{\beta} = \ck$. Moreover for all $\gamma \in \ca{N}$ if $\gamma \in [S]$ we have that $\gamma = \beta$ and so $\gamma \not \tleq \alpha^{|i|}$, and if $\gamma \in [T]$ we have that $\gamma = \alpha$ and so $\gamma \not \tleq \beta^{|i|}$. Hence $D(i,\alpha,T,\beta,S)$ holds.

From Theorem \ref{theorem Kreisel compactness} there exists some $(\alpha,T,\beta,S)$ in $\cap_{i \in \W}D_i$. Clearly $T$ and $S$ are recursive trees and $\alpha, \beta$ are in $[T]$ and $[S]$ respectively. We now verify that $\del(\alpha) \cap [S] = \emptyset$. Suppose that $\gamma \in \del(\alpha)$. Then we have that $\gamma \tleq \alpha^\xi$ for some $\xi < \ckr{\alpha}$. Since $\ckr{\alpha} = \ck$ there exists an $i \in \W$ such that $\xi = |i|$. Since $(\alpha,T,\beta,S)$ is in $D_i$ and $\gamma \tleq \alpha^{|i|}$ it follows that $\gamma$ is not in $[S]$ and so $\del(\alpha) \cap [S] = \emptyset$. Similarly one shows that $\del(\beta) \cap [T] = \emptyset$.

In order to get the result for infinitely many trees we consider the recursively presented space $\Tr^\om$ and we define $D^\ast \subseteq \om \times \ca{N} \times \Tr^\om$ by\allowbreak
\begin{align*}
D^\ast(i,\alpha,(T_n)_{\n}) \iff& i \in \W\\
                          &\& \ (\forall n)[\textrm{the tree $T_n$ is recursive} \ \& \ (\alpha)_n \in [T_n]]\\
                          &\& \ (\forall n)[\ckr{(\alpha)_n}=\ck]\\
                          &\& \ (\forall n)(\forall m \neq n)(\forall \gamma \in \del((\alpha)_n))\\
                          &\hspace*{28mm}\big \{\gamma \tleq (\alpha)_n^{|i|} \ \longrightarrow \ \gamma \not \in [T_m]\big \}.
\end{align*}
The set $D^\ast$ is computed by a \sig \ set on $\om \times \ca{N} \times \Tr^\om$. By applying Theorem \ref{theorem Harrington xi-incomparable} one can verify that $D_i$ is non-empty for all $i \in \W$ and if $(\alpha,(T_n)_{\n})$ is in $\cap_{i \in \W}D_i$ then the sequences $((\alpha)_n)_{\n}$ and $(T_n)_{\n}$ satisfy the conclusion.

\smallskip

\subsection{Analogy with recursive pseudo-well-orderings} A recursive linear ordering $\preceq$ on a subset of \om \ is a \emph{pseudo-well-ordering} if it has an infinite strictly decreasing sequence but no such sequence is \del \ (\cf \cite{harrison_recursive_pseudo_wellorderings}). Notice that a recursive pseudo-well-ordering has necessarily a least element, for otherwise one could construct a recursive strictly decreasing sequence.

Kleene pointed out (\cf \cite{kleene_on_the_forms_of_the_predicates_in_the_theory_of_constructive_ordinals_II} pp.421--422) that Kleene trees are naturally connected to the class of recursive pseudo-well-orderings: the Kleene-Brouwer ordering on a Kleene tree is a pseudo-well-ordering (see also the first lines of the proof of Lemma \ref{lemma isomorphic orders give isomorphic Nts}). On the other hand if $\preceq$ is a recursive pseudo-well-ordering then it is clear the set
\[
T =\set{u \in \omseq}{u(\lh(u)-1) \prec \dots \prec u(0)}
\]
is a Kleene tree. By exploiting these ideas we can translate results of this section into the language of recursive pseudo-well-orderings.

A function $f: {\rm Field}(\preceq_1) \to {\rm Field}(\preceq_2)$ is an \emph{isomorphism} between the linear orderings $\preceq_1$ and $\preceq_2$ if it is bijective and
\[
x \preceq_1 y \iff f(x) \preceq_2 f(y)
\]
 for all $x,y \in {\rm Field}(\preceq_1)$. The function $f$ is an \emph{initial similarity} if it is an isomorphism between $\preceq_1$ and an initial segment of $\preceq_2$.

Suppose that $\preceq_1$ and $\preceq_2$ are recursive linear orderings on a (necessarily recursive) subset of \om. We say that $f: {\rm Field}(\preceq_1) \to {\rm Field}(\preceq_2)$ is a \emph{\del-isomorphism} if $f$ is an isomorphism and  \del-recursive. Similarly $f$ is a \emph{\del \ initial similarity} if $f$ is an initial similarity and \del-recursive. We say that $\preceq_1$ is \emph{strictly below} $\preceq_2$ \emph{under \dleq} \ if there exists a \del \ initial similarity $f: {\rm Field}(\preceq_1) \to {\rm Field}(\preceq_2)$ but there does not exist a \del \ initial similarity $g: {\rm Field}(\preceq_2) \to {\rm Field}(\preceq_1)$. We also say that two recursive orderings are \emph{\del-comparable} or that are \emph{comparable under \dleq} if there exists a function from the field of the one order to the field of the other which is a \del \ initial similarity.

Let us recall the following result of Gandy.

\begin{theorem}[Gandy \cf \cite{gandy_proof_of_mostowskis_conjecture}]
\label{theorem Gandy length of well-founded part}
For every recursive pseudo-well-ordering $\preceq$ with well-founded part $\WF(\preceq)$ we have that
\[
\sup\set{|\iniseg(n)|}{n \in \WF(\preceq)} = \ck,
\]
where $\iniseg(n)$ is the initial segment of $n$ with respect to $\preceq$ and $|\cdot|$ is the order type.\smallskip

In particular the well-founded part of a recursive pseudo-well-ordering is not a \del \ set.
\end{theorem}

\begin{proof}[\textit{Proof} \tu{(Well-known).}]
It is clear that $$\sup\set{|\iniseg(n)|}{n \in \WF(\preceq)} \leq \ck,$$ since $\preceq$ is recursive. Assume
towards a contradiction that the preceding supremum is less than $\ck$. Then there exists some $e^\ast \in \W$ such that $|\iniseg(n)| \leq |\leq_{e^\ast}|$ for all $n \in \WF(\preceq)$. It follows that
\begin{align*}
n \in \WF(\preceq) \iff& (\exists \gamma)[\gamma \ \textrm{embedds} \ \iniseg(n) \ \textrm{into} \ \mathrm{Field}(\leq_{e^\ast})]\\
                  \iff& (\exists \gamma)(\forall i,j)\big \{ [i \preceq n \longrightarrow (\exists t)[\rfn{e^\ast}(\langle i,t \rangle) = 0]]\\
                      & \hspace*{15mm} \ \& \ [i \preceq j \preceq n \longleftrightarrow \rfn{e^\ast}(\langle \gamma(i),\gamma(j) \rangle) = 0]\big \},
\end{align*}
for all \n. From the latter equivalence we have that $\WF(\preceq)$ is \sig, and since the well-founded part of a recursive linear ordering is $\pii$, we have that $\WF(\preceq) \in \del$.

Now suppose that $m_0$ is in the ill-founded part of $\preceq$, which is non-empty since $\preceq$ is a pseudo-well-ordering. Define recursively $\alpha: \om \to \om$ as follows
\begin{align*}
\alpha(0)   =& m_0\\
\alpha(n+1) =& \textrm{the least $m \not \in \WF(\preceq)$ such that} \ m \prec \alpha(n).
\end{align*}
It follows that $\alpha$ is a $\WF(\preceq)$-recursive function. Since $\WF(\preceq) \in \del$ we have that $\alpha \in \del$. Moreover it is clear that $\alpha(n+1) \prec \alpha(n)$ for all $n$, \ie \ $\preceq$ has a \del \ strictly decreasing sequence, a contradiction.
\end{proof}

It is worth putting down the following characterization of the well-founded part of a recursive linear ordering, which follows from the proofs of Theorem \ref{theorem Gandy length of well-founded part} and of \sig-boundedness.

\begin{lemma}[Well-known]
\label{lemma well founded part is del characterization of}
The well-founded part $\WF(\preceq)$ of a recursive linear ordering $\preceq \ \subseteq \om \times \om$ is \del \
exactly when
\[
\sup\set{|\iniseg(n)|}{n \in \WF(\preceq)} < \ck.
\]
\end{lemma}

\begin{proof}
By definition we have that $\sup\set{|\leq_e|}{e \in \W} = \ck$. Suppose that $\sup\set{|\iniseg(n)|}{n \in \WF(\preceq)} = \ck$. It is then easy to see that for all $e$ for which $\rfn{e}$ is total and $\leq_e$ is a linear ordering we have that
\begin{align*}
\W(e) \iff& (\exists \gamma)(\exists n)[n \in \WF(\preceq) \ \& \ \gamma \ \textrm{embedds} \ \mathrm{Field}(\leq_e) \ \textrm{into} \ \iniseg(n)]\\
     \iff& (\exists \gamma)(\exists n)[n \in \WF(\preceq)\\
         & \ \ \& \ (\forall i,j,t) \big \{ [(\rfn{e}(\langle i,t \rangle) = 0) \longrightarrow \gamma(i) \preceq n]\\
         &\hspace{2cm} \ \& \ [\rfn{e}(\langle i,j \rangle) = 0 \longleftrightarrow \gamma(i) \preceq \gamma(j)] \big \}].
\end{align*}
This shows that if $\sup\set{|\iniseg(n)|}{n \in \WF(\preceq)} = \ck$ then $\W$ is $\sig(\WF(\preceq))$ and since $\W$ is not in \sig \ the set $\WF(\preceq)$ cannot be in $\del$. Therefore we have proved the left-to-right-hand direction of the statement.

The converse is immediate from the proof of Theorem \ref{theorem Gandy length of well-founded part}, \ie if the given supremum is less than \ck, then from the first part of the preceding proof it follows that $\WF(\preceq)$ is \del.
\end{proof}

Harrison \cite{harrison_recursive_pseudo_wellorderings} proved that the order type of a recursive pseudo-well-ordering has the form $\ck + \Q \times \ck + \eta$ for some $\eta < \ck$. Hence for every two recursive pseudo-well-orderings there exists an initial similarity from the field of the one ordering to the field of the other.

\begin{theorem}
\label{theorem every recursive pseudo wo has a recursive pseudo wo from below}
For every recursive pseudo-well-ordering $\preceq \subseteq \om^2$ there exists some $a \in {\rm Field}(\preceq)$ not in the well-founded part $\WF(\preceq)$ of $\preceq$ such that there is no \del \ initial similarity from ${\rm Field}(\preceq)$ to ${\rm Field}(\preceq_b)$ for all $b \preceq a$, where $$\preceq_b = \set{(n,m)}{n \preceq m \prec b}.$$

In particular every recursive pseudo-well-ordering is the top of a strictly decreasing sequence of recursive pseudo-well-orderings under \dleq.
\end{theorem}

\begin{proof}
We consider the set of strictly decreasing chains
\[
C = \set{\beta \in \ca{N}}{(\forall n)[\beta(n+1) \prec \beta(n)]}.
\]
Clearly $C$ is a non-empty $\Pi^0_1$ set. From the Gandy Basis Theorem there exists a $\gamma \in C$ such that $\W \not \in \del(\gamma)$. We claim that there is an $a \in {\rm Field}(\preceq) \setminus \WF(\preceq)$ such that for all $\beta \in C$ with $\beta(0) \preceq a$ it holds that $\beta \not \in \del(\gamma)$. Indeed if this were not the case then we would have that
\[
a \in \WF(\preceq) \iff (\forall \beta \in \del(\gamma))[\beta \in C \ \longrightarrow \ a \prec \beta(0)],
\]
for all $a \in {\rm Field}(\preceq)$. The preceding equivalence implies that the set $\WF(\preceq)$ is $\sig(\gamma)$. Since evidently $\WF(\preceq)$ is a \pii \ set it would follow that $\WF(\preceq) \in \del(\gamma)$.

On the other hand it follows from Theorem \ref{theorem Gandy length of well-founded part} that $\WF(\preceq)$ is not a \del \ set, so from Spector's Theorem \ref{theorem spector two hyperdegrees} $\W \heq \WF(\preceq)$ and so we would have that $\W \in \del(\gamma)$, contradicting the choice of $\gamma$.

Now we choose some $a \in {\rm Field}(\preceq) \setminus \WF(\preceq)$ such that $a \prec \beta(0)$ for all $\beta \in C \cap \del(\gamma)$. If there were a \del-isomorphism between $\preceq$ and $\preceq_b$ for some $b \preceq a$ then $\gamma \in C$ would be carried to some $\beta \in C \cap \del(\gamma)$ with $\beta(0) \prec b \preceq a$, a contradiction.
\end{proof}


Kreisel (\cf \cite{kreisel_some_axiomatic_results_second_order_arithmetic}) showed that there are recursive pseudo-well-orderings which are \del-incomparable. We will show how one can prove this result from Theorem \ref{theorem Fokina-Friedman-Toernquist}.

\begin{lemma}[Folklore?]
\label{lemma isomorphic orders give isomorphic Nts}
Suppose that $T$ and $S$ are recursive trees and that
$f: \om \to \om$ is a \del \ function such that
\begin{align*}
(\forall t)[\dec{t} \in T \ \Longrightarrow \ \dec{f(t)} \in S]
\end{align*}
and
\begin{align*}
\dec{s} \kbleq \dec{t} \ \iff \ \dec{f(s)} \kbleq \dec{f(t)}
\end{align*}
for all $\dec{s}$, $\dec{t} \in T$. Then there exists a \del \ function $$\pi: \ca{N} \to \ca{N}$$ which is injective on $[T]$ and $\pi[[T]] \subseteq [S]$. In particular we have that $\spat{T} \dleq \spat{S}$.
\end{lemma}

\begin{proof}
Let us recall that if $(u_n)_{\n}$ is a sequence in a tree $K$ which is strictly decreasing with respect to $\kbleq$ then for all $i$ there exists some $n_i$ such that $i < \lh(u_n)$ and $u_n(i)$ is constant for all $n \geq n_i$. This defines the infinite branch $\gamma \in [K]$ by
\begin{align*}
\gamma(i) = j \iff& \textrm{for all but finitely many $n$'s we have that} \ i < \lh(u_n)\\
                         & \textrm{and} \ u_n(i) = j.
\end{align*}
Moreover we have that $\gamma \upharpoonright k \kbleq u_n \upharpoonright k$ for all $n$ and all $k < \lh(u_n)$.

For every $\alpha \in [T]$ we define
\[
t^\alpha_n = \langle \alpha(0),\dots, \alpha(n-1) \rangle,
\]
so that $\alpha \upharpoonright n = \dec{t^\alpha_n}$ for all $n$. From our hypothesis it follows that
\[
\dec{f(t^\alpha_0)} \kbgreat \dec{f(t^\alpha_1)} \kbgreat \dots \kbgreat \dec{f(t^\alpha_n)} \kbgreat \dots.
\]
We define the function $\pi: [T] \to [S]$ by
\begin{align*}
\pi(\alpha)(i) = j \iff& \textrm{for all but finitely many $n$'s we have}\\
                         & i < \lh(\dec{f(t^\alpha_n)}) \ \textrm{and} \ \dec{f(t^\alpha_n)}(i) = j.
\end{align*}
We extend $\pi$ to \ca{N} by assigning the value $0$ to every $\alpha \not \in [T]$. Clearly the function $\pi$ is \del \ recursive and $\pi[[T]] \subseteq [S]$.

We now verify that $\pi$ is injective on $[T]$. Indeed let $\alpha, \beta \in [T]$ be such that $\alpha \upharpoonright n = \beta \upharpoonright n$ and $\alpha(n) < \beta(n)$ so that $\alpha \upharpoonright (n'+1) \kbless \beta \upharpoonright m$ for all $m > n' \geq n$. It follows that
\[
\dec{f(t^\alpha_{n'+1})} \kbless \dec{f(t^\beta_m)}
\]
for all $m > n' \geq n$. From this it is easy to see that for some big enough $n'$, the finite sequence \dec{f(t^\alpha_{n'+1})} is not an initial segment of $\pi(\beta)$ and that $\dec{f(t^\alpha_{n'+1})}(k) < \pi(\beta)(k)$ if $k < \lh(\dec{f(t^\alpha_{n'+1})})$ is the least such that $\dec{f(t^\alpha_{n'+1})} \upharpoonright (k+1) \neq \pi(\beta) \upharpoonright (k+1)$. Hence
\[
\pi(\alpha) \upharpoonright (k+1) \kbleq \dec{f(t^\alpha_{n'+1})} \upharpoonright (k+1) \kbless \pi(\beta) \upharpoonright (k+1)
\]
and so $\pi(\alpha) \neq \pi(\beta)$.
\end{proof}

From the preceding lemma it is clear that if $T$ and $S$ are as in Theorem \ref{theorem Fokina-Friedman-Toernquist} then the ordered spaces $(T,\kbleq), (S,\kbleq)$ are not \del-comparable.

Harrison \cite{harrison_recursive_pseudo_wellorderings} proved a considerable extension of Kreisel's result on \del-incomparable recursive pseudo-well-orderings. Using the methods of the proof of Lemma \ref{lemma for every kleene tree there exists an incomparable kleene tree} we can prove a similar assertion. Given a linear ordering $\preceq$ we say that two
strictly decreasing sequences $(x_n)_{\n}$, $(y_n)_{\n}$ are \emph{equivalent} if for all $n$ there exists $m$ such that $y_m \prec x_n$ and for all $n$ there exists $m$ such that $x_m \prec y_n$. Let us recall the set \lo \  of codes of countable recursive linear orderings defined in the Introduction.

\begin{theorem}[Compare with Theorem 1.9 in \cite{harrison_recursive_pseudo_wellorderings}]
\label{theorem similar to Harrison}
Suppose that $\ca{A}$ is the set of all $e \in \lo$ such that $\leq_e$ has exactly one equivalence class of strictly decreasing sequences and that $\ca{S} \subseteq \om$ is \sig \ with $\ca{A} \subseteq \ca{S} \subseteq \lo$. Then \ca{S} contains the code of a recursive pseudo-well-ordering.\smallskip

In fact for every recursive pseudo-well-ordering $\preceq$ there exists an $e \in \ca{S}$ such that $\leq_e$ is a pseudo-well-ordering and incomparable with $\preceq$ under \dleq.
\end{theorem}

\begin{proof}
Suppose that $\preceq$ is a recursive pseudo-well-ordering. We repeat the proof of Lemma \ref{lemma for every kleene tree there exists an incomparable kleene tree} but this time we replace the body of a tree with the set of all infinite decreasing sequences with respect to the given order. To be more specific for every linear ordering $\preceq^\ast$ on a subset of \om \ we consider the set of all infinitely decreasing sequences under $\preceq^\ast$,
\[
C(\preceq^\ast) = \set{\beta \in \ca{N}}{(\forall n)[\beta(n+1) \prec^\ast \beta(n)]},
\]
and let $I, \varphi, \leq^\varphi_{\sig}$ and $\leq^\varphi_{\pii}$ be as in the proof of Lemma \ref{lemma for every kleene tree there exists an incomparable kleene tree}.

We define the set $D^\ast \subseteq \om \times \ca{N}^2 \times \om$ by saying that $D^\ast(i,\alpha,\beta,e)$ holds exactly when
\begin{align*}
 & \ \ \ \ \alpha \in C(\preceq) \ \& \ \beta \in C(\leq_e) \ \& \ e \in \ca{S}\\
 &\& \ (\forall \gamma \in \del)(\forall j \leq^\varphi_{\pii} i)[\gamma \in C(\leq_e) \ \longrightarrow \ \gamma \neq \pdel(j)]\\
 &\& \ (\forall \delta \in \del(\beta))[\delta \not \in C(\preceq)]\\
 &\& \ (\forall \gamma \in \del(\alpha))[\alpha \in \del(\gamma) \ \longrightarrow \ \gamma \not \in C(\leq_e)].
\end{align*}
We notice that $D^\ast$ is \sig. The proof that $D^\ast_i$ is non-empty for all $i \in I$ is exactly as before, with the minor addition that $e$ is the code of $(S,\kbleq)$, which belongs to $\ca{A}$, since $S$ has exactly one infinite branch. Hence $e \in \ca{S}$.

Now we consider $(\alpha,\beta,e)$ in the intersection $\cap_{i \in I}D^\ast_i$. The order $\leq_e$ has a strictly decreasing sequence, namely $\beta$, but no such sequence is \del, hence $\leq_e$ is a pseudo-well-ordering. Also no strictly decreasing sequence in $\preceq$ is $\del(\beta)$ and no strictly decreasing sequence in $\leq_e$ has the same hyperdegree as $\alpha$.

We notice that every strict-order-preserving function $$f: {\rm Field}(\preceq_1) \to {\rm Field}(\preceq_2)$$ gives rise to the injective function
\[
g: C(\preceq_1) \to C(\preceq_2): g(x) = (f(x(n)))_{\n}.
\]
Moreover if $f$ is \del-recursive so is $g$ and therefore $g(x) \heq x$ for all $x \in C(\preceq_1)$.

So if the orders $\preceq$ and $\leq_e$ were \del-comparable then $C(\leq_e)$ or $C(\preceq)$ would have a member with the same hyperdegree as the one of $\alpha$ or $\beta$ respectively. In either case we have a contradiction.
\end{proof}

\begin{corollary}
\label{corollary increasing sequences of pseudo_well_orderings}
For every recursive pseudo-well-ordering $\preceq$ there exists a recursive pseudo-well-ordering $\preceq'$ which is strictly above $\preceq$ under \dleq.\smallskip

Therefore every recursive pseudo-well-ordering is the bottom of a strictly increasing sequence of recursive pseudo-well-orderings under \dleq.
\end{corollary}

\begin{proof}
The same as the proof of Theorem \ref{theorem strictly increasing Kleene spaces}. Here one uses Theorem \ref{theorem similar to Harrison} and sums of linear orderings rather than Lemma \ref{lemma for every kleene tree there exists an incomparable kleene tree} and sums of trees.
\end{proof}

\subsection{Incomparable hyperdegrees in Kleene spaces} Given a recursively presented metric space \ca{X} we say that $x, y \in \ca{X}$ have \emph{incomparable hyperdegree} or they are \emph{hyperarithmeticaly incomparable} if $x \not \hleq y$ and $y \not \hleq x$. The existence of $x,y \in \ca{C}$ with incomparable hyperdegree was proved by Spector \cite{spector_incomparable_hyperdegrees}. Let us illustrate the methodology. One defines the set $P \subseteq \ca{C} \times \ca{C}$ by
\[
P(\alpha,\beta) \iff \alpha \in \del(\beta).
\]
Clearly every section $P^\beta$ is countable and thus it is meager. Since $P$ is coanalytic and thus has the property of Baire, we have from the Kuratowski-Ulam Theorem (\cf \cite{kechris_classical_dst} Theorem 8.41) that the set $P$ is also meager. It follows that the set
\[
Q(\alpha,\beta) \iff P(\alpha,\beta) \ \vee \ P(\beta,\alpha)
\]
is meager as well. Therefore there exist $(\alpha,\beta) \in \ca{C}$ which do not belong to $Q$, \ie they have incomparable hyperdegree.\footnote{The same method applies to Turing degrees as well.}

One can in fact prove that there are $\alpha, \beta \in \ca{C}$ of incomparable hyperdegree with $\alpha, \beta \hless \W$. To see this notice that the preceding set $Q$ is \pii \ and apply the Gandy Basis Theorem \ref{theorem Gandy basis}.

It is natural to ask if incomparable hyperdegrees occur in every uncountable recursively presented metric space \ca{X}. This is clear if \ca{X} is \del-isomorphic to the Baire space, since \del-injections preserve hyperdegrees (see Lemma \ref{lemma the inverse function is del and same hyperdegree}). The question in the general case is still open though, \cf Question \ref{question incomparable hyperdegrees everywhere}.

The reason that the preceding arguments do not go through in the general case, is because the given space may not be perfect. This allows the existence of countable non-meager sets. Consider for example a space of the form \spat{T}. Since $T$ consists of recursive members we have that $T \times T \subseteq Q$, where $Q$ is as above by replacing \ca{C} with \spat{T}. Also as we can see from Theorem \ref{theorem properties of Nt} the set $[T]$ is nowhere dense in \spat{T}. In particular the complement of $Q$ in $\spat{T} \times \spat{T}$ is meager and so we cannot infer that it is non-empty as we did in the case of $\ca{C}$.

One could still ask if it is possible to use measure theoretic arguments for some carefully chosen measure $\mu$ on $\spat{T}$. As we will see in Section \ref{section characterizations of the Baire space up to del-isomorphism} this is not possible, unless the space \spat{T} is \del-isomorphic to the Baire space, (see the comments following Corollary \ref{corollary Baire space and measure}).

Nevertheless we will show that incomparable hyperdegrees do exist in Kleene spaces, and therefore the general question is reduced to the existence of a \del-copy of a Kleene space inside the given uncountable recursively presented metric space. The latter is also an open problem, \cf Question \ref{question NT contains Kleene}. We note that our method provides another way of producing incomparable hyperdegrees.\footnote{The method of producing incomparable hyperdegrees through the Kuratowski-Ulam Theorem is by no means the only such method. Simpson for example, using results of \cite{gandy_sacks_a_minimal_hyperdegree}, proves the existence of a pair of hyperdegrees which among other things are incomparable (\cf \cite{simpson_minimal_covers_and_hyperdegrees} Theorem 3.1). This method however does not help us here as it requires the existence of perfect \del \ sets, which is not the case in Kleene spaces. Another method is Harrington's Theorem \ref{theorem Harrington xi-incomparable} together with Kreisel compactness as applied in the proof of Theorem \ref{theorem Fokina-Friedman-Toernquist}. But again this is not helpful, since Harrington's result refers to \del \ points and therefore it does not apply when we replace \ca{N} with a Kleene space.}

\begin{theorem}
\label{theorem incomparable hyperdegrees in Kleene spaces}
For every Kleene tree $T$ there exists an infinite countable set $C \subseteq [T]$ such that all distinct $\alpha, \beta \in C$ have incomparable hyperdegree. Moreover $C$ can be chosen so that $\alpha \hless \W$ for all $\alpha \in C$.
\end{theorem}

\begin{proof}
We define $L \subseteq \ca{N}$ as follows
\[
L(\alpha) \iff \alpha \in [T] \ \& \ (\forall \gamma \in \del(\alpha))[\gamma \in [T] \ \longrightarrow \ \alpha \leq_{\rm lex} \gamma].
\]
Clearly the leftmost infinite branch of $T$ belongs to $L$, so in particular $L \neq \emptyset$. It is also clear that $L$ is a \sig \ set and that it contains no \del \ members, since $T$ is Kleene tree. Therefore from the Effective Perfect Set Theorem $L$ contains a non-empty perfect subset $Q$. We have $Q = [S]$, where $S$ is a copy of the complete binary tree.

We notice that for all $\alpha, \beta \in L$ with $\alpha <_{\rm lex} \beta$ we have $\alpha \not \in \del(\beta)$. Let $\alpha_R$ be the rightmost infinite branch of $S$. Clearly for all $\alpha \in [S] \setminus \{\alpha_R\}$ the set
\[
R(\alpha) = \set{\beta \in [S]}{\alpha <_{\rm lex} \beta}
\]
is uncountable.

Now pick some $\alpha_0 \in [S] \setminus \{\alpha_R\}$. Since the set $\del(\alpha_0) \cap \ca{N}$ is countable there exists some $\alpha_1 \in R(\alpha_0)$ such that $\alpha_1 \not \in \del(\alpha_0)$ and $\alpha_1 \neq \alpha_R$. Since $\alpha_1$ is a member of $R(\alpha_0)$ from the preceding remarks we have $\alpha_0 \not \in \del(\alpha_1)$, and so $\alpha_0, \alpha_1$ have incomparable hyperdegree.

The set $\del(\alpha_0, \alpha_1) \cap \ca{N}$ is countable, so there exists some $\alpha_2 \in R(\alpha_1)$ such that $\alpha_2 \not \in \del(\alpha_0,\alpha_1)$ and $\alpha \neq \alpha_R$. Since $\alpha_0 <_{\rm lex} \alpha_1 <_{\rm lex} \alpha_2$ and all are members of $L$ it follows that $\alpha_i \not \in \del(\alpha_2)$ for $i=0,1$. Proceeding inductively we construct the required sequence $(\alpha_i)_{\iin}$ in $[S] \subseteq [T]$.

To get the result with hyperdegrees strictly below the hyperdegree of $\W$ we consider the non-empty \sig \ set
\[
L^\ast = \set{\alpha \in \ca{N}}{(\forall i \neq j)[\alpha_i \in [T] \ \& \ \alpha_i \not \in \del(\alpha_j) \ \& \ \alpha_j \not \in \del(\alpha_i)]}
\]
and we apply the Gandy Basis Theorem \ref{theorem Gandy basis}.
\end{proof}

\begin{corollary}
\label{corollary incomparable hyperdegrees}
Every recursively presented metric space, which contains a \del-isomorphic copy of a Kleene space, contains also a countable set, whose members are pairwise hyperarithmetically incomparable and their hyperdegrees are strictly below the hyperdegree of $\W$.
\end{corollary}
Simpson \cite{simpson_minimal_covers_and_hyperdegrees} proved the existence a non-empty \sig \ set whose members have minimal hyperdegree (see comments following Question \ref{question hyperdegrees perfect set} in Section \ref{section questions open problems} for the definition). From this it follows that there exists a non-empty perfect set, whose members are pairwise hyperarithmetically incomparable. As we will see in Section \ref{section questions open problems} minimal hyperdegrees do not have to occur in Kleene spaces (see Lemma \ref{lemma minimal hyperdegree does not always occur}), but it would nevertheless be interesting to see if there exists a perfect set $P \subseteq [T]$, where $T$ is Kleene tree, whose members are pairwise hyperarithmetically incomparable, \cf Question \ref{question hyperdegrees perfect set}.

\section{\sc Characterizations of \ca{N} up to \del \ isomorphism}
\label{section characterizations of the Baire space up to del-isomorphism}

In this section we give a characterization of the condition $\spat{T} \dequal \ca{N}$ in terms of the intrinsic properties of $T$. Our arguments imply another interesting characterization of the Baire space as the unique up to \del-isomorphism recursively presented metric space which admits a Borel measure satisfying some reasonable definability properties. We also give a characterization of being \del-isomorphic to the Baire space in terms of perfect sets.

\subsection{Copies of the complete binary tree} We have seen in Theorem \ref{theorem all between om and N} that $\ca{X} \dleq \ca{N}$ for all recursively presented metric spaces \ca{X}. Since the relation $\dleq$ is antisymmetric and $\ca{C} \dequal \ca{N}$ in order to prove that $\ca{X}$ is \del-isomorphic to \ca{N} it is enough to prove that $\ca{C} \dleq \ca{X}$.

\begin{definition}
\label{definition of del-embedding between trees}\normalfont
Suppose that $S$ and $T$ are trees on \om. We say that the function $f: \om \to \om$ is an
\emph{embedding of $S$ into $T$} if the following conditions hold.
\begin{itemize}
\item[(a)] For all $\dec{t} \in T$ we have that $\dec{f(t)} \in S$.
\item[(b)] If $\dec{t_0}$ and $\dec{t_1}$ are members of $T$ and $\dec{t_1}$ is a proper extension of $\dec{t_0}$ then
$\dec{f(t_1)}$ is a proper extension of $\dec{f(t_0)}$.
\item[(c)] If $\dec{t_0}$ and $\dec{t_1}$ are incompatible members of $T$ then $\dec{f(t_0)}$ and $\dec{f(t_1)}$ are
incompatible.
\end{itemize}
We say that $f$ is a \emph{\del-embedding} if it is an embedding and \del-recursive and that $T$ contains a
\emph{\del \ copy of $S$} or that $S$ is \emph{\del-embedded} into $T$ if there exists a \del-embedding of $S$ into
$T$.
\end{definition}

Every embedding $f$ from $S$ into $T$ gives rise to the continuous injective function
\[
\barr{f}: [S] \to [T]: \barr{f}(\alpha) = \cup_{\n} f(\langle \alpha(0), \dots, \alpha(n-1)\rangle).
\]
If $S$ and $T$ are recursive and $f$ is \del-recursive then $\barr{f}$ is extended to a \del-recursive function $\pi:
\ca{N} \to \ca{N}$. It follows from Lemma \ref{lemma bodyT less than bodyS implies Nt less than Ns} that if $T$ contains
a \del \ copy of $S$ then $\spat{S} \dleq \spat{T}$ for all recursive trees $S$ and $T$. The converse fails in general. To see this consider a Kleene tree $T$ and take $S = T \cup \set{0^{(n)}}{\n}$. This also shows that the converse of Lemma \ref{lemma bodyT less than bodyS implies Nt less than Ns} fails in general as well.

We will prove a partial converse to the previous assertion. To be more specific we will show that if $\spat{2^{< \om}}
\dleq \spat{T}$ then $T$ contains a \del \ copy of $2^{< \om}$. Since $\spat{2^{< \om}}$ is \del \ isomorphic to
$\ca{C}$ the condition  $\spat{2^{< \om}} \dleq \spat{T}$ is equivalent to $\ca{N} \dequal \spat{T}$. Therefore our result provides a characterization of the relation $\ca{N} \dequal \spat{T}$ in terms of the intrinsic
properties of $T$.

We are going to use a measure-theoretic argument. Consider $2 = \{0,1\}$ with the measure $\nu(\{0\}) = \nu(\{1\}) = \frac{1}{2}$ and let $\mu$ be the product measure on the Borel subsets of $2^\om$. The measure $\mu$ is the natural Borel probability measure on $\ca{C}$. It is easy to see that the measure $\mu$ is \emph{non-atomic}, \ie $\mu(\{\alpha\}) = 0$ for all $\alpha \in \ca{C}$. The key tool is the following result of Tanaka.

\begin{theorem}[Tanaka \cf \cite{tanaka_a_basis_result_for_pii_sets_of_positive_measure}]
\label{theorem Tanaka computing the complexity}
For every recursively presented metric space \ca{X} and every \sig \ set $P \subseteq \ca{X} \times \ca{C}$ the set $Q \subseteq \ca{X} \times \om^2$ defined by
\[
Q(x,k,m) \iff \mu(P_x) > m \cdot (k+1)^{-1}
\]
is \sig.\smallskip

In particular if the preceding $P$ is \del \ then the set $Q$ is \del \ as well.
\end{theorem}

Sacks \cite{sacks_measure_theoretic_uniformity} and Tanaka \cite{tanaka_a_basis_result_for_pii_sets_of_positive_measure} have proved the similar version for the classes $\Sigma^0_n$. The corresponding category result for the classes $\Sigma^0_n$ and $\sig$ has been proved by Kechris \cite{kechris_measure_and_category_in_effective_descriptive_set_theory} (see also 4F.19 in \cite{yiannis_dst}).\footnote{Here it is worth noting the following related result of Tanaka \cite{tanaka_a_basis_result_for_pii_sets_of_positive_measure} - Sacks \cite{sacks_measure_theoretic_uniformity}: if \ca{X} is a recursively presented metric space, $\mu$ is a $\sigma$-finite measure on \ca{X} and $P \subseteq \ca{X}$ is \pii \ with $\mu(P) > 0$, then $P$ contains a \del \ member. The category version of this result has been proved by Thomason \cite{thomason_the_forcing_method_and_the_upper_lattice_of_hyperdegrees} - Hinman \cite{hinman_some_applications_of_forcing_to_hierarchy_problems_in_arithmetic}, see also 4F.20 \cite{yiannis_dst}. To illustrate the potential of these results, consider the well-known theorem of Feferman \cite{feferman_some_applications_of_the_notions_of_forcing_and_generic_sets} that there exists some $\alpha \in \del \cap \ca{C}$ such that $\{\alpha\}$ is not $\Pi^0_1$. Let $G \subseteq \om \times \ca{C}$ be $\Pi^0_1$ and universal for $\Pi^0_1 \upharpoonright \ca{C}$ and define the \del \ set
\[
P(e) \iff \mu(G_e) = 0.
\]
The set $M=\cup_{e \in P} G_e$ is also \del \ and has zero-measure. Therefore its complement contains a \del \ member $\alpha$. If it were $\{\alpha\} = G_e$ for some $e$, then $G_e$ would have zero-measure and so $e$ would be a member of $P$. In particular we would have $\alpha \in M$, a contradiction. So $\{\alpha\}$ is not $\Pi^0_1$ and we have proved the preceding result of Feferman.}

\begin{theorem}
\label{theorem characterization of isomorphic to the Baire space}
For every recursive tree $T$ on \om \ the space \spat{T} is \del \ isomorphic to the Baire space if and only if the complete binary tree $2^{< \om}$ is \del \ embedded into $T$.
\end{theorem}

\begin{proof}
If the complete binary tree is \del-embedded into $T$ then from the preceding remarks it follows that $\spat{T}$ is \del-isomorphic to \ca{N}, so we need to prove the converse. Assume that $\spat{T}$ is \del-isomorphic to \ca{N} and let $f: \ca{C} \inj \spat{T}$ be a \del-injection. For all $u \in T$ define
\[
V_u = \set{x \in \spat{T}}{u \sqsubseteq x} \quad \textrm{and} \quad P_u = f^{-1}[V_u].
\]
Clearly we have that
\[
V_u = \cup_{\cn{u}{(n)} \in T} V_{\cn{u}{(n)}} \cup \{u\},
\]
and by taking the inverse images
\begin{align}
\label{equation theorem  characterization of isomorphic to the Baire space A} P_{u} = \cup_{\cn{u}{(n)} \in T} P_{\cn{u}{(n)}} \cup f^{-1}[\{u\}].
\end{align}

The relation $Q \subseteq \om \times \ca{C}$ defined by
\[
Q(t,\alpha) \iff f(\alpha) \in V_{\dec{t}}
\]
(so that the $t$-section of $Q$ is the set $P_{\dec{t}}$) is \del \ since
\[
f(\alpha) \in V_{\dec{t}} \iff d^T(\dense^T(r_t),f(\alpha)) \leq (\lh(\dec{t})+1)^{-1}.
\]
Therefore from Theorem \ref{theorem Tanaka computing the complexity} the relation
\[
R(t) \iff \mu(Q_t) > 0 \iff \mu(P_{\dec{t}}) > 0
\]
is \del \ as well.

\emph{Claim.} For all $u \in T$ with $\mu(P_u) > 0$ there exist $v, w \in T$ incompatible proper extensions of $u$ such that $\mu(P_v), \mu(P_w) > 0$.

\emph{Proof of the Claim.} Suppose that $c=\mu(P_u) > 0$ for some $u \in T$ and assume towards a contradiction that the conclusion of the Claim fails at $u$. Clearly the conclusion fails at every extension of $u$ as well.

Since $f$ is injective and $\mu$ is non-atomic the set $f^{-1}[\{u\}]$ has zero measure. We notice that the quantification $\cn{u}{(n)} \in T$ in the equality (\ref{equation theorem  characterization of isomorphic to the Baire space A})  is not vacuous, \ie $u$ is not terminal in $T$, for otherwise $P_u$ would be the set $f^{-1}[\{u\}]$, which has zero measure contradicting our hypothesis about $P_u$. By applying the measure $\mu$ in (\ref{equation theorem  characterization of isomorphic to the Baire space A}) we have that
\[
\mu(P_u) \leq \sum_{\cn{u}{(n)} \in T} \mu(P_{\cn{u}{(n)}}).
\]
(In fact we have equality above since the $P_{\cn{u}{(n)}}$'s are pairwise disjoint from the injectiveness of $f$. We will only use the inequality though, for reasons that will become clear when we reach Theorem \ref{definition of reasonable measure}.) Using that $\mu(P_u) > 0$ there exists some $n_0$ such that $\cn{u}{(n_0)} \in T$ and $\mu(P_{\cn{u}{(n_0)}}) > 0$.

Since the conclusion fails at $u$ we have in particular that $\mu(P_{\cn{u}{(n)}}) = 0$ for all $n \neq n_0$ for which $\cn{u}{(n)} \in T$ -without of course excluding the possibility that no $n \neq n_0$ with $\cn{u}{(n)} \in T$ exists. In any case we have that
\[
\mu(P_u) \leq \sum_{\cn{u}{(n)} \in T} \mu(P_{\cn{u}{(n)}}) = \mu(P_{\cn{u}{(n_0)}}) \leq \mu(P_u),
\]
hence $\mu(P_{\cn{u}{(n_0)}}) = \mu(P_u) = c$.

Since the conclusion of the Claim fails at every extension of $u$ by proceeding inductively we construct an infinite sequence of natural numbers  $\gamma = (n_k)_{k \in \om}$ such that $\cn{u}{(n_0, \dots,n_k)} \in T$ and $\mu(P_{\cn{u}{(n_0, \dots,n_k)}}) = c$ for all $k$. Since $\cap_k V_{\cn{u}{\gamma \upharpoonright k}} = \{\cn{u}{\gamma}\}$ we have that $\cap_k P_{\cn{u}{\gamma \upharpoonright k}} = f^{-1}[\{\cn{u}{\gamma}\}]$. Using that $f$ is injective it follows that $\mu(\cap_k P_{\cn{u}{\gamma \upharpoonright k}}) = 0$. On the other hand since the sequence of sets $(P_{\cn{u}{\gamma \upharpoonright k}})_{k}$ is obviously decreasing we have that
\[
\mu(\cap_k P_{\cn{u}{\gamma \upharpoonright k}}) = \lim_k \mu(P_{\cn{u}{\gamma \upharpoonright k}}) = c > 0,
\]
a contradiction. This completes the proof of the Claim.

We define the set $A \subseteq \om$ by
\[
t \in A \iff \dec{t} \in T \ \& \ \mu(P_{\dec{t}}) > 0.
\]
The set $A$ is a non-empty \del \ subset of \om. From the Claim we have that
\begin{align*}
t \in A \Longrightarrow& (\exists s, p \in A)[\textrm{$\dec{s}, \dec{p}$ are incompatible proper}\\
                   &\hspace*{19mm}\textrm{extensions of $\dec{t}$ of equal length}]
\end{align*}
see Figure \ref{fig:binarytree}.
\begin{figure}[t]
\begin{picture}(300,160)(0,-5)
\put(100,135){$T$}
\put(139,135){\vdots}
\put(125,115){$u$}
\put(145,115){\Small{$\mu >0$}}
\thicklines
\put(140,130){\vector(0,-1){30}}
\thinlines
\put(140,100){\vector(0,-1){20}}
\put(140,80){\vector(0,-1){20}}
\put(140,60){\vector(0,-1){20}}
\thicklines
\put(140,130){\line(0,-1){90}}
\thinlines
\put(140,100){\vector(1,-1){20}}
\put(140,80){\vector(-1,-1){20}}
\thicklines
\put(140,40){\vector(-1,-1){30}}
\put(140,40){\vector(1,-1){30}}
\thinlines
\put(158,65){\vdots}
\put(165,67){\Small{$\mu = 0$}}
\put(118,45){\vdots}
\put(90,47){\Small{$\mu = 0$}}
\put(108,-5){\vdots}
\put(80,-3){\Small{$\mu >0 $}}
\put(115,25){$v$}
\put(168,-5){\vdots}
\put(175,-3){\Small{$\mu >0 $}}
\put(160,25){$w$}
\end{picture}
\caption{Building the complete binary tree inside $T$.}
\label{fig:binarytree}
\end{figure}
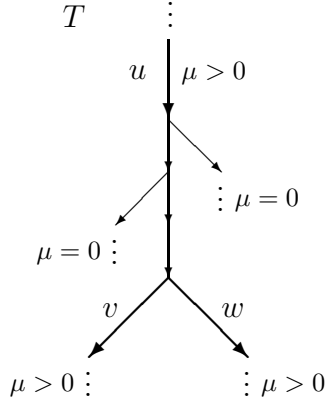

We define the function $f: A \to \om^2$ as follows
\begin{align*}
f(t) =& \ \textrm{the lexicographically least} \ (s,p) \in A^2 \ \textrm{such that \dec{s},
\dec{p} are}\\
      & \ \textrm{proper incompatible extensions of \dec{t} of equal length}.
\end{align*}
We extend $f$ to \om \ by defining $f(t) =(0,0)$ for $t \not \in A$. Since $A$ is \del \ it is not hard to verify that
the function $f$ is \del-recursive. Let $l, r: \om \to \om$ be the \del-recursive functions for which $f =
(l,r)$. We define the function $\varphi: 2^{<\om} \to \om$ by the following recursion
\begin{align*}
\varphi(\emptyset) &= 1 \ \textrm{(the code of the empty sequence)}\\
\varphi(\cn{u}{(0)}) &= l(\varphi(u))\\
\varphi(\cn{u}{(1)}) &= r(\varphi(u))
\end{align*}
and finally
\[
g: \om \to \om: g(t) = \begin{cases}
                       \varphi(\dec{t}), & \ \textrm{if} \ \dec{t} \in 2^{< \om}\\
                       0, & \ \textrm{otherwise}.
                       \end{cases}
\]
Using the closure of \del-recursive functions under primitive recursion (\cf 7A.3 \cite{yiannis_dst}) it is not hard to verify that $g$ is \del-recursive. Moreover it is clear from the definition that $g$ is an embedding of $2^{< \om}$ into $T$.
\end{proof}

The following can be considered as an effective analogue to the statement that every uncountable Polish space contains a homeomorphic copy of $\ca{C}$.

\begin{corollary}
\label{corollary del injection from C to X implies continuous del injection}
For all recursively presented metric spaces \ca{X} if there exists a \del-injection $f: \ca{C} \inj \ca{X}$ then there exists a continuous \del-injection $\tau: \ca{C} \inj \ca{X}$.
\end{corollary}

\begin{proof}
If $\ca{X} = \spat{T}$ for some recursive tree $T$ then the conclusion is immediate from Theorem \ref{theorem characterization of isomorphic to the Baire space}, for if $f: \om \to \om$ is a \del \ embedding of $2^{<\om}$ into $T$ then the function $\barr{f}$ defined before, \ie
\[
\barr{f}: \ca{C} \to [T]: \barr{f}(\alpha) = \cup_{\n} f(\langle \alpha(0), \dots, \alpha(n-1)\rangle)
\]
is injective, continuous and \del-recursive as a function from $\ca{C}$ to \spat{T}.

If \ca{X} is arbitrary we go back to the proof of Theorem \ref{theorem Nt del isomorphism} and we notice that there is
a co-countable \del \ set $C \subseteq \ca{X}$, a recursive (and thus continuous) function $\pi: \ca{N} \to \ca{X}$ and
a recursive tree $T$ such that $\pi$ is injective on $[T]$, $\pi[[T]] = C$ and $\ca{X} \dequal \spat{T}$. From the hypothesis there exists a \del-injection from $\ca{C}$ into \spat{T}. Hence as above there exists a continuous \del-injection $\barr{f}: \ca{C} \to [T]$. The result follows by taking $\tau$ to be
$\pi \circ \barr{f}$.
\end{proof}

\subsection{A characterization in terms of the perfect kernel} In the discussion preceding Theorem \ref{theorem kreisel scattered part} we suggested that it is necessary for the perfect kernel of a recursively presented metric space to contain \del \ points in order for this space to be \del-isomorphic to \ca{N}. With the help of the preceding corollary we can prove this claim and in fact give a characterization of the Baire space up to \del-isomorphism.

\begin{theorem}
\label{theorem characterization of Baire space in terms of perfect sets}
For every recursively presented metric space the following are equivalent.\smallskip

\tu{(1)} \ca{X} is \del-isomorphic to \ca{N}.\smallskip

\tu{(2)} \ca{X} contains a compact perfect $\Pi^0_1$ set $K$ such that $\del \cap K$ is dense in $K$.\smallskip

\tu{(3)} \ca{X} contains a perfect $\Pi^0_1$ set $F$ such that $\del \cap F$ is dense in $F$.\smallskip

\tu{(4)} \ca{X} contains a \del \ set $A$ with no isolated points such that $\del \cap A$ is dense in $A$.
\end{theorem}

If $K, F$ and $A$ are as above then the sets $K,F$ and $\overline{A}$ are contained in the perfect kernel of \ca{X}. Nevertheless it is not true in general that the perfect kernel of \ca{X} is necessarily one of $K,F$ and $\overline{A}$ because it might not be a \del \ set, \cf Remark \ref{remark perfect kernel not in del and isomorphic to Baire space}.

\begin{proof}
If $\ca{X} \dequal \ca{N}$ then from Corollary \ref{corollary del injection from C to X implies continuous del injection} the space \ca{X} contains a homeomorphic \del-copy of \ca{C}, so we proved $(1) \Longrightarrow (2)$. The implications $(2) \Longrightarrow (3) \Longrightarrow (4)$ are clear.

Suppose now that \ca{X} contains a \del \ set $A$ with no isolated points and that $\del \cap A$ is dense in $A$. The relation
\[
R(s) \iff N(\ca{X},s) \cap A \neq \emptyset
\]
is \del, since it is evidently \sig \ and from the density of \del \ in $A$ we have that
\[
R(s) \iff (\exists y \in \del)[ y \in N(\ca{X},s) \cap A].
\]
The closure $\overline{A}$ of $A$ is a \del \ set, since
\[
x \in \overline{A} \iff (\forall s)[x \in N(\ca{X},s) \longrightarrow R(s)].
\]
Hence $\overline{A}$ is a closed \del \ set. Moreover $\overline{A}$ does not have isolated points, for otherwise $A$ would have isolated points as well. It is also clear that \del \ is dense in $\overline{A}$.

From the uniformization property of \pii \ and the fact that \del \ is dense in $\overline{A}$ there exists a \del-recursive sequence $(x_s)_{s \in \om}$ such that $x_s \in N(\ca{X},s) \cap \overline{A}$ for all $s$ for which $N(\ca{X},s) \cap A \neq \emptyset$. It follows that the sequence $(x_s)_{s \in \om}$ is an $\ep$-recursive presentation of the complete metric space $(\overline{A},d)$ for some $\ep$ in \del, where $d$ is the restriction of the metric of \ca{X} on $\overline{A}$.

Summing up $(\overline{A},d)$ is a perfect $\ep$-recursively presented metric space. It follows that $(\overline{A},d)$ is $\del(\ep)$-isomorphic to the Baire space. Let $f: \ca{N} \inj \overline{A}$ be a $\del(\ep)$-injection\footnote{Notice that we do not replace $\del(\ep)$ with \del \ although \ep \ is a \del \ parameter. The reason for this, is because we need to include all parameters with respect to which the spaces admit a recursive presentation.}. The identity function ${\rm id}: \overline{A} \inj \ca{X}$ is easily a $\del(\ep)$-injection, and so the composition ${\rm id} \circ f: \ca{N} \inj \ca{X}$ is a $\del(\ep)$-injection as well. Since $\ep \in \del$ and the spaces \ca{N} and \ca{X} are recursively presented it follows that ${\rm id} \circ f$ is a \del-injection and we are done.
\end{proof}

\subsection{A measure-theoretic characterization} In the proof of Theorem \ref{theorem characterization of isomorphic to the Baire space} the hypothesis about $f$ being injective was used only to infer that the inverse image of a singleton under $f$ has zero measure. Therefore the condition that $f$ is \emph{countable-to-one} (\ie the inverse image of a singleton is a countable set) or even more general \emph{measure-zero-to-one} (\ie the inverse image of a singleton has zero measure) suffices for the conclusion. Notice also that when applying our hypothesis $\ca{C} \dleq \spat{T}$ we did not use any particular intrinsic properties about $\ca{C}$ besides the fact that it is a recursively presented metric space which admits a finite Borel measure satisfying some certain definability properties.

\begin{definition}
\label{definition of reasonable measure}\normalfont
We say that a Borel measure $\mu$ on a recursively presented metric space \ca{X} is \emph{reasonable} if
\begin{itemize}
\item[(a)] $\mu(\ca{X}) \neq 0$,
\item[(b)] $\mu(\{x\}) = 0$ for all $x \in \ca{X}$ and
\item[(c)] for all \del \ sets $P \subseteq \om \times \ca{X}$ the set $R \subseteq \om$ defined by
\[
R(t) \iff \mu(P_t) > 0
\]
is \del.
\end{itemize}
\end{definition}

\begin{theorem}
\label{theorem characterization of isomorphic to the Baire space generalization of}
Suppose that $T$ is a recursive tree, \ca{Y} is a recursively presented metric space and that $\mu$ is a finite reasonable Borel measure on \ca{Y}. If there exists a \del-recursive function $f: \ca{Y} \to \spat{T}$ which is also measure zero-to-one, then $T$ contains a \del \ copy of $2^{<\om}$.
\end{theorem}

\begin{proof}
The same as the proof of Theorem \ref{theorem characterization of isomorphic to the Baire space}.
\end{proof}

\begin{corollary}
\label{corollary measure of the bodyT is zero}
For every recursive tree $T$ exactly one of the following cases is true.\smallskip

\begin{itemize}
\item[(a)] $T$ contains a \del \ copy of $2^{<\om}$.
\item[(b)] \spat{T} does not admit a finite reasonable Borel measure.
\end{itemize}\smallskip

Moreover in the case of \tu{(b)} we have $\mu([T]) = 0$ for all finite reasonable Borel measures $\mu$ on \ca{N}.
\end{corollary}

\begin{proof}
First notice that (a) and (b) are exclusive, for if $T$ contains a \del \ copy of $2^{<\om}$ then $\spat{T}$ is \del-isomorphic to $\ca{C}$ and so it admits a finite reasonable Borel measure. Now if (b) is not true we consider the identity function ${ \rm id}: \spat{T} \inj \spat{T}$ and we apply Theorem \ref{theorem characterization of isomorphic to the Baire space generalization of} to infer that $T$ contains a \del \ copy of $2^{< \om}$.

For the second assertion assume that (b) is true and that $\mu$ is a finite reasonable Borel measure on \ca{N}. Suppose towards a contradiction that $\mu([T]) > 0$. It follows then that the function $\mu_T: \textrm{Borel}(\ca{N}) \to \R: \mu_T(B) = \mu(B \cap [T])$ is a finite reasonable Borel measure on \ca{N}. The function
\[
f: \ca{N} \to \spat{T}: f(\alpha) = \begin{cases}
                                    \alpha, & \ \textrm{if} \ \alpha \in [T]\\
                                    \emptyset, & \textrm{otherwise}
                                    \end{cases}
\]
is \del-recursive and $\mu_T$-measure-zero-to-one. Therefore from Theorem \ref{theorem characterization of isomorphic to the Baire space generalization of} the tree $T$ contains a \del \ copy of $2^{<\om}$, \ie (a) is true in contradiction to our hypothesis.
\end{proof}

\begin{corollary}
\label{corollary Baire space and measure}
The Baire space is the unique up to \del-isomorphism recursively presented metric space which admits a finite reasonable Borel measure.
\end{corollary}

\begin{proof}
Since the Baire space is \del-isomorphic to $\ca{C}$ it is clear that \ca{N} admits a finite reasonable Borel measure. Conversely assume that \ca{X} is a recursively presented metric space, which admits a finite reasonable Borel measure and let $T$ be a recursive tree such that $\ca{X} \dequal \spat{T}$. Then $\spat{T}$ admits a finite reasonable Borel measure and so from Corollary \ref{corollary measure of the bodyT is zero} the tree $T$ contains a \del \ copy of $2^{<\om}$. It follows that $\ca{X} \dequal \spat{T} \dequal \ca{N}$.
\end{proof}

Corollaries \ref{corollary measure of the bodyT is zero} and \ref{corollary Baire space and measure} leave us little hope that any measure-theoretic arguments would help us say something serious about recursively presented metric spaces which are not \del-isomorphic to the Baire space.

We conclude this paragraph with some remarks about isomorphisms on a higher level than that of \del. By standard relativization arguments one can see that for every $\ca{X}$ and every $\ep \in \ca{N}$ if there exists a $\del(\ep)$-injection $\pi: \ca{C} \inj \ca{X}$ then there exists a $\del(\ep)$-isomorphism between \ca{N} and \ca{X}. It turns out that for all recursively presented metric spaces \ca{X} which are uncountable sets there exists a $\del(\W)$-injection from \ca{C} to \ca{X}, so that every uncountable recursively presented metric space is $\del(\W)$-isomorphic to \ca{N}.\footnote{This result makes it perhaps tempting to assume that if \ca{X} is uncountable $\ep$-recursively presented with $\W \hleq \ep$, then \ca{X} is \del(\ep)-isomorphic to the Baire space, and so the problem of \del-isomorphism trivializes for this type of spaces \ca{X}. However the latter is not correct and one can see this by taking a \W-Kleene tree and repeating the proof of Theorem \ref{theorem counterexample to del with Baire (Kleene space)}. The correct relativized version is that every uncountable $\ep$-recursively presented metric space is $\del(\W^\ep)$-isomorphic to the Baire space. (For the definition of the relativized $\W^\ep$ refer to the proof of Theorem \ref{theorem strong form spector-gandy general}.)} In fact we will prove a mild strengthening of this result in Theorem \ref{theorem NT is del(gamma) isomorphic with gamma less than W}. This implies that every uncountable recursively presented metric space is $\Delta^1_2$-isomorphic to the Baire space. Therefore all results about subsets of the Baire space from the level of $\Delta^1_2$ and above are carried to subsets of uncountable recursively presented metric spaces.

\subsection{The tree of attempted embeddings} We introduce an operation on the space of trees which gives some interesting results about \del-isomorphisms.

\begin{definition}
\label{definition tree of attempts}\normalfont
We fix the following recursive enumeration of $2^{< \om}$, $$s_0 = \empt, s_1 = (0), s_2 = (1), s_3 = (0,0), s_4 = (0,1), s_5=(1,0)\dots$$ and we define the function $\att: \Tr \to \Tr$ by saying that $w \in \att{T}$ exactly when for all $n, m < \lh(w)$ we have that $w(n) \in \Seq$, $\dec{w(n)} \in T$ and the following implications are satisfied:
\begin{align*}
\mbox{$s_n$ is incompatible with $s_m$} &\Longrightarrow \\& \hspace*{-10mm}\mbox{\dec{w(n)} is incompatible with \dec{w(m)}}\\
\mbox{$s_n$ is proper initial segment of $s_m$} &\Longrightarrow \\& \hspace*{-30mm} \mbox{\dec{w(n)} is a proper initial segment of \dec{w(m)}}.
\end{align*}
for all $w \in \omseq$ and for all $T \in \Tr$.  In other words \att{T} encodes the set of all attempts to get an embedding of the complete binary tree into $T$. It is easy to see that the function $\att$ is recursive so that if $T$ is a recursive tree then \att{T} is recursive as well.
\end{definition}

\begin{remark}
\label{remark tree of attempts}\normalfont
Every $\alpha \in [\att{T}]$ gives rise to an $\alpha$-recursive injection from the Cantor space into \spat{T} which in fact takes all of its values inside $[T]$. Indeed suppose that $\alpha \in [\att{T}]$ and $\gamma \in \ca{C}$. Choose the unique $\gamma$-recursive sequence $(k^\gamma_n)_{\n}$ of naturals such that $\gamma \upharpoonright n = s_{k^\gamma_n}$ for all \n. It follows that $\dec{\alpha(k^\gamma_n)}$ is a proper initial segment of $\dec{\alpha(k^\gamma_{n'})}$ for all $n < n'$. Thus $\cup_{\n} \dec{\alpha(k^\gamma_n)} \in [T]$. Define $$\sigma: \ca{C} \to \spat{T}: \sigma(\gamma) = \cup_{\n} \dec{\alpha(k^\gamma_n)},$$ where $(k^\gamma_n)_{\n}$ is as above. Then the function $\sigma$ is $\alpha$-recursive and injective.

If $[T]$ is uncountable then $[\att{T}] \neq \emptyset$. To see this we notice that if $[T_u]$ is uncountable then it contains two distinct infinite branches which belong to its perfect kernel. So there exist incomparable extensions $v, w$ of $u$ such that $[T_v]$ and $[T_w]$ are uncountable. Summing up we conclude to the following.
\end{remark}

\begin{lemma}
\label{lemma attT is a Kleene tree}
If $T$ is a recursive tree such that \spat{T} is uncountable and not \del-isomorphic to \ca{N} then $\att{T}$ is a Kleene tree.
\end{lemma}

From the Kleene Basis Theorem and the preceding remarks it follows that every uncountable recursively presented metric space is \del(\W)-isomorphic to \ca{N}. Using the Gandy Basis Theorem we may improve upon the parameter.

\begin{theorem}
\label{theorem NT is del(gamma) isomorphic with gamma less than W}
Every uncountable recursively presented metric space is $\del(\gamma)$-isomorphic to \ca{N} for some $\gamma \hless \W$.
\end{theorem}

\begin{proof}
Suppose that \ca{X} is a uncountable recursively presented metric space. From Theorem \ref{theorem Nt del isomorphism} we may assume that $\ca{X} = \spat{T}$ for some recursive tree $T$. From the Gandy Basis Theorem there exists some $\gamma \in [\att{T}]$ such that $\gamma \hless \W$.
\end{proof}

What is perhaps more interesting about the operation $\att{}$ is that under some conditions it provides a method for producing non-\del-isomorphic spaces.

\begin{theorem}
\label{theorem about tree of attempts}
Suppose that $T$ and $S$ are recursive trees and that $\gamma$ is in $[T] \setminus \del$. If the set $ \del(\gamma) \cap [S]$ is $\del(\gamma)$ then
\[
\spat{T} \not \dleq \spat{\att{S}}.
\]
In particular for every Kleene tree $T$ there exists some $u \in T$ with $[T_u] \neq \emptyset$ and
\[
\spat{T} \not \dleq \spat{\att{T_u}}.
\]
\end{theorem}

\begin{proof}
Let $T$, $S$ and $\gamma \in [T]$ be as in the statement and assume towards a contradiction that there is a \del-injection $\pi: \spat{T} \to \spat{\att{S}}$. From Lemma \ref{lemma the inverse function is del and same hyperdegree} we have that $\pi(\gamma) \heq \gamma$. So if it were $\pi(\gamma) \in \att{S}$ then $\gamma$ would be \del, contradicting the hypothesis about $\gamma$. So $\pi(\gamma) \in [\att{S}]$.

Since $\pi(\gamma) \in \del(\gamma)$ from Remark \ref{remark tree of attempts} there exists a $\del(\gamma)$-injection $\tau: \ca{C} \to \spat{S}$. Therefore $\spat{S}$ is $\del(\gamma)$-isomorphic to the Baire space and the set $\del(\gamma) \cap \spat{S}$ is not $\del(\gamma)$. On the other hand $\del(\gamma) \cap \spat{S} = (\del(\gamma) \cap [S]) \cup S$ and from our hypothesis the set $\del(\gamma) \cap [S]$ is $\del(\gamma)$, a contradiction.

Regarding the second assertion we apply the Gandy Basis Theorem to find some $\gamma \in [T]$ with $\W \not \in \del(\gamma)$. From Corollary \ref{corollary del of gamma is dense in a Kleene tree} there exists some $u \in T$ such that $[T_u] \neq \emptyset$ and $\del(\gamma) \cap [T_u] = \emptyset$, so in particular the set $\del(\gamma) \cap [T_u]$ is $\del(\gamma)$. Finally $\gamma$ is not a \del \ point since $T$ is a Kleene tree.
\end{proof}

We have already seen in Theorem \ref{theorem for every kleene tree there exists another one incomparable and incomparable Kleene spaces} that for every Kleene tree $T$ there exists a Kleene tree $S$ such that the spaces \spat{T} and \spat{S} are \dleq-incomparable. Therefore the question is whether one can apply the preceding theorem to non-Kleene spaces. We refer to Section \ref{section questions open problems} for further discussion.

\section{\sc Spector-Gandy spaces}

\label{section SG spaces}

So far we have been rejecting the existence of \del \ injections $f: \ca{N} \inj \ca{X}$ on the grounds that the set of all \del \ points of \ca{X} is \del. Here we define another category of spaces, where the set of all \del \ points is not \del \ and still they are not \del-isomorphic to the Baire space.\footnote{So the plausible conjecture $\ca{X} \dequal \ca{N} \iff \del \cap \ca{X} \not \in \del$ is not correct.}

The idea behind these spaces can be traced back to a construction of Kreisel \cite{kreisel_cantor_bendixson} of a $\Pi^0_1$ subset of the reals whose perfect kernel is a $\sig \setminus \del$ set. Although our motivation was to turn the $\Pi^0_1$ set of Kreisel's example into a recursively presented metric space, it emerged that the construction yields a whole new category of spaces. The main tool is a theorem of Spector and Gandy, after which these spaces are named.

\subsection{The Spector-Gandy Theorem} Spector and independently Gandy showed that \pii \ subsets of \om \ can be fully described in terms of \del \ points $\alpha \in \ca{N}$. This result comes in two forms: the \emph{weak form} \cite{gandy_proof_of_mostowskis_conjecture} and the \emph{strong form} \cite{spector_hyperarithmetical_quantifiers}. The weak form can be extended to \pii \ subsets of recursively presented metric spaces \cf 4F.3  in \cite{yiannis_dst}. We will show that the latter is true for the strong form as well.

\begin{theorem}[Strong Form of the Spector-Gandy Theorem]
\label{theorem Spector-Gandy Theorem}
For every \pii \ set $P \subseteq \om$ there exists a recursive $R \subseteq \om^2$ such that
\begin{align*}
P(n) \iff& (\exists \alpha \in \del)(\forall t)R(n,\barr{\alpha}(t))\\
        \iff& (\exists! \alpha)(\forall t)R(n,\barr{\alpha}(t)).
\end{align*}
\end{theorem}
The weak form of the Spector-Gandy Theorem states only the first of the preceding equivalences. It is usually this weak form which finds its way to the bibliography, see for example \cite{sacks_higher_recursion_theory}, \cite{chong_yu_recursion_theory}, \cite{yiannis_elementary_induction_on_abstract_structures} and \cite{yiannis_dst}, although in the latter the strong form is mentioned in the case of $\del$ sets. Spector's proof of the strong form uses some advanced considerations on Kleene's $\mathit{O}$ including some results of Markwald \cite{markwald_zur_theorie_der_konstruktiven_wohlordnungen}. Here we present a milder proof, which uses only elementary properties of \W. Moreover our proof can be carried out in all recursively presented metric spaces in a straightforward way.\footnote{It seems that Spector's proof can also be carried out in all recursively presented metric spaces but with a significantly more effort than carrying out our proof.} Let us establish the necessary terminology.

For every linear ordering $\preceq$ on a subset of \om \ we define the partial function $$\suc_{\preceq}: \om \to \om:$$
\begin{align*}
\suc_{\preceq}(n) \downarrow &\iff (\exists k)[n \preceq k \ \& \ (\forall m)[m \preceq n \vee k \preceq m]]\\
\suc_{\preceq}(n) \downarrow &\Longrightarrow \suc_{\preceq}(n) = \textrm{the unique $k$ as above},
\end{align*}
\ie $\suc_{\preceq}(n)$ \ is the $\preceq$-successor of $n$ whenever it exists.

Suppose that $\lin$ and $\preceq$ are two linear orderings on a subset of \om \ with least elements. We say that a partial function $f: \om \rightharpoonup \om$ is \emph{$(\lin,\preceq)$-admissible} if the following conditions hold:
\begin{align}
\label{equation admissible A} & \Gr(f) \subseteq \Field(\lin) \times \Field(\preceq)\\
\label{equation admissible B} & (\forall n \in \Dom(f))(\forall n')[n' \lin n \longrightarrow n'\in\Dom(f)]\\
\label{equation admissible C} & (\forall n,n'\in \Dom(f))[n' \lin n \longleftrightarrow f(n') \preceq f(n)]\\
\label{equation admissible D} & \textstyle (\forall n \in \Dom(f))[f(n) = \sup_{\preceq}\set{\suc_{\preceq}(f(n'))}{n' \slin n}].
\end{align}
By $\sup_{\preceq}\emptyset$ we mean the least element of $\preceq$. When the linear orderings are clear from the context we just say admissible instead of $(\lin,\preceq)$-admissible. An admissible function $f$ is \emph{strongly admissible} if for all $n \not \in \Dom(f)$ and for all $m$ the function with graph $\Gr(f) \cup \{(n,m)\}$ is not admissible.\footnote{It is possible for a strongly admissible function to admit an admissible extension. This is the reason why we prefer the term ``strongly" instead of ``maxinal". To see this consider the space $X = (\om,\leq) + (\om, \geq)$. Then the identity function from $X$ to $X$ as well as its restriction on $(\om,\leq)$ are both strongly admissible functions.}

It is clear that if $\lin$ and $\preceq$ are well-orderings and the order type of $\lin$ is less than or equal to the order type of $\preceq$ then the unique embedding of $\Field(\lin)$ into an initial segment of $\Field(\preceq)$ is the unique strongly $(\lin,\preceq)$-admissible function. Moreover the latter function is total and so it does not admit a proper admissible extension, \ie it is \emph{maximal admissible}. This fact is also true even if $\lin$ is not a well-ordering, as long as it has a least element.

\begin{lemma}
\label{lemma unique strongly admissible}
For every linear ordering $\lin$ with a least element and every well-ordering $\preceq$, there exists a unique strongly $(\lin,\preceq)$-admissible function. \tu{(}All orderings are assumed to be on subsets of \om.\tu{)}\smallskip

In fact the unique strongly $(\lin,\preceq)$-admissible function is also maximal admissible.
\end{lemma}

\begin{proof}
Let us take the case where the well-founded part $\WF(\lin)$ of \lin \ has order type less than or equal to the order type of $\preceq$, and let $$f: \WF(\lin) \inj \Field(\preceq)$$ be the unique embedding onto an initial segment of $\preceq$. We view $f$ as a partial function on $\Field(\lin)$. It is easy to verify that $f$ is admissible. (Notice that $\sup_\preceq A$ exists for every $\preceq$-bounded set $A$, since $\preceq$ is a well-ordering.)

Moreover if $f': \Field(\lin) \rightharpoonup \Field(\preceq)$ is a proper extension of $f$ then there exists some $n$ in the ill-founded part of $\lin$ which is mapped by $f'$ to the field of $\preceq$. Since $\preceq$ is a well-ordering, the function $f'$ cannot be defined on an initial segment of $\lin$ and preserve the strict orders at the same time. So $f'$ cannot satisfy both conditions (\ref{equation admissible B}) and (\ref{equation admissible C}) and in particular is not admissible.

Now let us take the case where the order type of $\WF(\lin)$ is greater than the order type of $\preceq$. Then there exists (a unique) $N \in \WF(\lin)$ such that the order type of $\iniseg_{\lin}(N)$ is equal to the order type of $\preceq$. As above we consider the unique isomorphism $f$ from $\iniseg_{\lin}(N)$ to $\Field(\preceq)$ and we view $f$ as a partial function on $\Field(\lin)$. The function $f$ is admissible and since it is surjective it does not admit a proper injective extension. In particular no proper extension $f'$ of $f$ preserves the strict orders and so
$f'$ cannot be admissible.
\end{proof}

\begin{remark}\normalfont
\label{remark maximal admissible function}
There might be infinitely many maximal -and hence strongly- admissible functions. Moreover it is possible that no \del \ strongly admissible function exists. To see this consider two recursive pseudo-well-orderings $\lin, \preceq$ of the same order type $\om_1^{\rm CK} +  \mathbb{Q} \times \om_1^{\rm CK}$, which are \del-incomparable, (Kreisel, Harrison). It is clear that every isomorphism $f: (\mathbb{Q},\leq_{\mathbb{Q}}) \bij (\mathbb{Q},\leq_{\mathbb{Q}})$ defines an isomorphism $f^ \ast$ between $\lin$ and $\preceq$, which also satisfies the property (\ref{equation admissible D}) of admissibility. Therefore $f^\ast$ is a maximal admissible function. It is also clear that $f^\ast$ determines $f$, so there are infinitely many isomorphisms between the two linear orderings.

We now prove that there cannot exist a \del \ strongly $(\lin,\preceq)$-admissible function. To see this assume that $g: \Field(\lin) \rightharpoonup \Field(\preceq)$ is strongly admissible. If $g$ is total then it cannot be \del \ since the orderings are \del-incomparable. If $g$ is surjective then $g$ cannot be \del \ as well, for otherwise from Lemma \ref{lemma the inverse function is del and same hyperdegree} the inverse $g^{-1}$ would be a total \del-embedding of $\preceq$ into \lin. Thus we may assume that $g$ is neither total nor surjective.

Now we claim that the complement of $\Dom(g)$ in $\Field(\lin)$ does not have a \lin-least element or that the complement of the image of $g$ in $\Field(\preceq)$ does not have a $\preceq$-least element. Indeed if this were not the case and $n \in \Field(\lin), m \in \Field(\preceq)$ are the $\lin,\preceq$-least such naturals then it is not hard to verify that
\[
\textstyle m = \sup_{\preceq} \set{\suc_{\preceq}(f(n'))}{n' \lin n}.
\]
(Notice that every point in the field of $\preceq$ has a successor.) Therefore the function with graph $\Gr(g) \cup \{(n,m)\}$ would be admissible, contradicting the fact that $g$ is strongly admissible.

If the complement of the domain of $g$ does not have a \lin-least element, since the latter is a $\del(g)$ set, one can find a $\del(g)$ strictly decreasing sequence in $\Field(\lin) \setminus \Dom(g)$. Using that $\lin$ is a recursive pseudo-well-ordering, $g$ cannot be \del. Similarly if the complement of the image of $g$ does not have a $\preceq$-least element, then there exists a $\del(g)$ strictly decreasing sequence in $\Field(\preceq)$, and so $g$ cannot be \del \ as well.
\end{remark}

\begin{proof}[\textit{Proof of Theorem \ref{theorem Spector-Gandy Theorem}.}]
Since \W \ is a complete \pii \ set it is enough to prove the claim for $P = \W$. Moreover it is enough to show the conclusion having an arithmetical set $S \subseteq \om \times \ca{N}$ instead of a $\Pi^0_1$ set $C$. One can see the latter by standard unfolding arguments or by using Theorem \ref{theorem alltogether}-(\ref{theorem 4D.9}) -see also the following proof.

We fix for the rest of the proof a Kleene tree $K$. We consider the Kleene-Brouwer ordering $\leq_{\rm KB}$ on $K$. From Theorem \ref{theorem Gandy length of well-founded part} the order type of the well-founded part of $K$ with respect to $\leq_{\rm KB}$ is $\ck$.

The idea is to define $S(e,\beta)$ so that $\beta$ encodes an embedding from $\leq_e$ onto an initial segment of $K$. Therefore when $\leq_e$ is a well-ordering then there exists a unique (and necessarily \del) embedding $\beta$ such that $S(e,\beta)$. However this will not be enough to give us the converse direction, for $\leq_e$ might be the ordering of $\iniseg(u)$ for some $u$ in the ill-founded part of $K$ and $\beta$ is just the encoding of the identity function, which is of course \del. We will overcome this by adding more information on $S$, namely some point $\alpha \in \ca{N}$ so that $(\alpha)_n$ encodes a partial embedding from $\leq_n$ into $\leq_e$ in such a way that when $\leq_e$ is a well-ordering then $(\alpha)_n$ is the unique such embedding -this is the idea behind the strongly admissible functions. Moreover when $\leq_e$ is not a well-ordering then $\alpha$ contains all the information of recursive well-orderings and thus cannot be \del. Now let us make all of these precise.

We define the relation $Q \subseteq \om \times \om \times \ca{N}$ by
\begin{align*}
Q(n,e,\alpha) \iff& \{n\}, \{e\} \ \textrm{are linear orderings with least elements}\\
                  &\textrm{and the set}\ \set{(i,j)}{\alpha(\langle i,j \rangle) = 1} \ \textrm{is the graph}\\
                  & \hspace{5mm}\textrm{of a strongly} \ (\leq_n,\leq_e)\textrm{-admissible function}\\
                  & \textrm{and} \ \alpha(t) = 0 \ \textrm{for all $t$ with $t \neq \langle i,j \rangle$ for all $i,j$}.
\end{align*}
It is not hard to verify that $Q$ is arithmetical. From Lemma \ref{lemma unique strongly admissible} we have that for all $e \in \W$ and for all linear orderings $\{n\}$ with least element there exists exactly one $\alpha \in \ca{N}$ such that $Q(n,e,\alpha)$.

Now we define the relation $L \subseteq \om \times \ca{N}$
\begin{align*}
L(e,\alpha) \iff& \ \ \ \ \{e\} \ \textrm{is a linear ordering with a least element}\\
                  &\& \ \big \{(\forall n)[\textrm{$\{n\}$ is a linear ordering with a least element}\\
                  & \hspace*{2cm} \textrm{and} \ Q(n,e,(\alpha)_n)] \ \textrm{or}\\
                  & \hspace*{5mm} [\{n\} \ \textrm{is not a linear ordering with a least element}\\
                  & \hspace*{2cm} \textrm{and} \ (\alpha)_n = 0]\ \big \}\\
                  &\& \ \ \alpha(t) = 0 \ \textrm{for all $t$ with $t \neq \langle n,s \rangle$ for all $n,s$}.
\end{align*}
Using the preceding comments it is clear that for all $e \in \W$ there exists exactly one $\alpha$ such that $L(e,\alpha)$.

We define $S \subseteq \om \times \ca{N}^2$ as follows
\begin{align*}
S(e,\alpha,\beta) \iff& \ \ \big \{ \{e\} \ \textrm{is a linear ordering and $\beta$ embeds} \ \leq_e \ \textrm{into}\\
                      & \hspace*{5mm} \textrm{an initial segment of} \ (K,\leq_{\rm KB}) \ \textrm{and $\beta$ is $0$ outside}\\
                      & \hspace*{5mm} \textrm{of the domain of} \ \leq_e \big \}\\
                      &\& \ \  L(e,\alpha)
\end{align*}
It is not hard to verify that $S$ is arithmetical. We claim that
\begin{align*}
e \in \W &\iff (\exists ! (\alpha,\beta))S(e,\alpha,\beta)\\
         &\iff (\exists (\alpha,\beta) \in \del)S(e,\alpha,\beta)
\end{align*}
for all $e \in \W$. Once we prove this we are done since $\ca{N} \times \ca{N}$ is recursively isomorphic to \ca{N}.

We prove these equivalences round-robin style. Suppose that $e$ is in $\W$. Since the order type of the well-founded part of $(K,\leq_{\rm KB})$ is equal to $\om_1^{\rm CK}$ there exists a unique $\beta$ which embeds $\leq_e$ into an initial segment of $(K,\leq_{\rm KB})$ and is $0$ outside of the domain of $\leq_e$. Moreover from the preceding remarks there exists exactly one $\alpha$ such that $L(e,\alpha)$. So $(\alpha,\beta)$ is the unique pair for which we have that $S(e,\alpha,\beta)$.

If $S(e,\alpha,\beta)$ for a unique pair $(\alpha,\beta)$ then, since $S$ is arithmetical, it is clear from the Effective Perfect Set Theorem \ref{theorem alltogether}-(\ref{theorem 4F.1}) that $(\alpha,\beta)$ is \del.

Now suppose that $S(e,\alpha,\beta)$ holds for some $(\alpha,\beta) \in \del$ and assume towards a contradiction that $e \not \in \W$. Since $\leq_e$ is isomorphic to an initial segment of $(K,\leq_{\rm KB})$ and is not a well-ordering, then the well-founded part of $\leq_e$ has the same order type as the one of $(K,\leq_{\rm KB})$, \ie $\om_1^{\rm CK}$.

Suppose that $A \subseteq \om$ is an arbitrary \del \ set. It is well known that there exists a recursive function $f: \om \to \om$ such that $\leq_{f(n)}$ is a linear ordering and that
\[
n \in A \iff f(n) \in \W
\]
for all \n. (This is true for all \pii \ subsets of \om, see the following proof or 4A.3 in \cite{yiannis_dst}.) Moreover it is easy to arrange that $\leq_{f(n)}$ has a least element, in fact this is implicit from the proof of this statement.

For every \n, since $\leq_{f(n)}$ is a linear ordering with a least element and $L(e,\alpha)$ it follows that the set
\[
\set{(i,j)}{(\alpha)_{f(n)}(\langle i,j\rangle) =1}
\]
is the graph of a strongly $(\leq_{f(n)},\leq_e)$-admissible function, which we denote by $\pi_n$.

Set $\xi = \sup\set{|\leq_{f(n)}|}{n \in A}$. Since $A$ is in \del \ we have that $\xi < \om_1^{\rm CK}$, and since the well-founded part of $\leq_e$ is $\om_1^{\rm CK}$ there exists some $k$ in the well-founded part of $\leq_e$ whose order type with respect to $\leq_e$ is $\xi$. We claim that
\begin{align*}
f(n) \in \W \iff (\forall i \in \Field(\leq_{f(n)}))[\pi_n(i) \downarrow \ \& \ \pi_n(i) \leq_e k],
\end{align*}
for all \n. To see this assume first that $f(n) \in \W$. Since $\leq_{f(n)}$ is a well-ordering the function $\pi_n$ is the unique embedding of $\leq_{f(n)}$ into an initial segment of $\leq_e$. In particular $\pi_n$ is defined on the whole $\Field(\leq_{f(n)})$. Moreover from the choice of $k$ we have that $\pi_n(i) \leq_e k$ for all $i \in \Field(\leq_{f(n)})$. Conversely suppose that $\pi_n(i)$ is defined and that $\pi_n(i) \leq_e k$ for all $i \in \Field(\leq_{f(n)})$. If $f(n)$ is not in $\W$ then there exists a sequence $(i_m)_{m \in \om}$ with $i_{m+1} <_{f(n)} i_m$ for all $m \in \om$. Since $\pi_n$ is $(\leq_{f(n)},\leq_e)$-admissible which is defined on the whole field of $\leq_{f(n)}$ it follows from property (\ref{equation admissible C}) that $\pi_n(i_{m+1}) <_e  \pi_n(i_m) \leq_e k$ for all $m \in \om$, contradicting that $k$ is in the well-founded part of $\leq_e$. Hence $f(n)$ is in $\W$ and the preceding equivalence is proved.

Summing up we have that
\begin{align*}
n \in A \iff& f(n) \in \W\\
        \iff& (\forall i \in \Field(\leq_{f(n)}))[\pi_n(i) \downarrow \ \& \ \pi_n(i) \leq_e k]\\
        \iff& (\forall i)\big \{[i \leq_{f(n)} i]  \longrightarrow (\exists j)[(\alpha)_{f(n)}(\langle i,j\rangle) = 1 \ \& \ j \leq_e k]  \big\},
\end{align*}
for all \n. This shows that $A$ is $\Pi^0_2(\alpha)$. Since $A$ is an arbitrary \del \ subset of $\om$ it follows that $\alpha$ is not in \del, a contradiction. Therefore $e$ is in $\W$ and we are done.
\end{proof}

\begin{theorem}[The strong form of the Spector-Gandy Theorem for Polish spaces]
\label{theorem strong form spector-gandy general}
For every recursively presented metric space \ca{X} and every \pii \ set $P \subseteq \ca{X}$ there exists a $\Pi^0_1$ set $C \subseteq \ca{X} \times \ca{N}$ such that
\begin{align*}
P(x) \iff& (\exists ! \alpha)C(x,\alpha)\\
       \iff& (\exists \alpha \in \del(x))C(x,\alpha),
\end{align*}
for all $x \in \ca{X}$.\footnote{It is worth pointing out the following analogy with classical theory. Lusin (see \cite{kechris_classical_dst} Theorem 18.11) has proved that the set of all \emph{unicity} points of a Borel set $B \subseteq \ca{X} \times \ca{Y}$, \ie the set $\set{x\in \ca{X}}{(\exists ! y)B(x,y)}$, is coanalytic. The converse is also true, \ie every coanalytic set is the set of all unicity points of a closed set, \cf \cite{kechris_classical_dst} Exercise 18.13. The strong form of the Spector-Gandy Theorem for Polish spaces gives a straightforward proof of the latter result and therefore it can be considered (at least in part) as its effective analogue.}
\end{theorem}

\begin{proof}
We first prove the result for the Baire space. For the needs of the proof we need to consider the relativized version of Spector's \W. For every $x \in \ca{N}$ we denote by $\W^x$ the set which is obtained by replacing the term ``recursive" with ``recursive in $x$" in the definition of \W, so that
\[
e \in \W^x \iff \leq_e \ \textrm{is an $x$-recursive well-ordering}.
\]
Now let $P$ be a \pii \ subset of \ca{N}. Following the proof of the basic representation theorem for \pii \ (\cf 4A.1 and 4A.4 \cite{yiannis_dst}) we put down $P$ in the form
\[
P(x) \iff (\forall \alpha)(\exists t)R(x,\barr{\alpha}(t)),
\]
where $R$ is a recursive\footnote{In the case of an arbitrary space \ca{X} the set $R$ would be semirecursive but not necessarily recursive. One could still proceed from this, but they would need to make more modifications on \W.} subset of $\ca{N} \times \om$. We define
\[
u \in T(x) \iff (\forall t < \lh(u))\neg R(x,\langle u(0),\dots,u(t-1) \rangle)
\]
so that $T(x)$ is a tree and
\[
P(x) \iff T(x) \ \textrm{is well-founded},
\]
for all $x \in \ca{X}$. It is easy to find a recursive function $f: \ca{N} \to \om$ such that
\[
\leq_{f(x)} = \textrm{is the linear ordering \kbleq \ on} \ T(x)
\]
and therefore
\[
P(x) \iff f(x) \in \W^x
\]
for all $x \in \ca{N}$.

Now let us take a $\Pi^0_1$ set $F \subseteq \ca{N} \times \ca{N}$ as in Theorem \ref{theorem Gandy extension of Kleene}, \ie every section $F_x$ is non-empty and has no members in $\del(x)$. From this it follows that there exists a recursive relation $K \subseteq \ca{N} \times \Seq$ such that $K(x)$ is a tree and $[K(x)] = F_x$ for all $x \in \ca{N}$. By standard relativization arguments the order type of the well-founded part of $K(x)$ is $\ckr{x}$.

We go back to the preceding proof and we relativize the definition of $L$ and $S$ with respect to $x$. (In particular we replace every instance of the Kleene tree $K$ with the $K(x)$ from above.) Let $S^\ast(e,x,\alpha,\beta)$ be the arithmetical relation which corresponds to $S(e,\alpha,\beta)$. With the same arguments we have that
\begin{align*}
e \in \W^x &\iff (\exists ! (\alpha,\beta))S^\ast(e,x,\alpha,\beta)\\
                 &\iff (\exists (\alpha,\beta) \in \del(x))S^\ast(e,x,\alpha,\beta),
\end{align*}
so that
\begin{align*}
P(x) \iff& f(x) \in \W^x\\
       \iff& (\exists ! (\alpha,\beta))S^\ast(f(x),x,\alpha,\beta)\\
        \iff& (\exists (\alpha,\beta) \in \del(x))S^\ast(f(x),x,\alpha,\beta).
\end{align*}
Therefore we take
\[
S^{\ast \ast}(x,\alpha,\beta) \iff S^\ast(f(x),x,\alpha,\beta)
\]
Clearly $S^{\ast \ast}$ is arithmetical and hence is the recursive injective image of a $\Pi^0_1$ subset of \ca{N} (Theorem \ref{theorem alltogether}-(\ref{theorem 4D.9})). It is not hard to verify that the graph of a recursive function is a $\Pi^0_1$ set, so there exists $\Pi^0_1$ set $C$ such that
\begin{align*}
S^{\ast \ast}(x,\alpha,\beta) \iff& (\exists \gamma)C(x,\alpha,\beta,\gamma)\\
                              \iff& (\exists! \gamma)C(x,\alpha,\beta,\gamma).
\end{align*}
One can then verify that $C$ is the required set. This finishes the proof for the Baire space.

In the general case we have from Theorem \ref{theorem alltogether}-(\ref{theorem 3E.6}) a recursive surjection $\pi: \ca{N} \surj \ca{X}$ and a $\Pi^0_1$ set $A \subseteq \ca{X}$ such that $\pi$ is injective on $A$ and $\pi[A] = \ca{X}$. The set $Q=\pi^{-1}[P] \cap A$ is a \pii \ subset of \ca{N}, so from above there exists a $\Pi^0_1$ set $C$ such that
\begin{align*}
Q(\ep) \iff& (\exists ! \alpha)C(\ep,\alpha)\\
       \iff& (\exists \alpha \in \del(\ep))C(\ep,\alpha)
\end{align*}
for all $\ep \in \ca{N}$. We define
\[
C'(x,\ep,\alpha) \iff C(\ep,\alpha) \ \& \ \pi(\ep) = x.
\]
It is now easy to verify that
\begin{align*}
P(x) \iff& (\exists ! (\ep,\alpha))C'(x,\ep,\alpha)\\
       \iff& (\exists (\ep,\alpha) \in \del(x))C'(x,\ep,\alpha).
\end{align*}
for all $x \in \ca{X}$. (Notice that if $\pi(\ep) = x$ with $\ep \in A$ then $\ep \heq x$. So if $(\ep,\alpha)$ is in $\del(x)$ and $C'(x,\ep,\alpha)$ holds, then we have in particular that $\alpha$ is in \del(\ep).)
\end{proof}

\subsection{Application to the spaces \spat{T}} Suppose that $P$ is a \pii \ subset of \om \ and that $R$ is a recursive set such that
\begin{align*}
P(n) \iff& (\exists \alpha \in \del)(\forall t)R(n,\barr{\alpha}(t))\\
     \iff& (\exists! \alpha)(\forall t)R(n,\barr{\alpha}(t)).
\end{align*}
We define the recursive tree $T$ by
\[
u \in T \iff u = \empt \ \vee \ (\forall t < \lh(u))R(u(0),\langle u(1),\dots,u(t)\rangle),
\]
so that
\begin{align*}
P(n) \iff& (\exists \alpha \in \del)[\cn{(n)}{\alpha} \in [T]]\\
     \iff& (\exists! \alpha)[\cn{(n)}{\alpha} \in [T]].
\end{align*}
It follows that $[T_{(n)}]$ is a singleton for $n \in P$. Moreover if $x$ is in $[T]$ and $x(0) \not \in P$ then $x$ is not a \del \ point. Hence the set of all \del-points of the space \spat{T} is $\set{x \in [T]}{x(0) \in P} \cup T$. Also notice that we cannot reject the possibility that \spat{T} is countable. In fact if $P$ is a recursive set one can easily choose $R$ so that if $n \not \in P$ then $[T_{(n)}] = \emptyset$. So in order for this space to be more interesting we will ask that $P$ is a $\pii \setminus \del$ set.

\begin{definition}
\label{definition of Spector-Gandy tree and space}
\normalfont
Suppose that $T$ is a recursive tree on \om \ and consider the set
\[
P(n) \iff (\exists x \in \del)[ x \in [T] \ \& \ x(0)=n].
\]
Clearly from the Theorem on Restricted Quantification the set $P$ is \pii. We say that $T$ is a \emph{Spector-Gandy tree} if $P$ is not \del \ and
\[
P(n) \iff (\exists! x)[ x \in [T] \ \& \ x(0)=n]
\]
for all \n. A space of the form \spat{T} is a \emph{Spector-Gandy space} if $T$ is a Spector-Gandy tree. The set $P$ as above is the \emph{companion set} of the Spector-Gandy space \spat{T} (or of the Spector-Gandy tree $T$).
\end{definition}

It is clear from the preceding comments that for all sets $P \subseteq \om$ which are \pii \ and not \del \ there is a Spector-Gandy space \spat{T} whose companion set is exactly $P$.

\begin{remark}
\label{remark first remark about Spector-Gandy spaces}
\normalfont
For a Spector-Gandy space \spat{T} with companion set $P$ we have the following.

\tu{(1)} The set of all \del \ points of \spat{T} is not \del. This is because
\[
P(n) \iff (\exists x)[ x \in [T] \ \& \ x \in \del \ \& \ x(0)=n ],
\]
so if the set $\del \cap \spat{T}$ were \del \ then $P$ would be \sig.\smallskip

\tu{(2)} There is some $n \not \in P$ such that $[T_{(n)}] \neq \emptyset$. In fact the \sig \ set
\[
I=\set{n \not \in P}{[T_{(n)}] \neq \emptyset}
\]
is not \del \ (and in particular is non-empty), since
\[
\neg P(n) \iff n \in I  \ \vee \ (\forall x \in [T])[x(0) \neq n]
\]
and $\neg P$ is not in \pii. Notice that $T_{(n)}$ is a Kleene tree for all $n \in I$. It follows that a Spector-Gandy tree is the infinite sum of Kleene trees and of recursive trees whose body is at most a singleton, see Figure \ref{fig:sgtree}. (Both cases occur infinitely many times.)

In particular a Spector-Gandy space contains a Kleene space and therefore is uncountable.\smallskip

\tu{(3)} The preceding Kleene trees $T_{(n)}$, $n \in I$ are not all similar in the sense that they cannot all be the tree $\set{\cn{(n)}{u}}{u \in S} \cup \{\emptyset\}$ for some fixed Kleene tree $S$. Indeed if this were not the case then we would have that
\[
P(n) \iff (\exists x)[\cn{(n)}{x} \in [T] \ \& \ x \not \in [S]]
\]
for all \n. According to the preceding equivalence the set $P$ would be \sig, a contradiction.\smallskip

\tu{(4)} Using sums of trees one can see that every Kleene tree has a recursive copy inside a Spector-Gandy tree. In particular every Kleene space is below a Spector-Gandy space under \dleq.\smallskip

\tu{(5)} Since the set of \del \ points of a Spector-Gandy space is not \del \ it follows that no Spector-Gandy space is \del-isomorphic to a Kleene space. The next step is to show that Spector-Gandy spaces are different from the Baire space.

\tu{(6)} From Corollary \ref{corollary incomparable hyperdegrees} and the preceding remarks it follows that every Spector-Gandy space contains an infinite set whose members are pairwise hyperarithmetically incomparable and their hyperdegree is below \W.
\end{remark}

\begin{figure}[t]
\begin{picture}(300,100)(0,0)
\put(140,90){\vector(-3,-1){60}}
\put(140,90){\vector(-1,-1){20}}
\put(140,90){\vector(1,-4){5.1}}
\put(140,90){\vector(2,-1){40}}
\put(140,90){\vector(4,-1){80}}
\put(235,70){$\dots$}
\put(80,70){\line(1,-2){35}}
\put(80,70){\line(-1,-2){35}}
\put(64.5,30){\small{Kleene}}
\put(68.5,20){\small{tree}}
\put(80,5){$\vdots$}
\put(180,70){\line(1,-2){35}}
\put(180,70){\line(-1,-2){35}}
\put(164.5,30){\small{Kleene}}
\put(168.5,20){\small{tree}}
\put(180,5){$\vdots$}
\put(145,70){\vector(-2,-3){13}}
\put(132,51){\vector(0,-1){21}}
\put(131,15){$\vdots$}
\put(127,3){\del}
\put(220,70){\vector(0,-1){20}}
\put(220,70){\vector(1,-1){20}}
\put(240,50){\vector(1,-2){10}}
\put(248,15){$\vdots$}
\put(245,3){\del}
\end{picture}
\caption{A typical Spector-Gandy tree.}
\label{fig:sgtree}
\end{figure}
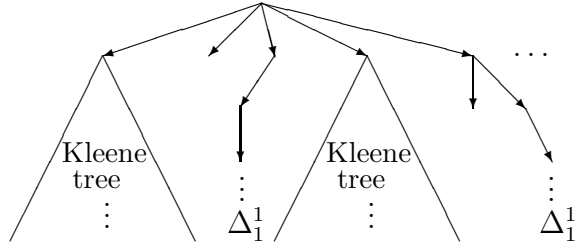

\begin{theorem}

\label{theorem Spector-Gandy spaces are not del isomorphic}

No Spector-Gandy space is \del-isomorphic to the Baire space.
\end{theorem}

\begin{proof}
If \spat{T} is Spector-Gandy space which is \del-isomorphic to \ca{N} then from Theorem \ref{theorem characterization of isomorphic to the Baire space} the tree $T$ would contain a \del-isomorphic copy of the complete binary tree. It follows that for some \n \ the subtree $T_{(n)}$ would contain a \del \ copy of the complete binary tree. However $[T_{(n)}]$ is either a singleton or does not contain \del \ members and we have a contradiction.

We also put down a second proof which does not invoke Theorem \ref{theorem characterization of isomorphic to the Baire space} - in fact the proof of the latter was motivated by the measure-theoretic argument that we present below.

Suppose towards a contradiction that there is a Spector-Gandy space \spat{T} with companion set $P$, which is \del \ isomorphic to the Baire space. In particular there exists a \del-injection $f: \ca{C} \inj \spat{T}$. Let $\mu$ be the usual coin-toss measure on the Cantor space \ca{C}. For all $(n) \in T$ we define the set
\[
C_n = f^{-1}[[T_{(n)}]] = \set{\alpha \in \ca{C}}{f(\alpha) \in [T] \ \& \ f(\alpha)(0)=n}.
\]
It is clear that each $C_n$ is a \del \ set and that $\cup_{(n) \in T} C_n = f^{-1}[[T]]$. Since $f$ is injective the set $f^{-1}[\spat{T} \setminus [T]] = f^{-1}[T]$ is at most countable and hence it has zero $\mu$-measure. It follows that 
\[
\mu(\ca{C}) = \mu(f^{-1}[\spat{T}]) = \mu(f^{-1}[[T]]) = \mu(\cup_{(n) \in T} C_n).
\]
Since $\mu \neq 0$ there is some $(n_0) \in T$ such that $\mu(C_{n_0}) > 0$. In particular the set $C_{n_0}$ is uncountable and since $f$ is injective, the set $[T_{(n_0)}]$ is uncountable as well. It follows that $n_0$ does not belong to the companion set $P$.

On the other hand since $\mu(C_{n_0}) > 0$ and $C_{n_0}$ is a \del \ set it follows from Tanaka \cite{tanaka_a_basis_result_for_pii_sets_of_positive_measure} - Sacks \cite{sacks_measure_theoretic_uniformity} that $C_{(n_0)}$ contains some $\alpha \in \del$. So $f(\alpha)$ is a $\del$ infinite branch of $T$ with $f(\alpha)(0) = n_0$. This implies that $n_0$ belongs to the companion set $P$, a contradiction.
\end{proof}

If $T, S$ are Spector-Gandy trees then the tree $T \oplus S$ a Spector-Gandy tree as well. So from Lemma \ref{lemma sum and product of trees} it follows that the sum of Spector-Gandy spaces is recursively isomorphic to a Spector-Gandy space. It is not clear if the similar assertion for products is true, since the product of two Spector-Gandy trees is not a Spector-Gandy tree. Nevertheless the product of two Spector-Gandy spaces is different under \dleq \ from the other known spaces.

\begin{theorem}
\label{theorem product of Spector-Gandy spaces is neither Kleene nor Baire}
The product of two Spector-Gandy spaces is not \del-isomorphic to any Kleene space or to the Baire space.
\end{theorem}

\begin{proof}
Suppose that $T$ and $S$ are Spector-Gandy trees and let $P, Q$ be the corresponding companion sets. From Lemma \ref{lemma sum and product of trees} the product $\spat{T} \times \spat{S}$ is recursively isomorphic to $\spat{\prodt{T}{S}}$.

We first show the assertion about Kleene spaces. We define
\[
R(m) \iff (\exists \gamma \in \del)[\cn{(m)}{\gamma} \in [\prodt{T}{S}]].
\]
Pick some $(k) \in S$ with $k \not \in Q$. We claim that
\begin{align}
\label{equation theorem product of Spector-Gandy spaces is neither Kleene nor Baire A}P(n) \iff R(\pairint{n}{k})
\end{align}
for all \n. For the left-to-right-hand implication let $\alpha \in \del$ be such that $\cn{(n)}{\alpha} \in [T]$. We take $\gamma$ to be $$(\pairint{\alpha(0)}{-1},\pairint{\alpha(1)}{-1},\dots,\pairint{\alpha(n)}{-1},\dots).$$ Clearly $\gamma \in \del$ and $\cn{(\pairint{n}{k})}{\gamma} \in [\prodt{T}{S}]$. Therefore $\pairint{n}{k} \in R$. Conversely assume that there exists some $\gamma \in \del$ such that $\cn{(\pairint{n}{k})}{\gamma} \in [\prodt{T}{S}]$. There exist \del-recursive functions $x,y : \om \to \om \cup \{-1\}$ such that $\gamma = \prodt{x}{y}$. If $y(i) \geq 0$ for all $i$, then from the fact that $\cn{(\pairint{n}{k})}{\prodt{x}{y}} \in [\prodt{T}{S}]$ we would have that $\cn{(k)}{y} \in [S]$. Since $y \in \del$ it would follows that $k \in Q$, a contradiction. Therefore $y$ is eventually $-1$, and since $\gamma = \prodt{x}{y} \in \ca{N}$ it follows that $x(i) \geq 0$ for all $i$ and $\cn{(n)}{x} \in [T]$. Using that $x \in \del$ we have that $n \in P$.

It is clear that
\[
R(m) \iff (\exists \delta)[\delta(0) = m \ \& \ \delta \in \del \cap [\prodt{T}{S}]].
\]
If the space $\spat{\prodt{T}{S}}$ were \del-isomorphic to a Kleene space then the set $\del \cap [\prodt{T}{S}] = (\del \cap \spat{\prodt{T}{S}}) \setminus \prodt{T}{S}$ would be \del \ as well. From the preceding equivalence $R$ would be \sig, and from (\ref{equation theorem product of Spector-Gandy spaces is neither Kleene nor Baire A}) it would follow that $P$ is also a \sig \ set, contradicting the definition of a Spector-Gandy space.

Now we show the assertion about the Baire space. Assume towards a contradiction that \spat{\prodt{T}{S}} is \del-isomorphic to the Baire space. From Theorem  \ref{theorem characterization of isomorphic to the Baire space} the tree \prodt{T}{S} contains a \del-copy of the complete binary tree. Then for some $m$ the subtree $(\prodt{T}{S})_{(m)}$ contains a \del-copy of the complete binary tree and so the space \spat{(\prodt{T}{S})_{(m)}} is \del-isomorphic to the Baire space. There exist $n^\ast, k^\ast \in \om \cup \{-1\}$ not both $-1$ with $$m = \pairint{n^\ast}{k^\ast}.$$

If $n^\ast \in P$ and $k^\ast \in Q$ then the body of the subtree $(\prodt{T}{S})_{(m)}$ is countable, a contradiction. If $n^\ast \in \om \setminus P$ and $k^\ast \in \om \setminus Q$ then $(\prodt{T}{S})_{(m)}$ is a Kleene tree, a contradiction again. If $n^\ast \in \om \setminus P$ and $k^\ast \in Q$ then the \del \ infinite branches of $(\prodt{T}{S})_{(m)}$ have the form
\begin{align*}
(\pairint{n^\ast}{\beta(0)},\pairint{u(0)}{\beta(1)}&,\dots,\\&\hspace*{-10mm}\pairint{u(t-1)}{\beta(t)},\pairint{-1}{\beta(t+1)}, \pairint{-1}{\beta(t+2)},\dots),
\end{align*}
where $u \in T_{(n^\ast)}$ and $\beta$ is the unique infinite branch of $S_{(k^\ast)}$. From this it follows that the set $\del \cap \spat{(\prodt{T}{S})_{(m)}}$ is \del, contradicting that the space $\spat{(\prodt{T}{S})_{(m)}}$ is \del-isomorphic to the Baire space. Similarly one rejects the case $n^\ast \in P$ and $k^\ast \in \om \setminus Q$. Finally assume that $n^\ast=-1$ and that $k^\ast \geq 0$. In this case every $z \in [(\prodt{T}{S})_{(m)}]$ has the form
\[
(\pairint{-1}{y(0)},\pairint{-1}{y(1)},\dots, \pairint{-1}{y(n)}, \dots)
\]
for some $y \in [S_{(k^\ast)}]$. If $k^\ast \in Q$ then the body of $(\prodt{T}{S})_{(m)}$ would contain only one infinite branch and if $k^\ast \not \in Q$ then the tree $(\prodt{T}{S})_{(m)}$ would be a Kleene tree. In both cases we have a contradiction. The case $n^\ast \geq 0$ and $k^\ast =-1$ is treated similarly.
\end{proof}

Now we prove the basic facts about the Cantor-Bendixson decomposition of a Spector-Gandy space.

\begin{theorem}

\label{theorem CB decomposition of Spector-Gandy spaces}
\tu{(1)} For every Spector-Gandy space \spat{T} with scattered part $ \scat{\spat{T}}$ we have that
\begin{align*}
\scat{\spat{T}} =& \set{x \in \spat{T}}{x \in \del}.
\end{align*}
In particular the scattered part of \spat{T} is $\pii \setminus \sig$ and the perfect kernel is $\sig \setminus \pii$.\smallskip

\tu{(2) (\cf \cite{kreisel_cantor_bendixson})} For every Spector-Gandy tree the scattered part \scat{[T]} of $[T]$ as a closed subset of the Baire space is $\pii \setminus \sig$ and its perfect kernel is $\sig \setminus \pii$. Moreover\footnote{It is not hard to see that the set of isolated points
of a recursively presented metric space is arithmetical. This theorem says that when it comes to $\Pi^0_1$ sets the question of being an
isolated point of $A$ is no better than a \pii \ question.}
\[
\scat{[T]} = \set{x \in [T]}{x \ \mathrm{is \ an \ isolated \ point \ of} \ [T]} = \del \cap [T].
\]
\end{theorem}

\begin{proof}
Let $P$ be the companion set of \spat{T}. If $x$ is in $\scat{\spat{T}}$ then $x \in \del$ from Remark \ref{remark scattered part consists of del points}. Conversely if $x$ is in \del \ and $x \in [T]$ then $x(0) \in P$. Therefore the open neighborhood $\set{y \in \spat{T}}{y(0) = x(0)}$ of $x$ is contained in $T \cup \{x\}$ and so it is countable. On the other hand if $x \in T$ then $x$ is an isolated point of \spat{T}. In any case we have that $x \in \scat{\spat{T}}$.

Since the set of all \del \ points of a Spector-Gandy space is $\pii \setminus \sig$ it follows that the scattered part of \spat{T} is in $\pii\setminus\sig$ as well. The assertion about the perfect kernel follows by taking the complements.

Regarding the assertion about $[T]$ we recall from Theorem \ref{theorem properties of Nt} that
\[
\scat{\spat{T}} = T \cup \scat{[T]}.
\]
Since $T$ is a $\Sigma^0_1$ subset of \spat{T} it follows that \scat{[T]} is a $\pii \setminus \sig$ subset of \spat{T} and thus a $\pii \setminus \sig$ subset of \ca{N}. It follows that the perfect kernel $[T] \setminus \scat{[T]}$ is $\sig \setminus \pii$. The equality $\scat{[T]} = \del \cap [T]$ is proved exactly as above.

Finally if $x \in [T]$ is in \del \ then from the properties of a Spector-Gandy tree the body $[T_{(x(0))}]$ is the singleton $\{x\}$, hence $x$ is an isolated point of $[T]$. On the other hand it is clear that the isolated points of $[T]$ are \del, therefore we have the equality $$\set{x \in [T]}{x \ \textrm{is an isolated point of} \ [T]} = \del \cap [T].$$
\end{proof}

\begin{remark}
\label{remark perfect kernel not in del and isomorphic to Baire space}\normalfont
Suppose that \spat{T} is a Spector-Gandy space and that $\ca{X} = \ca{N} \oplus \spat{T}$, \ie \ca{X} is the topological sum of \ca{N} and \spat{T}. From the preceding theorem we can see that the perfect kernel of \ca{X} is not a \del \ set. On the other hand $\ca{N} \dleq \ca{X}$ and so \ca{X} is \del-isomorphic to the Baire space.

Therefore it is possible for a recursively presented metric space to be \del-isomorphic to the Baire space and yet its perfect kernel fails to be a \del \ set.
\end{remark}

As it is with Kleene spaces, one can pass from \del \ injections $f: \spat{T} \inj \spat{S}$ to \del-injections on $[T]$ to $[S]$, where $S$ is a Spector-Gandy tree. In particular the converse of Lemma \ref{lemma bodyT less than bodyS implies Nt less than Ns} is true in Spector-Gandy spaces. The proof is less trivial and uses the result of Luckham that every infinite \pii \ set of naturals contains an infinite \del \ subset, \cf Corollary XXVI \cite{rogers_theory_of_recursive_functions_effective_computability}.

\begin{lemma}

\label{lemma [T] is embedded into [S] in the Spector-Gandy case}

Suppose that $T$ and $S$ are recursive trees with non-empty body and that there is a \del \ injection $$f: \spat{T} \inj \spat{S}.$$ If $S$ is a Spector-Gandy tree then there exists a \del \ recursive function $$\pi: \ca{N} \to \ca{N}$$ which is injective on $[T]$ such that $\pi[[T]] \subseteq [S]$.
\end{lemma}

\begin{proof}
Let $f: \spat{T} \bij \spat{S}$ be a \del \ isomorphism. The idea is to collect all points of $[T]$ which are mapped by $f$ inside $S$ and reassign them to \del \ points in $ [S]$.

Define the set $A \subseteq \spat{T}$ by
\[
x \in A \eq x \in [T] \ \& \ f(x) \in S.
\]
If $A = \emptyset$ then we are done, so assume that $A \neq \emptyset$. Clearly $A$ is in \del \ and since $f$ is injective it follows that the set $A$ is countable. Moreover every $x \in A$ is a \del \ point since from Lemma \ref{lemma the inverse function is del and same hyperdegree} $x \heq f(x) \in S$ for all $x \in A$.

Therefore we may write $A = \set{x_n}{n \in I}$, where $I$ is either \om \ or a finite segment of \om, $x_n \neq x_m$ for $n \neq m$ in $I$ and the relation
\[
W(n,t) \eq n \in I \ \& \ x_n \in N(\spat{T},t)
\]
is \del. Let $Q$ be the companion set of the Spector-Gandy space \spat{S} and choose an infinite \del \ set $J$ which is contained in $Q$.

We consider the set
\begin{align*}
B =& \set{y \in [S]}{y(0) \in J}\\
    =& \set{y \in [S] \cap \del}{y(0) \in J}.
\end{align*}
Then $B$ is a \del \ subset of \spat{S} which consists of \del \ points. Also $B$ is infinite since $J$ is infinite. So we can write $B = \set{y_n}{\n}$, where $(y_n)_{\n}$ is an injective \del \ sequence. Notice that $f^{-1}[B] \cap A = \emptyset$ since $B \subseteq [S]$ while $f(x) \in S$ for all $x \in A$.

Finally define the function $\pi: \ca{N} \to \ca{N}$ by:
\[
\pi(x)  =\begin{cases}
              y_{2n}, &  \textrm{if $x \in A$ and $x=x_n$ for some $n \in I$},\\
              y_{2n+1}, &  \textrm{if $x \in f^{-1}[B] \cap [T]$ and $f(x)=y_n$ for some $n$},\\
             f(x), & \textrm{if $x \in [T] \setminus (A \cup f^{-1}[B])$},\\
             \beta_0, & x \not \in [T],
              \end{cases}
\]
where $\beta_0$ is some fixed recursive point in $\ca{N} \setminus [S]$.  (Notice that $[S] \neq \ca{N}$ for otherwise \spat{S} would be \del \ isomorphic to the Baire space contradicting Theorem \ref{theorem Spector-Gandy spaces are not del isomorphic}.)

The function $\pi$ is clearly \del-recursive, injective on $[T]$ and $\pi[[T]] \subseteq [S]$.
\end{proof}

\subsection{Chains and antichains in Spector-Gandy spaces under \dleq} Here we prove the existence of antichains, strictly increasing and strictly decreasing sequences of Spector-Gandy spaces. As before Kreisel compactness is the key tool, but here there is an extra difficulty: we need to make sure that at the end we will have constructed a Spector-Gandy tree. In the case of Kleene spaces it was fairly simple to end up with a Kleene tree, for all we had to do was to ensure that enough \del \ points were being removed at a given stage. In the case of Spector-Gandy spaces however we need to make sure that the body of the resulting tree contains many \del \ points. The idea in order to achieve this, is to use an already given Spector-Gandy tree as a pilot. This is explained in the following simple but useful lemma.

\begin{lemma}
\label{lemma obtaining spector-gandy tree}
Suppose that $K$ is a Spector-Gandy tree with a companion set $P$ and that $T$ is a recursive tree such that $T_{(n)} = K_{(n)}$, whenever $n \in P$ or $[T_{(n)}]$ has a \del \ member. Then $T$ is also a Spector-Gandy tree with the same companion set $P$.
\end{lemma}

\begin{proof}
We need to show that
\begin{align*}
P(n) \iff& (\exists \alpha \in \del)[\cn{(n)}{\alpha} \in [T]]\\
     \iff& (\exists! \alpha)[\cn{(n)}{\alpha} \in [T]],
\end{align*}
for all \n. Suppose that $n$ is a member of $P$. From our hypothesis we have that $K_{(n)} = T_{(n)}$. Since $P$ is the companion set of $K$ there exists some $\alpha \in \del$ such that that $\cn{(n)}{\alpha} \in [K]$. Hence $\cn{(n)}{\alpha} \in [T]$. The rest implications are proved similarly.
\end{proof}

We are now ready for our next incomparability result.

\begin{lemma}
\label{lemma incomparable Spector-Gandy spaces}
There exists a sequence of Spector-Gandy trees $(T_{i})_{i \in \om}$ and a sequence $(\alpha_i)_{i \in \om}$ in \ca{N} such that $\alpha_i \in [T_i]$ and $\alpha_i \not \heq \gamma$ for all $\gamma \in [T_j]$ and all $i \neq j$.\footnote{Since Spector-Gandy trees contain \del \ infinite branches it is not possible to obtain the stronger incomparability condition $\del(\alpha_i) \cap [T_j] = \emptyset$ of Theorem \ref{theorem Fokina-Friedman-Toernquist}.}
\end{lemma}

\begin{proof}
We first prove the assertion for only two trees and then we describe how one can get the result for an infinite sequence. We will construct two Spector-Gandy trees $T$ and $S$ with the help of a given set $P \subseteq \om$ in $\pii \setminus \sig$. The idea is to approximate $T$ and $S$ from two directions: from ``inside" of $P$, which together with the preceding lemma will guarantee that we will end up with Spector-Gandy trees (with the same companion set $P$), and also from ``outside" of $P$ by plugging in some Kleene trees which satisfy the incomparability conditions of Lemma \ref{lemma for every kleene tree there exists an incomparable kleene tree}. This will settle the incomparability between $T$ and $S$.

We now proceed to the construction. Let $\preceq \subseteq \om^2$ be a recursive pseudo-well-ordering and $P$ be the well-founded part of $\preceq$. It is not hard to see that $P$ is a \pii \ set. From Theorem \ref{theorem Gandy length of well-founded part} the set $P$ is not \del. It follows that for all \del \ sets $H \subseteq P$ there exists some $i \in P$ such that $j \prec i$ for all $j \in H$, for otherwise we would have that
\[
i \in P \iff (\exists j)[j \in H \ \& \ i \preceq j]
\]
and $P$ would be \del. We also notice that $P$ is closed from below under $\preceq$.

From the strong form of Spector-Gandy Theorem there is a recursive tree $K$ such that
\begin{align*}
P(n)   \iff& (\exists ! \alpha)[\cn{(n)}{\alpha} \in [K]]\\
         \iff& (\exists \alpha \in \del)[\cn{(n)}{\alpha} \in [K]].
\end{align*}
We define the relations $A \subseteq \om \times \Tr$, $B \subseteq \Tr$ and $Hd \subseteq \ca{N} \times \Tr$ by
\begin{align*}
A(i,T) \iff& (\forall j \preceq i)[T_{(j)} = K_{(j)}],\\
B(T) \iff& (\forall \alpha \in \del)(\forall n)[\cn{(n)}{\alpha} \in [T] \ \longrightarrow \ T_{(n)} = K_{(n)}],\\
Hd(\alpha,S) \iff& (\forall \gamma \in \del(\alpha))[\alpha \in \del(\gamma) \ \longrightarrow \gamma \not \in [S]]
\end{align*}
In other words condition $A(i,T) \ \& \ B(T) \ \& \ Hd(\alpha,S)$ means that $T_{(j)}$ agrees with $K_{(j)}$ whenever $j \preceq i$ or $[T_{(j)}]$ has a \del \ infinite branch, and the hyperdegree of $\alpha$ does not occur in $[S]$.

It is clear that $A$ is \del. Using that the relation $\alpha \in \del(\beta)$ is \pii \ and the Theorem on Restricted Quantification one can verify that the sets $B$ and $Hd$ are \sig. Moreover if $j \preceq i$ we have that
\begin{align}
\label{equation incomparable sg A3} A(i,T) \ \Longrightarrow \ A(j,T)
\end{align}
for all $T \in \Tr$.

Now we define the set $D \subseteq \om \times (\ca{N} \times \Tr)^2$ by saying that $D(i,\alpha,T,\beta,S)$ holds
exactly when
\begin{align}
\label{equation incomparable sg D1}& \ \ \ \ T, S \ \textrm{are recursive trees} \ \& \ \alpha \in [T] \ \& \ \beta \in [S]\\
\label{equation incomparable sg D2}&\& \ A(i,T) \ \& \ A(i,S)\\
\label{equation incomparable sg D6}&\& \ B(T) \ \& \ B(S)\\
\label{equation incomparable sg D4}&\& \ Hd(\alpha,S) \ \& \ Hd(\beta,T)
\end{align}
We notice that it is only condition (\ref{equation incomparable sg D2}) which depends on $i$. Our purpose is to show
that $\cap_{i \in P}D_i \neq \emptyset$ by applying Kreisel compactness. Clearly $D$ is a \sig \ set, so in particular $D$ is computed by a \sig \ set on $P \times (\ca{N} \times \Tr)^2$.

From (\ref{equation incomparable sg A3}) we have that $D_i \subseteq D_j$ for all $j \preceq i$. Now let $H
\subseteq P$ be in \del \ and $i \in P$ be such that $j \preceq i$ for all $j \in H$. It follows that $D_i \subseteq
\cap_{j \in H} D_j$. Hence in order to verify that the set $\cap_{j \in H} D_j$ is non-empty it is enough to show that
$D_i \neq \emptyset$ for all $i \in P$.

We now prove the latter. Let $i \in P$. From Lemma \ref{lemma for every kleene tree there exists an incomparable kleene tree} there exist
Kleene trees $T^0$ and $S^0$ and $\alpha^\ast, \beta^\ast$ in $[T^0]$ and $[S^0]$ respectively such that
\begin{align}
\label{equation incomparable sg D5}\beta^\ast \not \heq \delta \quad \textrm{and} \quad \alpha^\ast \not \heq \gamma,
\end{align}
for all $\delta \in [T^0]$ and $\gamma \in [S^0]$. (Alternatively one can use Theorem \ref{theorem Fokina-Friedman-Toernquist}.) We also pick some $n \not \in P$ and we define
\begin{align*}
\cn{(j)}{u} \in T \iff [j \preceq i \ \& \ \cn{(j)}{u} \in K] \ \vee \  [j =n \ \& \ u \in T^0]
\end{align*}
see Figure \ref{fig:approx-sg}. Notice that exactly one of the cases in the preceding equivalence applies. We include the empty sequence to $T$ as well, so that $T$ becomes a tree.
\begin{figure}[t]
\begin{picture}(300,130)(0,-5)
\put(120,110){\tiny{$\bullet$}}
\put(125,115){$K$}
\put(122,112){\line(-4,-1){75}}
\put(123,112){\line(3,-1){59}}
\put(27,80){$\prec i$}
\put(57,80){$\not \in P$}
\put(87,80){$(i)$}
\put(124,80){$\prec i$}
\put(157,80){$\succ i$}
\put(254,80){$(n)$}
\put(30,70){\tiny{$\bullet$}}
\put(32,71){\vector(-1,-1){19}}
\put(32,71){\vector(1,-4){5.1}}
\put(37,51){\vector(-2,-3){13}}
\put(24,32){\vector(0,-1){21}}
\put(22,-5){$\vdots$}
\put(30,3){\del}
\put(60,70){\tiny{$\circ$}}
\put(90,70){\tiny{$\bullet$}}
\put(92,72){\vector(0,-1){20}}
\put(92,72){\vector(-2,-3){13.5}}
\put(92,72){\vector(2,-3){13.5}}
\put(92,53){\vector(1,-4){5.5}}
\put(97,32){\vector(0,-1){21}}
\put(96,-5){$\vdots$}
\put(102,3){\del}
\put(125,70){\tiny{$\bullet$}}
\put(127,72){\vector(0,-1){20}}
\put(127,72){\vector(2,-3){13.5}}
\put(140,53){\vector(1,-4){5.5}}
\put(145,32){\vector(0,-1){21}}
\put(144,-5){$\vdots$}
\put(150,3){\del}
\put(160,70){\tiny{$\circ$}}
\put(180,70){$\dots$}
\put(260,70){\tiny{$\bullet$}}
\put(262,72){\line(1,-2){35}}
\put(262,72){\line(-1,-2){35}}
\put(255,30){$T^0$}
\end{picture}
\caption{The approximation at the stage $i$.}
\label{fig:approx-sg}
\end{figure}
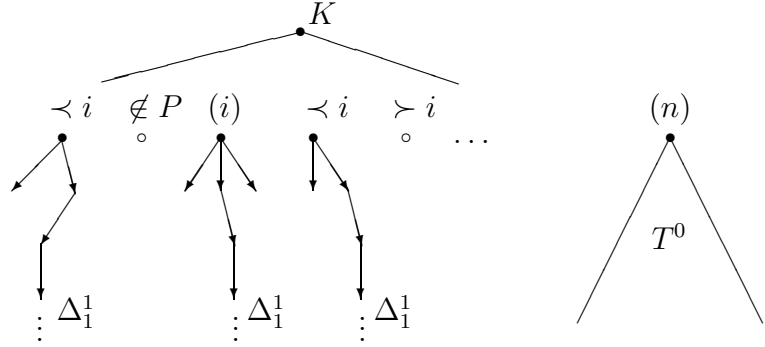

Similarly we define $S$ by replacing $T^0$ with $S^0$. Since $\preceq$ is recursive we have that $T$ and $S$ are recursive trees.  Moreover $T_{(j)} = S_{(j)} = K_{(j)}$ for all $j \preceq i$ and $T_{(n)}, S_{(n)}$ are Kleene trees. In particular we have that $A(i,T)$ and $A(i,S)$ hold, \ie condition (\ref{equation
incomparable sg D2}) is satisfied.

We verify that $D(i,\alpha,T,\beta,S)$ holds, where $\alpha = \cn{(n)}{\alpha^\ast}$ and $\beta = \cn{(n)}{\beta^\ast}$.
Condition (\ref{equation incomparable sg D1}) is satisfied, since  $\alpha^\ast$ and $\beta^\ast$ are in $[T^0]$ and
$[S^0]$ respectively. Now we show that $B(T)$ holds. (The argument for $B(S)$ is similar.) If $\cn{(n')}{\alpha'}$ is
in $[T]$ for some $\alpha'$ in \del \ then $n'$ is not equal to $n$ since $T_{(n)}$ is a Kleene tree. Therefore
$n' \preceq i$ and from the definition we have that $T_{(n')} = K_{(n')}$ and $B(T)$ holds. Hence condition (\ref{equation
incomparable sg D6}) is satisfied.

Finally we show condition (\ref{equation incomparable sg D4}). Let $\gamma \in \ca{N}$ be such that $\alpha \heq
\gamma$ and suppose towards a contradiction that $\gamma \in [S]$. If $\gamma(0) \preceq i$ then (since $i$ is in $P$) we have that $\gamma(0) \in P$ and $\gamma \in [K]$. It follows from the key properties of $K$ that $[K_{(\gamma(0))}]$ has exactly one infinite branch and that this branch is \del. Hence $\gamma$ is this unique infinite branch and in particular we have that $\gamma \in \del$. Since $\alpha \in \del(\gamma)$ it follows that $\alpha \in \del$, contradicting that $\alpha^\ast \not \in \del$. So it cannot be the case $\gamma(0) \preceq i$ and hence $\gamma(0) = n$. In this case however we have that
$\gamma^\ast \in [S^0]$, where $\gamma^\ast = (t \mapsto \gamma(t+1))$. On the other hand $\gamma^\ast \heq \gamma \heq \alpha \heq \alpha^\ast$ contradicting the second inequality of (\ref{equation incomparable sg D5}). It follows that $\gamma$ is not a member of $[S]$, \ie
$Hd(\alpha,S)$ holds. Similarly one shows that $Hd(\beta,T)$ holds as well and condition (\ref{equation incomparable sg D4}) is
satisfied.

Having proved that $D_i \neq \emptyset$ for all $i \in P$ we may apply Kreisel compactness (Theorem \ref{theorem Kreisel
compactness}) and obtain $(\alpha,T,\beta,S)$ in the intersection $\cap_{i \in P}D_i$. From condition (\ref{equation
incomparable sg D1}) we have that $T$ and $S$ are recursive trees and that $\alpha \in [T]$ and $\beta \in [S]$.
Moreover $\alpha \not \heq \gamma$ for all $\gamma \in [S]$ because of condition (\ref{equation incomparable sg D4}).
Similarly we have that $\beta \not \heq \delta$ for all $\delta \in [T]$.

Finally since $B(T)$ holds and $A(i,T)$ is satisfied for all $i \in P$, we have $T_{(n)} = K_{(n)}$, whenever $n$ is in $P$ or $[T_{(n)}]$ contains a \del \ member. Hence from Lemma \ref{lemma obtaining spector-gandy tree} it follows that $T$ is a Spector-Gandy tree with companion set $P$. Similarly one shows the same assertion about $S$ and the proof for two trees is complete.

We now describe the modifications needed to get the result for infinitely many trees. Keeping the same notation we define the set $D^\ast \subseteq \om \times \ca{N} \times \Tr^\om$ by saying that $D^\ast(i,\alpha,(T_m)_{m \in \om})$ holds exactly when
\begin{align*}
& \ \ \ (\forall m)[T_m\ \textrm{is a recursive tree} \ \& \ (\alpha)_m \in [T_m]]\\
&\& \ (\forall m)(\forall n \neq m)[A(i,T_m) \ \& \ B(T_m) \ \& \ Hd((\alpha)_m,T_n)].
\end{align*}
In order to show that $D^\ast_i$ is not-empty for all $i \in P$ we apply Lemma \ref{lemma for every kleene tree there exists an incomparable kleene tree} infinitely many times to get a sequence $(T^0_m)_{m \in \om}$ of Kleene trees and $\alpha \in \ca{N}$ such that $(\alpha)_m \in [T^0_m]$ and
\[
(\alpha)_m \not \heq \gamma
\]
for all $\gamma \in [T^0_n]$ and all $m \neq n$. (Alternatively one can use Theorem \ref{theorem
Fokina-Friedman-Toernquist}.) We then define
\begin{align*}
\cn{(j)}{u} \in T_m \iff [j \preceq i \ \& \ \cn{(j)}{u} \in K] \ \vee \  [j =n \ \& \ u \in T^0_m]
\end{align*}
and we include in $T_m$ the empty sequence as well. Clearly each $T_m$ is a recursive tree and with the same arguments as above one can show that $(\alpha,(T_m)_{m \in \om}) \in D^\ast_i$.

Finally if $(\alpha,(T_m)_{m \in \om})$ is in $\cap_{i \in P}D^\ast_i$ then $T_m$ is a Spector-Gandy tree with companion set $P$ and also $(\alpha)_m \not \heq \gamma$ for all $\gamma \in [T_n]$ and for all $n \neq m$.
\end{proof}

\begin{remark}\normalfont
\label{remark after incomparability of Spector-Gandy spaces}
In the preceding proof we have actually shown that for every recursive pseudo-well-ordering $\preceq$ there exists a sequence $(T_i)_{\iin}$ of Spector-Gandy trees as in Lemma \ref{lemma incomparable Spector-Gandy spaces} which have the same companion set $\WF(\preceq)$. Moreover if $T$ is a Spector-Gandy tree with companion set $\WF(\preceq)$ then the $T_i$'s can be chosen so that $\del \cap [T_i] = \del \cap [T]$. The latter follows from conditions (\ref{equation incomparable sg D2}) and (\ref{equation incomparable sg D6}) in the preceding proof.\smallskip
\end{remark}

We proceed to the consequences of Lemma \ref{lemma incomparable Spector-Gandy spaces}.

\begin{theorem}
\label{theorem incomparable Spector-Gandy spaces}
There exists a sequence $T_0, T_1, \dots, T_i,\dots$ of Spector-Gandy trees such that $\spat{T_i} \not \dleq \spat{T_j}$ for all $i \neq j$.
\end{theorem}

\begin{proof}
We consider the sequences $(T_i)_{i \in \om}$ and $(\alpha_i)_{i \in \om}$ as in Lemma \ref{lemma incomparable Spector-Gandy spaces}. If there were a \del \ injection $f: \spat{T_i} \to \spat{T_j}$ for some $i \neq j$ then from Lemma \ref{lemma [T] is embedded into [S] in the Spector-Gandy case} there would be a \del-recursive function $\pi: \ca{N} \to \ca{N}$ which is injective on $[T_i]$ and $\pi[[T_i]] \subseteq [T_j]$. Since $\alpha_i \in [T_i]$ we would have that $\pi(\alpha_i) \heq \alpha_i$ and $\pi(\alpha_i) \in [T_j]$ contradicting the conclusion Lemma \ref{lemma incomparable Spector-Gandy spaces}.\footnote{Remark \ref{remark after incomparability of Spector-Gandy spaces} applies here as well.}
\end{proof}

\begin{theorem}
\label{theorem increasing sequences of SG spaces}
There exists a sequence $(S_i)_{\iin}$ of Spector-Gandy trees such that
\[
\spat{S_0} \dless \spat{S_1} \dless \dots \dless \spat{S_i} \dless \dots.
\]
\end{theorem}

\begin{proof}
Suppose that $(T_i)_{i \in \om}$ and $(\alpha_i)_{i \in \om}$ are as in Lemma \ref{lemma incomparable Spector-Gandy spaces}. We define $\ca{Y}_i = \oplus_{j \leq i} \ \spat{T_j}$ for all $i \in \om$. From the remarks following Theorem \ref{theorem Spector-Gandy spaces are not del isomorphic} we have that every $\ca{Y}_i$ is recursively isomorphic to the Spector-Gandy space $\spat{S_i}$, where $S_i = \oplus_{j \leq i} T_j$. Also it is clear that $\ca{Y}_i \dleq \ca{Y}_{i+1}$ for all $i \in \om$.

If it were $\ca{Y}_{i+1} \dleq \ca{Y}_i$ for some $i$, then we would have in particular that $\spat{T_{i+1}} \dleq \ca{Y}_i$ and so $\spat{T_{i+1}} \dleq \spat{S_i}$. From Lemma \ref{lemma [T] is embedded into [S] in the Spector-Gandy case} there would be a \del-recursive function $\pi: \ca{N} \to \ca{N}$ which is injective on $[T_{i+1}]$ with $\pi[[T_{i+1}]] \subseteq [S_i]$. Hence the hyperdegree of $\alpha_{i+1}$ would appear in $[S_i]$, and from the definition of $S_i$ there would be some $j \leq i$ such that $\alpha_{i+1} \heq \gamma$ for some $\gamma \in [T_j]$. This however would contradict the key property of $\alpha_{i+1}$. Hence $\ca{Y}_i \dless \ca{Y}_{i+1}$ for all $i \in \om$ and we are done.
\end{proof}

The next problem is about strictly decreasing sequences.

\begin{lemma}
\label{lemma decreasing sequence of Spector-Gandy spaces}
Suppose that $P \subseteq \om$ is in $\pii \setminus \sig$ and $\preceq$ is a \del \ linear pre-ordering, whose field includes $P$ and satisfies
\[
P(i) \ \& \ j \preceq i \Longrightarrow P(j)
\]
for all $i, j \in \om$. Then for every Spector-Gandy tree $T$ with companion set $P$ there exists a Spector-Gandy tree $S$ with the same companion set and some $\gamma \in [T]$ not in \del, such that $[T]$ contains the recursive injective image of $[S]$ and $\alpha \not \heq \gamma$ for all $\alpha \in [S]$.\smallskip

In fact if $\preceq$ is recursive the tree $S$ can be taken to be a subset of $T$.
\end{lemma}

\begin{proof}[\textit{Proof \tu{(}first part\tu{)}}.]
Let us assume for the time being that $\preceq$ is recursive. We define the set $D \subseteq \om \times \ca{N} \times \Tr$ by saying that $D(i,\gamma,S)$ holds exactly when
\begin{align}
\label{equation lemma decreasing sequence of Spector-Gandy spaces A}    & S \ \textrm{is a recursive tree and} \ S \subseteq T\\
\label{equation lemma decreasing sequence of Spector-Gandy spaces E}\& \ & \gamma \in [T] \ \& \ \gamma \not \in \del\\
\label{equation lemma decreasing sequence of Spector-Gandy spaces B}\& \ &(\forall j \preceq i)[S_{(j)} = T_{(j)}]\\
\label{equation lemma decreasing sequence of Spector-Gandy spaces C}\& \ &(\forall n)(\forall \alpha^\ast \in \del)[\cn{(n)}{\alpha^\ast} \in [S] \longrightarrow S_{(n)} = T_{(n)}]\\
\label{equation lemma decreasing sequence of Spector-Gandy spaces D}\& \ &(\forall \alpha \in \del(\gamma))[\gamma \in \del(\alpha) \longrightarrow \alpha \not \in [S]].
\end{align}
If $(\gamma, S)$ is in the intersection $\cap_{i \in P}D_i$ then $S$ is a recursive subtree of $T$ and from conditions (\ref{equation lemma decreasing sequence of Spector-Gandy spaces B}), (\ref{equation lemma decreasing sequence of Spector-Gandy spaces C}) we have $S_{(n)} = T_{(n)}$ whenever $n$ is in $P$ or $S_{(n)}$ contains a \del \ member. Hence from Lemma \ref{lemma obtaining spector-gandy tree} $S$ is a Spector-Gandy tree with companion set $P$. The rest assertions about $\gamma$ and $S$ follow from conditions (\ref{equation lemma decreasing sequence of Spector-Gandy spaces E}) and (\ref{equation lemma decreasing sequence of Spector-Gandy spaces D}).

We verify that the intersection $\cap_{i \in P}D_i$ is non-empty using Kreisel compactness. Clearly $D$ is computed by a \sig \ set on its domain. Moreover every \del \ set $H \subseteq P$ is $\preceq$-bounded in $P$, for otherwise (using our hypothesis about $\preceq$ and $P$) we would have that
\[
i \in P \iff (\exists j)[j \in H \ \& \ i \prec j]
\]
for all $i \in \om$. This would imply that $P$ is \del \ contradicting our hypothesis about $P$. Since $\preceq$ is transitive we have $D_{i'} \subseteq D_i$ for all $i \preceq i'$ in $P$. Therefore if $H$ is a \del \ subset of $P$ and $i \in P$ is an upper bound of $H$ we have that $D_i \subseteq \cap_{j \in H}D_j$. Hence it is enough to prove that $D_i \neq \emptyset$ for all $i \in P$.

Now we verify the latter. Suppose that $i$ is in $P$. We pick some $n \not \in P$ such that $T_{(n)}$ is a Kleene tree. Following the proof of Theorem \ref{theorem for every kleene space there is another one from below} there exists some $\gamma \in [T_{(n)}]$ and some $w$, which is an initial segment of the leftmost infinite branch of $T_{(n)}$, such that $[T_w] \cap \del(\gamma) = \emptyset$. We define the tree $S$ by
\[
\cn{(j)}{u} \in S \iff [j \preceq i \ \& \ \cn{(j)}{u} \in T] \ \vee \ [j = n \ \& \ \cn{(j)}{u} \in T_w]
\]
see Figure \ref{fig:approx2-sg}.
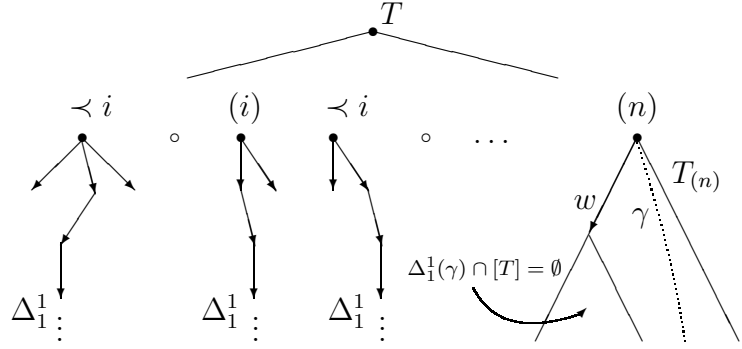
\begin{figure}[t]
\begin{picture}(300,130)(0,-10)
\put(140,110){\tiny{$\bullet$}}
\put(145,115){$T$}
\put(142,112){\line(-4,-1){70}}
\put(142,112){\line(4,-1){70}}
\put(27,80){$\prec i$}
\put(87,80){$(i)$}
\put(124,80){$\prec i$}
\put(234,80){$(n)$}
\put(30,70){\tiny{$\bullet$}}
\put(32,71){\vector(-1,-1){19}}
\put(32,71){\vector(1,-4){5.1}}
\put(32,72){\vector(1,-1){20}}
\put(37,51){\vector(-2,-3){13}}
\put(24,32){\vector(0,-1){21}}
\put(22,-5){$\vdots$}
\put(5,3){\del}
\put(65,70){\tiny{$\circ$}}
\put(90,70){\tiny{$\bullet$}}
\put(92,72){\vector(0,-1){20}}
\put(92,72){\vector(2,-3){13.5}}
\put(92,53){\vector(1,-4){5.5}}
\put(97,32){\vector(0,-1){21}}
\put(96,-5){$\vdots$}
\put(77,3){\del}
\put(125,70){\tiny{$\bullet$}}
\put(127,72){\vector(0,-1){20}}
\put(127,72){\vector(2,-3){13.5}}
\put(140,53){\vector(1,-4){5.5}}
\put(145,32){\vector(0,-1){21}}
\put(144,-5){$\vdots$}
\put(125,3){\del}
\put(160,70){\tiny{$\circ$}}
\put(180,70){$\dots$}
\put(240,70){\tiny{$\bullet$}}
\put(242,72){\line(1,-2){38.5}}
\put(242,72){\line(-1,-2){38.5}}
\put(255,55){$T_{(n)}$}
\put(242,72){\vector(-1,-2){18}}
\put(218,45){$w$}
\put(224,35){\line(1,-2){20}}
\qbezier[40](242,70)(255,30)(260,-5)
\put(240,40){$\gamma$}
\put(155,20){\tiny{$\del(\gamma) \cap [T] = \emptyset$}}
\qbezier(180,15)(190,-5)(220,5)
\put (222,6.5){\vector(1,1){1}}
\end{picture}
\caption{The approximation at the stage $i$.}
\label{fig:approx2-sg}
\end{figure}
Since $\preceq$ is recursive $S$ is a recursive subtree of $T$, and since $T_{(n)}$ is a Kleene tree we have $\gamma \not \in \del$. Clearly $\gamma$ is a member of $[T]$ and so conditions (\ref{equation lemma decreasing sequence of Spector-Gandy spaces A}) and (\ref{equation lemma decreasing sequence of Spector-Gandy spaces E}) are satisfied. Conditions (\ref{equation lemma decreasing sequence of Spector-Gandy spaces B}) and (\ref{equation lemma decreasing sequence of Spector-Gandy spaces C}) are also satisfied by the definition of $S$ (using that $T_w$ is a Kleene tree).

Assume towards a contradiction that $\alpha \in [S]$ has the same hyperdegree as $\gamma$. If $\alpha(0) \preceq i$ then, using that $i \in P$, we have that $\alpha(0) \in P$, and since $\alpha \in [T]$ it follows that $\alpha$  is \del, contradicting that $\alpha \heq \gamma$. If $\alpha(0) = n$ then $\alpha \in [T_w]$ and so $\del(\gamma) \cap [T_w] \neq \emptyset$, contradicting the choice of $w$ and $\gamma$. Hence no $\alpha \in [S]$ has the same hyperdegree as $\gamma$ and condition (\ref{equation lemma decreasing sequence of Spector-Gandy spaces D}) is satisfied.

It follows that $(\gamma,S) \in D_i$ and so $D_i \neq \emptyset$ for all $i \in P$. This finishes the proof in the case where $\preceq$ is recursive.
\end{proof}

Now we show how one can get the result when $\preceq$ is \del. A tree $T$ is a \emph{\del \ Spector-Gandy tree} if it is a \del \ tree and satisfies the remaining conditions in the definition of Spector-Gandy trees, \ie we just replace the term ``recursive tree" with ``\del \ tree" in the definition. Notice the companion set of a \del \ Spector-Gandy tree is also a $\pii \setminus \sig$ set.

\begin{lemma}
\label{lemma del Spector-Gandy tree}
For every \del \ Spector-Gandy tree $T$ there exists a Spector-Gandy tree $S$ with the same companion set and a recursive function $f: \ca{N} \to \ca{N}$, which is injective on $[S]$ and $f[[S]] = [T]$.
\end{lemma}

\begin{proof}
We define
\[
Q(n,\alpha) \iff \cn{(n)}{\alpha} \in [T].
\]
Clearly $Q$ is \del \ so from Theorem \ref{theorem alltogether}-(\ref{theorem 4D.9}) and the basic representation of $\Sigma^0_1$ (\cf 4A.1 \cite{yiannis_dst}) there exists a recursive $R \subseteq \om^3$ such that
\begin{align*}
Q(n,\alpha) \iff& (\exists \beta)(\forall t)R(n,\barr{\alpha}(t),\barr{\beta}(t))\\
            \iff& (\exists! \beta)(\forall t)R(n,\barr{\alpha}(t),\barr{\beta}(t)).
\end{align*}
For the needs of the proof we consider a recursive isomorphism $\tau: \om^2 \bij \om$ and for simplicity of notation we identify every pair of naturals $(u_i,v_i)$ with $\tau(u_i,v_i)$. Under this notation we define $S$ by
\begin{align*}
\cn{(n)}{(u_0,v_0,\dots,u_{m-1},v_{m-1})}& \in S \iff\\& (\forall t < m)R(n,\langle u_0,\dots,u_{t-1} \rangle, \langle v_0,\dots,v_{t-1}\rangle).
\end{align*}
We include to $S$ the empty sequence, so that it becomes a recursive tree. Notice that for all $n$ and $\alpha$ with $\cn{(n)}{\alpha} \in [T]$ there exists exactly one $\beta$ such that $\cn{(n)}{(\alpha(0),\beta(0),\dots)} \in [S]$. Conversely if $\cn{(n)}{(\alpha(0),\beta(0),\dots)} \in [S]$ then $\cn{(n)}{\alpha} \in [T]$.

Suppose now that $P$ is the companion set of $T$. It is easy to verify that
\begin{align*}
P(n) \iff& (\exists! \alpha,\beta)[\cn{(n)}{(\alpha(0),\beta(0),\dots,\alpha(t),\beta(t),\dots)} \in [S]]\\
     \iff& (\exists \alpha,\beta \in \del)[\cn{(n)}{(\alpha(0),\beta(0),\dots,\alpha(t),\beta(t),\dots)} \in [S]]
\end{align*}
for all \n. So $S$ is a Spector-Gandy tree with companion set $P$. Finally we define
\[
f: \ca{N} \to \ca{N} : f(\gamma) = \cn{\gamma(0)}{\alpha}, \quad \textrm{where} \ \gamma = \cn{\gamma(0)}{(\alpha(0),\beta(0),\dots)}.
\]
Clearly $f$ is recursive. For all $\gamma = \cn{\gamma(0)}{(\alpha(0),\beta(0),\dots)} \in [S]$ we have $Q(\gamma(0),\alpha)$ and so $f(\gamma) = \cn{\gamma(0)}{\alpha} \in [T]$. Conversely if $\cn{(n)}{\alpha} \in [T]$ then there exists (a unique) $\beta$ such that $\gamma = \cn{(n)}{(\alpha(0),\beta(0),\dots)}$ is a member of $[S]$. Hence $f(\gamma) = \cn{(n)}{\alpha} \in [T]$. Thus we have shown that $f[[S]] =[T]$.

To finish the proof suppose that $\gamma, \gamma'$ are in $[S]$ and $f(\gamma) = f(\gamma')$. In particular we have that $\gamma(0) = \gamma'(0)$. Suppose that $\alpha, \alpha', \beta, \beta'$ are the corresponding components. Since  $\cn{(n)}{\alpha} = \cn{(n)}{\alpha'} \in [T]$ from the uniqueness of the $\beta$'s we have $\beta = \beta'$. It follows that $\gamma = \gamma'$ and $f$ is injective on $[S]$.
\end{proof}

The preceding lemma is also true in Kleene trees and one can see this easily. If $T$ is a \del \ tree, whose body is non-empty and has no \del \ members, then the set $[T]$ is \del, and so it is the recursive injective image of the body of a recursive tree $S$. It is then clear that $S$ is a Kleene tree.

\begin{proof}[\textit{Proof of Lemma \ref{lemma decreasing sequence of Spector-Gandy spaces} \tu{(}continued\tu{)}}.]
We consider the general case where $\preceq$ is \del. By repetition of the arguments in the first part of the proof we obtain a \del \ Spector-Gandy tree $S \subseteq T$ and $\gamma \in [T]$ such that $\alpha \not \heq \gamma$ for all $\alpha \in [S]$. We apply Lemma \ref{lemma del Spector-Gandy tree} in order to obtain a Spector-Gandy tree $L$ and a recursive function $f: \ca{N} \to \ca{N}$, which is injective on $[L]$ and $f[[L]] = [S]$. Therefore we have that $\alpha \heq f(\alpha) \not \heq \gamma$ for all $\alpha \in [L]$.
\end{proof}

\begin{theorem}
\label{theorem decreasing sequence of Spector-Gandy spaces}
There exists a sequence of Spector-Gandy spaces, which is strictly decreasing under \dleq.\smallskip

In fact for every pair $(P,\preceq)$ as in Lemma \ref{lemma decreasing sequence of Spector-Gandy spaces} and every Spector-Gandy space \spat{T} with companion set $P$, there exists a strictly decreasing sequence of Spector-Gandy spaces starting with \spat{T}, whose companion set is $P$.
\end{theorem}

\begin{proof}
First notice that there exists a pair $(P,\preceq)$ as in Lemma \ref{lemma decreasing sequence of Spector-Gandy spaces}, for example one takes $P$ to be the well-founded part of a recursive pseudo-well-ordering $\preceq$.

Now assume that the pair $(P,\preceq)$ is as in Lemma \ref{lemma decreasing sequence of Spector-Gandy spaces} and that $\spat{T}$ is a Spector-Gandy space with companion set $P$. Let $\gamma$ and $S$ be as in the conclusion of the latter lemma. Since $[T]$ contains a recursive injective image of $S$ it is clear that $\spat{S} \dleq \spat{T}$. Assume towards a contradiction that there exists a \del-injection $\pi: \spat{T} \inj \spat{S}$. Since $\gamma \heq \pi(\gamma)$ and $\gamma$ is not in \del, we have that $\pi(\gamma) \not \in S$. Hence $\pi(\gamma) \in [S]$, contradicting the key property of $\gamma$ and $S$. (Alternatively one can use Lemma \ref{lemma [T] is embedded into [S] in the Spector-Gandy case}.)

This shows that $\spat{S} \dless \spat{T}$. Since $S$ has the same companion set $P$ we may repeat the preceding argument and construct a strictly decreasing sequence under \dleq \ of Spector-Gandy trees starting with $\spat{T}$.
\end{proof}

It is worth noting that whenever we proved $\ca{X} \dless \ca{Y}$ in this article, the $\dleq$ part is witnessed by a \emph{recursive} injection. In fact this recursive function is the identity with the only exception of Theorem \ref{theorem decreasing sequence of Spector-Gandy spaces} when the preorder $\preceq$ is not recursive. Therefore our inequality results are the strongest possible in terms of injective functions and recursion theory.

\section{\sc Questions and related results}
\label{section questions open problems}

We conclude this article with a list of questions together with some related results and comments.

\subsection{Parametrization} Perhaps the single most important topic that we hardly discussed is the one of parametrization of a pointclass by a recursively presented metric space.

As we mentioned in the Introduction for all recursively presented metric spaces \ca{X} there exists a set $G \subseteq \ca{N} \times \ca{X}$ in $\Gamma$ which is universal for $\tboldsymbol{\Gamma} \upharpoonright \ca{X}$, where $\Gamma$ is a Kleene pointclass (Theorem \ref{theorem alltogether}-(\ref{theorem parametrization over N})). The same is true if we replace \ca{N} by any perfect recursively presented metric space \ca{Y}, \cf 3E.9 \cite{yiannis_dst}. The proof however does not go through if \ca{Y} is uncountable non perfect. Since every uncountable recursively presented metric space is $\Delta^1_2$ isomorphic to the Baire space the problem of parametrization from the level of \del \ sets and above is reduced to the following.

\begin{question}
\label{question parametrization}
Suppose that $T$ is a recursive tree with uncountable body and \ca{X} is a recursively presented metric space. Is it true that there exists a set $G \subseteq \spat{T} \times \ca{X}$ in $\Sigma^1_1$, which is universal for $\tboldsymbol{\Sigma}^1_1 \upharpoonright \ca{X}$?
\end{question}

The existence of universal sets is important because it provides a way of distinguishing pointclasses. For example the existence of a set $G \subseteq \ca{X} \times \ca{X}$ in \sig \ which is universal for $\tboldsymbol{\Sigma}^1_1 \upharpoonright \ca{X}$ implies the existence of a \sig \ subset of $\ca{X}$ which is not co-analytic, \cf 3E.9 \cite{yiannis_dst}. In Theorem \ref{theorem sig not in pii subset of bodyT} we have shown that for all recursively presented spaces \ca{X} there exists a set $A \subseteq \ca{X}$ which is in $\sig \setminus \pii$. Our examples however are countable sets and thus Borel.

\begin{question}
\label{question existence of sigma and not coanalutic}
Suppose that \ca{X} is an uncountable recursively presented metric space. Does there exist a set $P \subseteq \ca{X}$ in \sig, which not coanalytic?
\end{question}

\subsection{Kleene spaces} H. Friedman \cf \cite{friedman_harvey_borel_sets_and_hyperdegrees} has proved that for every $\gamma \in \ca{N}$ with $\W \hleq \gamma$ and for every Kleene tree $T$ there exists some $\delta \in [T]$ such that $\gamma \heq \delta$, \ie every hyperdegree from $\W$ and above appears in every Kleene tree, and he conjectured that the latter is true for all recursive trees with uncountable body. This conjecture was eventually proved independently by D. Martin and H. Friedman, \cf \cite{martin_proof_of_a_conjecture_of_friedman}. It would be interesting to see if one can prove this result by reducing the general case of a recursive tree with an uncountable body to the case of Kleene trees.

\begin{question}
\label{question NT contains Kleene}
Is it true that the body of a recursive tree is either countable or contains the \del \ injective image of a Kleene tree?
\end{question}
The preceding question is of particular interest as an affirmative answer to this would provide affirmative answers to other questions that we pose in the sequel. For now we just mention that an affirmative answer to Question \ref{question NT contains Kleene} implies that every uncountable recursively presented metric space contains the \del-injective image of a Kleene space.

From Corollary \ref{corollary leftmost branch is del or has hyperdegree W refined} we can restrict Question \ref{question NT contains Kleene} to recursive trees in which the class \del \ is dense. It is very tempting to conjecture that if \del \ is dense in $[T]$ then \spat{T} is either countable or it is \del-isomorphic to \ca{N}, so that (in the light of the preceding comments) Question \ref{question NT contains Kleene} would have an affirmative answer. However the latter is not correct. To see this we take a Kleene tree $T$ and we consider the image of the space \spat{T} under $\rho^T[\spat{T}]$, where $\rho^T$ is as in Lemma \ref{lemma the space Nt is embedded in the Baire space}. We take the recursive tree $S$ of the latter lemma so that $[S] = \rho^T[\spat{T}]$. Finally we consider the space \spat{S}. Since \spat{T} has a recursive presentation and $\rho^T$ is a recursive isometry, we have in particular that \del \ is dense in $[S]$. Moreover for all $x \in \spat{S}$ we have that
\[
x \in \del \iff x \in S \ \vee \ [x \in [S] \ \& \ x \ \textrm{is ultimately} \ 0],
\]
since $[T^{+1}]$ does not contain \del \ members. The preceding equivalence shows that $\del \cap \spat{S}$ is a \del \ set and so from Theorem \ref{theorem characterization the space is bad} the space \spat{S} is a Kleene tree.\footnote{In fact it is easy to check that $\spat{S^T} \dequal \spat{T}$ for every tree $T$, where $S^T$ is as in Lemma \ref{lemma the space Nt is embedded in the Baire space}.} Therefore \spat{S} is not \del-isomorphic to the Baire space and is uncountable.

Using the fact that every \del \ set is the recursive injective image of a $\Pi^0_1$ set one may reformulate Question \ref{question NT contains Kleene} as follows.

\addtocounter{question}{-1}
\begin{question}[\textbf{reformulated}]
Does every uncountable \del \ subset of a recursively presented metric space contain a non-empty \del \ subset with no \del \
members?
\end{question}
Let us call a recursively presented metric space \ca{X} \emph{minimal} if \ca{X} is uncountable and for every uncountable \ca{Y} with $\ca{Y} \dleq \ca{X}$ we have that $\ca{X} \dequal \ca{Y}$. Theorem \ref{theorem for every kleene space there is another one from below} implies that no Kleene space is minimal.

\begin{question}
\label{question existence of minimal space}
Does there exist a recursive tree $T$ such that \spat{T} is minimal?
\end{question}
Of course at most one of Questions \ref{question NT contains Kleene} and \ref{question existence of minimal space} has an affirmative answer.

We have seen that there are \dleq-incomparable Kleene spaces and that every Kleene space $\dleq$-strictly contains a Kleene space. One may ask if these two facts can be put together in a splitting-type theorem.

\begin{question}[Compare with Theorem 1 in \cite{binns_stephen_a_splitting_theorem_Medvedev_Muchnik_lattices}]
\label{question Kleene space contains incomparable subspaces}
Is it true that every Kleene space contains strictly with respect to $\dleq$ two subspaces which are $\dleq$-incomparable?\smallskip

Moreover, is it true that every Kleene space is strictly contained with respect to \dleq \ in two \dleq-incomparable Kleene spaces?
\end{question}

We recall the incomparability condition of Theorem \ref{theorem Fokina-Friedman-Toernquist}: for recursive trees $T$, $S$ let us write $T \dleq^{\rm s} S$ for the condition
\[
(\forall \alpha \in [S])(\exists \beta \in [T])[\beta \in \del(\alpha)],
\]
so that Theorem \ref{theorem Fokina-Friedman-Toernquist} reads that there exists a sequence $(T_i)_{\iin}$ of recursive trees, which are pairwise incomparable under $\dleq^{\rm s}$. (This is the \del \ version of Muchnik reducibility that we mentioned in the Introduction.)

It is clear that incomparability under  $\dleq^{\rm s}$ implies incomparability under \dleq \ of the corresponding spaces.

\begin{question}
\label{question stronger incomparability}
Can our incomparability results about Kleene spaces be strengthened to incomparability results under $\dleq^{\rm s}$?
\end{question}

Granted the rich structure of the classes of Kleene spaces one might ask if they somehow resemble the structure of \emph{recursively enumerable} Turing degrees.

\begin{question}
\label{question Kleene spaces and recursively enumerable degrees}
Are the classes of \del-isomorphism of Kleene spaces dense under \dleq? \tu{(}\cf \cite{sacks_the_recursively_enumerable_degrees_are_dense}\tu{)}\smallskip

Is $(\mathbb{Q},\leq_{\mathbb{Q}})$ embedded into the classes of \del-isomorphism of Kleene spaces? \tu{(}\cf \cite{kleene_post_the_upper_semilattice_of_degrees_of_recursive_unsolvability}\tu{)}
\end{question}

In the proof of Theorem \ref{theorem Fokina-Friedman-Toernquist} one uses a result in recursion theory (Harrington's Theorem \ref{theorem Harrington xi-incomparable}) about the levels of Turing jumps up to \ck \ and with a proper use of a compactness argument one derives a result about hyperdegrees in Kleene trees. It would be interesting to find other similar applications. This is closer to a question scheme rather than a simple question, but we nevertheless put it down for completeness.

\begin{question}
\label{question scheme}
Can we extend results about low-level Turing jumps to all levels $\xi < \ck$ and with a proper use of Kreisel compactness give results about hyperdegrees and in particular in connection with $\Pi^0_1$ subsets of the Baire space?
\end{question}

\subsection{Incomparable and minimal hyperdegrees} As we have seen in Corollary \ref{corollary incomparable hyperdegrees} incomparable hyperdegrees do occur in every recursively presented metric space which contains a \del-isomorphic copy of a Kleene space. It would be interesting to see if the latter is true for all uncountable spaces.

\begin{question}
\label{question incomparable hyperdegrees everywhere}
Given an uncountable recursively presented metric space \ca{X}, do they exist $x,y \in \ca{X}$ which are hyperarithmetically incomparable?
\end{question}

Therefore if Question \ref{question NT contains Kleene} is answered affirmatively then the preceding question has an affirmative answer as well. It would also be interesting to see if a perfect subset of pairwise hyperarithmetically incomparable members can exist in Kleene spaces.

\begin{question}
\label{question hyperdegrees perfect set}
Given a Kleene tree $T$ does there exist a non-empty perfect $P \subseteq [T]$ such that all $x \neq y$ in $P$ are hyperarithmetically incomparable?
\end{question}

A point $x$ in a recursively presented metric space $\ca{X}$ has \emph{minimal hyperdegree} if $x \not \in \del$ and for all $y \hleq x$ we have either $y \in \del$ or $x \hleq y$. Using ideas of Spector's construction of a minimal Turing degree \cite{spector_on_degrees_of_recursive_unsolvability}, and \emph{Sacks forcing} (or \emph{perfect tree forcing}), Gandy and Sacks \cf \cite{gandy_sacks_a_minimal_hyperdegree} proved the existence of a minimal hyperdegree, (see also \cite{sacks_higher_recursion_theory} Chapter IV Sections 4 and 5). One could ask if minimal hyperdegrees exist in every uncountable recursively presented metric space. The latter however is not correct even for Kleene spaces.

\begin{lemma}
\label{lemma minimal hyperdegree does not always occur}
There exists a Kleene space which contains no minimal hyperdegree.
\end{lemma}

\begin{proof}
To begin with we need a non-empty \sig \ set $Q \subseteq \ca{N}$ whose members are not \del \ points and have non-minimal hyperdegree. A set $Q$, which satisfies these properties (as well as others) is constructed in \cite{harrison_recursive_pseudo_wellorderings}, but we can also give a proof. Fix a Kleene tree $K$ and take the set
\[
Q = \set{\gamma \in \ca{N}}{(\gamma)_0,(\gamma)_1 \in [K]^2 \ \& \ (\gamma)_0 \not \heq (\gamma)_1}.
\]
Clearly $Q$ is a \sig \ set. It is also easy to verify that $Q$ is non-empty: one way to see this is by applying Theorem \ref{theorem incomparable hyperdegrees in Kleene spaces}, another way is to apply the Gandy Basis Theorem and use the fact that the leftmost infinite branch of a Kleene tree has the same hyperdegree as the one of \W. For all $\gamma \in Q$ we have that  $(\gamma)_0, (\gamma)_1 \not \in \del$, since $K$ is a Kleene tree. It follows that no $\gamma \in Q$ is \del. Moreover no $\gamma \in Q$ has minimal hyperdegree, for otherwise it would follow that
\[
(\gamma)_0 \heq \gamma \heq (\gamma)_1
\]
contradicting that $\gamma \in Q$.

Now we choose a non-empty \sig \ set $Q$ as above and we put it down in normal form
\[
Q(\alpha) \iff (\exists \beta)F(\alpha,\beta)
\]
where $F$ is $\Pi^0_1$. Clearly $F$ is non-empty and does not contain \del \ members. Moreover if $(\alpha,\beta) \in F$ had minimal hyperdegree then, since $\alpha$ is recursive in $(\alpha,\beta)$ and $\alpha \not \in \del$, it would follow that $(\alpha,\beta) \hleq \alpha$. So we would have $(\alpha,\beta) \heq \alpha$ and $\alpha$ would have minimal hyperdegree, contradicting the choice of $Q$.

Finally $F$ is obviously recursively isomorphic to the body of a Kleene tree $T$, hence $[T]$ has no members of minimal hyperdegree.
\end{proof}

\begin{question}
\label{question minimal hyperdegree in a Kleene space}
Does there exist a Kleene space which contains a minimal hyperdegree?
\end{question}
An affirmative answer to the preceding question would probably provide a new method for proving the existence of minimal hyperdegrees, and this would perhaps be more interesting than the answer itself.

\subsection{Connections with lattice theory} Unlike various notions of reducibility it is not clear that the partial order \dleq \ admits a join, \ie that every two elements have a least upper bound. This brings us to the next question.

\begin{question}
\label{question upper semilattice}
Is it true that for all recursively presented metric spaces \ca{X} and \ca{Y} there exists a recursively presented metric space \ca{Z} which is the least upper bound with respect to \dleq \ of the set $\{\ca{X},\ca{Y}\}$?\smallskip

Equivalently is the poset $(\set{\spat{T}}{T \ \textrm{is a recursive tree}}, \dleq)$ an upper semi-lattice?
\end{question}

A good candidate for a join seems to be the operation of the sum, but still this claim is not clear even if $\ca{X} = \ca{Y}$.

\addtocounter{question}{-1}
\begin{question}[$\mathbf{\ast}$]
\label{question X sum X dequals X}
Is it true that $\spat{T} \oplus \spat{T} \dequal \spat{T}$ for all recursive trees $T$?
\end{question}

Evidently $\ca{X} \times \ca{X} \dequal \ca{X} \dequal \ca{N}$ for every perfect recursively presented metric space \ca{X}. One may ask if the same is true in the general case. This would provide us a way of encoding pairs of \ca{X} by one element of \ca{X}.

\addtocounter{question}{-1}
\begin{question}[$\mathbf{\dagger}$]
\label{question X times X dequals X}
Is it true that $\spat{T} \times \spat{T} \dequal \spat{T}$ for all recursive trees $T$?
\end{question}

Of course an affirmative answer to the preceding question would imply that $\ca{X} \dequal \ca{X} \oplus \ca{X} \dequal \ca{X} \times \ca{X}$ for all recursively presented metric spaces \ca{X}.

\subsection{Attempted Embeddings and incomparability}

We have seen that for every Kleene space there exists a \dleq-incomparable Kleene space. The question is whether this result can be extended to other spaces as well.

\begin{question}
\label{question for every space there is an incomparable space}
Given a recursive tree $T$ for which the space \spat{T} is uncountable and not \del-isomorphic to the Baire space does there exist a recursive tree $S$ such that the spaces \spat{T} and \spat{S} are $\dleq$-incomparable?
\end{question}
As we have seen in Theorem \ref{theorem about tree of attempts} one can give a partial answer to the preceding question under some conditions which hold in Kleene spaces.

\begin{question}
\label{question conditions of theorem about tree of attempts}
Given a recursive tree $T$ with uncountable body does there exist an uncountable \del \ set $C$ and some $\gamma \in [T] \setminus \del$ such that the set $\del(\gamma) \cap C$ is $\del(\gamma)$?
\end{question}
This question underlines the importance of having an affirmative answer to Question \ref{question NT contains Kleene}. Let us explain this in detail. Suppose that $T$ is a recursive tree with uncountable body and that $S$ answers Question \ref{question NT contains Kleene} for this $T$, \ie
$S$ is a Kleene tree and $\spat{S} \dleq \spat{T}$. From Corollary \ref{corollary del of gamma is dense in a Kleene tree} there exist $u \in S$ and $\gamma^\ast \in [S]$ such that $[S_u] \neq \emptyset$ and $\del(\gamma^\ast) \cap [S_u] = \emptyset$. Consider a \del-injection $f$ from \spat{S} into \spat{T} and take $C=[S_u]$ and $\gamma = f(\gamma^\ast) \in [T]$ (Lemma \ref{lemma converse of lemma bodyT less than bodyS in Kleene trees}). Then $\del(\gamma) \cap C = \emptyset$, $C$ is an uncountable $\Pi^0_1$ set and $\gamma \heq \gamma^\ast \not \in \del$. Hence the pair $(\gamma,C)$ witnesses that Question \ref{question conditions of theorem about tree of attempts} has an affirmative answer for $T$.

Now assume that Question \ref{question conditions of theorem about tree of attempts} has an affirmative answer for $T$ and let $(\gamma,C)$ be a witnessing pair. Since $C$ is a \del \ set from Theorem \ref{theorem alltogether}-(\ref{theorem 4D.9}) and Lemma \ref{lemma pi-zero-one sets recursive tree} there exists a recursive tree $S$ and a recursive function $f: \ca{N} \to \ca{N}$ which is injective on $[S]$ and $f[[S]] = C$. From Lemma \ref{lemma the inverse function is del and same hyperdegree} we have that
\[
\alpha \in \del(\gamma) \cap[S] \iff \alpha \in [S] \ \& \ f(\alpha) \in \del(\gamma) \cap C
\]
for all $\alpha \in \ca{N}$, and since $\del(\gamma) \cap C$ is a $\del(\gamma)$ set it follows that set $\del(\gamma) \cap [S]$ is $\del(\gamma)$. Hence the hypothesis of Theorem \ref{theorem about tree of attempts} is satisfied and thus we have that $\spat{T} \not \dleq \spat{\att{S}}$.

Putting the remarks of the preceding two paragraphs together we conclude to the following.
\begin{align*}
&\textrm{Question \ref{question NT contains Kleene} has an affirmative answer for $T$}\\
&\hspace{10mm}\Longrightarrow \ \textrm{Question \ref{question conditions of theorem about tree of attempts} has an affirmative answer for $T$}\\
&\hspace{15mm}\Longrightarrow \ \textrm{there exists a recursive tree $S$ such that $\spat{T} \not \dleq \spat{\att{S}}$}.
\end{align*}

This would give a partial answer to Question \ref{question for every space there is an incomparable space} for $T$. Now we deal with the other side of the incomparability of this question. A clue for the answer may lie in the structure of the set
\[
A_T = \set{x \in [T]}{(\exists \gamma \in \del(x)) [\gamma \in [\att{T}]]}.
\]
Clearly the set $A_T$ is \pii. What is less obvious is that $A_T$ is non-empty: from the result of Martin \cite{martin_proof_of_a_conjecture_of_friedman} there exists some $\alpha \in [T]$ which has the same hyperdegree as \W \  \ and so from the Kleene Basis Theorem  there exists some $\gamma \in [\att{T}]$ which is recursive in $\W$ -and hence $\del(\alpha)$. (In fact if $T$ is a Kleene tree one can take $\alpha$ to be the leftmost infinite branch of $T$ which from Lemma \ref{lemma leftmost branch is del or has hyperdegree W} has the same hyperdegree as \W.)

So the question is whether we are able to say anything else about $A_T$. Does it contain for example a non-empty perfect set? Or equivalently does it contain some $\gamma$ not in $L(\ckr{\gamma})$ (\cf \cite{sacks_higher_recursion_theory} Theorem 9.5 Ch. III)?

Let us assume for the sake of the discussion that the preceding set $A_T$ (with $[T]$ being uncountable) does not have a non-empty perfect subset. Suppose that there exists a recursive tree $S$ with uncountable body and a \del-recursive function $f: \ca{N} \to \ca{N}$ which carries $[S]$ in an injective way into $[T]$. We claim that there does not exist a \del-injection $\sigma: \spat{K} \inj \spat{T}$ for any $K \subseteq \att{S}$ with uncountable body. To see this assume otherwise, and notice that the set $\sigma[[K]] \cap [T]$ is uncountable \del. If $x \in [T]$ satisfies $x = \sigma(\gamma)$ for some $\gamma \in [K]$ we have in particular that $\gamma \in \del(x)$. Since $K \subseteq \att{S}$ the point $\gamma \in [K]$ gives rise to a $\del(\gamma)$-injection $\tau: \ca{C} \inj [S]$ and the function $\pi = f \circ \tau$ is a $\del(\gamma)$-injection from \ca{C} into $[T]$. Therefore from the relativized version of Theorem \ref{theorem characterization of isomorphic to the Baire space} there exists a $\del(\gamma)$-embedding of the complete binary tree into $T$, \ie there exists some $\gamma^\ast \in [\att{T}]$ such that $\gamma^\ast \in \del(\gamma) \subseteq \del(x)$. This shows that $x$ is a member of $A_T$ and so $\sigma[[K]] \cap [T]$ is a subset of $A_T$. From the Perfect Set Theorem $A_T$ has a non-empty perfect subset, contradicting our hypothesis.

\emph{Therefore if $T$ is a recursive tree with uncountable body, $A_T$ has no non-empty perfect subset and $S$ is a recursive tree for which there exists a \del-recursive function which carries $[S]$ in an injective way into $[T]$, then we have that $\spat{K} \not \dleq \spat{T}$ for all recursive trees $K \subseteq \att{S}$ with uncountable body.}

Putting all these remarks together we conclude to the following.

\begin{lemma}
\label{lemma conditions to produce incomparable spaces}
Suppose that $T$ is a recursive tree with uncountable body which satisfies the following.\smallskip

\tu{(1)} The set $\set{x \in [T]}{(\exists \gamma \in \del(x)) [\gamma \in [\att{T}]]}$ has no non-empty perfect subset.\smallskip

\tu{(2)} There exists a recursive tree $S$ with uncountable body, a \del-recursive function $f: \ca{N} \to \ca{N}$ and $\gamma \in [T]$ such that: \tu{(a)} the function $f$ is injective on $[S]$, \tu{(b)} $f[[S]] \subseteq [T]$, and \tu{(c)} the set $\del(\gamma) \cap [S]$ is $\del(\gamma)$.\smallskip

Then we have that
\[
\spat{T} \not \dleq \spat{\att{S}} \quad \textrm{and} \quad \spat{\att{S}} \not \dleq \spat{T}.
\]
\end{lemma}
Notice that the space \spat{T} is not \del-isomorphic to the Baire space for every $T$ which satisfies the hypothesis of the preceding lemma.

\begin{question}
\label{question about the structure of AT}
Suppose that $T$ is a recursive tree with uncountable body. What properties does the non-empty  \pii \ set
\[
A_T=\set{x \in [T]}{(\exists \gamma \in \del(x)) [\gamma \in [\att{T}]]}
\]
have?
\end{question}
One may also consider the set
\[
B_T = \set{x \in [T]}{(\exists \gamma \ \textrm{arithmetical in} \ x) [\gamma \in [\att{T}]]}.
\]
Clearly $B_T$ is a \del \ subset of $A_T$, but it might be empty. However if $T$ is a Kleene tree and one manages to prove that \W \ is in fact arithmetical in the leftmost infinite branch of $T$ then by applying the  Kleene Basis Theorem we will have that $B_T$ is non-empty (with no \del \ elements). From the Effective Perfect Set Theorem it would follow that $B_T$ and subsequently $A_T$ has a perfect subset. This leads us to the following.
\addtocounter{question}{-1}
\begin{question}[$\mathbf{\ast}$]
Does there exist a Kleene tree $T$ such that \W \ is arithmetical in $$\set{t \in \om}{[T_{\dec{t}}] = \emptyset}?$$
\end{question}
\addtocounter{question}{-1}
\begin{question}[$\mathbf{\dagger}$]
What else can be said about the relation between \spat{T} and \spat{\att{T}}?
\end{question}

\subsection{Spector-Gandy spaces} Besides some analogous results with Kleene spaces some key aspects about Spector-Gandy spaces remain open for investigation.

\begin{question}
\label{question product of Spetor-Gandy trees}
Is the product of two Spector-Gandy spaces \del-isomorphic to a Spector-Gandy space?
\end{question}

If the answer to the preceding question is negative then from Theorem \ref{theorem product of Spector-Gandy spaces is neither Kleene nor Baire} there would be Spector-Gandy trees $T, S$ such that the space $\spat{T} \times \spat{S}$ does not fall into the categories of Kleene spaces, Spector-Gandy spaces, \om \ and \ca{N}. In particular $\spat{T} \times \spat{S}$ would answer Question \ref{question spaces other than Kleene and SG} below.

It would also be interesting to see if Question \ref{question product of Spetor-Gandy trees} can be refuted by a single Spector-Gandy tree, \ie if there exists a Spector-Gandy tree $T$ such that the space $\spat{T} \times \spat{T}$ is not \del-isomorphic to a Spector-Gandy space. In this case we would have in particular that $\spat{T} \times \spat{T} \not \dequal \spat{T}$ and Question \ref{question X times X dequals X} $(\dagger)$ would also be answered.

We have seen that strictly decreasing and strictly increasing under \dleq \ sequences of Spector-Gandy spaces exist. One can ask if \dleq \ restricted on these spaces has no maximal or minimal elements, as it is the case with Kleene spaces.

\begin{question}
\label{question chains of SG spaces}
Is every Spector-Gandy space the top \tu{(}bottom\tu{)} of a strictly decreasing \tu{(}increasing\tu{)} under \dleq \ sequence of Spector-Gandy spaces?
\end{question}

It is easy to verify that the sum of a Spector-Gandy space and a Kleene space is a Spector-Gandy space. A clue for finding a \dleq \ strictly increasing sequence of Spector-Gandy spaces, which starts with a given Spector-Gandy space, might lie in the proof of Theorem \ref{theorem strictly increasing Kleene spaces}.

\addtocounter{question}{-1}
\begin{question}[$\mathbf{\ast}$]
If $T$ is a Spector-Gandy tree does there exist a Kleene tree $K$ such that $\spat{K} \not \dleq \spat{T}$?
\end{question}

Spector-Gandy spaces unlike Kleene spaces are clearly not closed from below with respect to \dleq, since every Spector-Gandy space  contains a Kleene space. Moreover there is little evidence to suggest that Spector-Gandy spaces are closed from above.

\begin{question}
\label{question SG closed from above}
Is it possible to have $\spat{T} \dless \ca{X} \dless \ca{N}$ with $\ca{X} \not \dequal \spat{S}$ for every Spector-Gandy tree $S$?
\end{question}
Of course no \ca{X} as in the preceding question is a Kleene space (\cf Corollary \ref{corollary Kleene spaces are closed from below}). This raises the following.
\begin{question}
\label{question spaces other than Kleene and SG}
Does there exist a recursively presented metric space which is not \del-isomorphic to any Kleene space, Spector-Gandy space or the Baire space?
\end{question}
One can state the analogue of Question \ref{question Kleene space contains incomparable subspaces}.

\begin{question}
\label{question splitting in SG spaces}
Is it true that every Spector-Gandy space contains strictly under $\dleq$ two Spector-Gandy subspaces which are $\dleq$-incomparable?\smallskip

Moreover, is it true that every Spector-Gandy space is contained strictly with respect to $\dleq$ into two Spector-Gandy spaces which are $\dleq$-incomparable?
\end{question}

\subsection{Connections with \pii \ equivalence relations} The relation $\dequal$ induces in a natural way an equivalence relation on the set of recursive trees.
\begin{definition}\normalfont
\label{definition pii equivalence relation}
Define
\[
\RTr = \set{T \in \Tr}{T \ \textrm{is a recursive tree}}
\]
and the relation $\equiv$ by
\[
T \equiv S \iff \spat{T} \dequal \spat{S}
\]
for $T, S \in \RTr$.
\end{definition}

\begin{lemma}
\label{lemma equivalence relation is pii}
The preceding equivalence relation $\equiv$ is the intersection of a \pii \ subset of $\Tr^2$ \ with $\RTr^2$.
\end{lemma}

\begin{proof}
It is enough to show that the relation
\[
R(T,S) \iff T,S \ \textrm{are recursive trees} \ \longrightarrow \ \spat{T} \dleq \spat{S},
\]
is \pii. For the needs of this proof we view the spaces $\spat{T}$ as $\Pi^0_1$ subsets of \ca{N}. We will encode all \del-recursive functions $f: \ca{N} \to \ca{N}$ in a \pii-way. The method is along the lines of the coding of $\Gamma$-recursive functions we gave in Section \ref{section Kleene spaces}.

Consider a set $G \subseteq \om \times \ca{N}^2$ which is universal for \pii \ and uniformize it with respect to the last variable by a \pii \ set $G^\ast$. We define $C \subseteq \om$ by
\begin{align*}
e \in C \iff& (\forall \alpha)(\exists \beta \in \del(\alpha))G^\ast(e,\alpha,\beta).
\end{align*}
Clearly $C$ is a \pii \ set and moreover for all functions $f: \ca{N} \to \ca{N}$ we have that
\[
f \ \textrm{is \del-recursive} \ \iff (\exists e)(\forall \alpha)[e \in C \ \& \ G^\ast(e,\alpha,f(\alpha))].
\]
In other words $C$ serves as set of codes of all \del \ functions on \ca{N} to \ca{N}. We also define the set $H \subseteq \om \times \ca{N}^2$ by
\[
H(e,\alpha,\beta) \iff (\forall \beta^\ast \in \del(\alpha))[G^\ast(e,\alpha,\beta^\ast) \ \longrightarrow \ \beta^\ast = \beta]
\]
so that $H$ is \sig \ and for all $e \in C, \alpha, \beta \in \ca{N}$ we have that
\[
G^\ast(e,\alpha,\beta) \iff H(e,\alpha,\beta).
\]
Using these remarks it is easy to see that
\begin{align*}
&\spat{T} \dleq \spat{S} \iff (\exists e \in C)\\
                            & \big \{(\forall \alpha)(\exists \beta \in \del(\alpha))[\alpha \in \spat{T} \ \longrightarrow \ G^\ast(e,\alpha,\beta) \ \& \ \beta \in \spat{S}] \\
                            &\ \ \& \ (\forall \alpha,\alpha',\beta)[[\alpha, \alpha'\in \spat{T} \ \& \ H(e,\alpha,\beta) \ \& \ H(e,\alpha',\beta)] \longrightarrow \ \alpha = \alpha'] \big \},
\end{align*}
from which follows that the preceding relation $R$ is \pii.
\end{proof}
We may view $\equiv$ as an equivalence relation on \om \ by identifying a recursive tree (in a \del-way) with a natural number, say its code through a universal $\Sigma^0_1$ set. This way the preceding lemma says that $\equiv$ is a \pii \ set.

We conjecture that $\equiv$ is not \sig. In fact we will ask for something stronger. An equivalence relation $\approx$ on \ca{X} is $\tboldsymbol{\Pi}^1_1$-\emph{hard} if for all $\tboldsymbol{\Pi}^1_1$ equivalence relations $\approx'$ on \ca{N} there exist a continuous function $f: \ca{N} \to \ca{X}$ such that
\[
\alpha \approx' \beta \iff f(\alpha) \approx f(\beta)
\]
for all $\alpha, \beta \in \ca{N}$, \ie $\approx'$ is \emph{continuously reduced} to $\approx$. G. Hjorth \cite{hjorth_universal_coanalytic_sets} has proved that there exists a \pii-equivalence relation on \ca{N} which is $\tboldsymbol{\Pi}^1_1$-hard.

In our case it does not make sense to ask if $\equiv$ is $\tboldsymbol{\Pi}^1_1$-hard simply because its domain is countable. But we may pose the countable version of this question.
\enlargethispage{5mm}
\begin{question}
\label{question equivalence relation pii complete}
Suppose that $\approx$ is a \pii \ equivalence relation on \om. Does there exist a \del-recursive function $f: \om \to \Tr$ such that
\[
n \approx m \iff f(n) \equiv f(m)
\]
for all $n,m$?
\end{question}

\ifx\undefined\bysame
\newcommand{\bysame}{\leavevmode\hbox to3em{\hrulefill}\,}
\fi

\end{document}